\documentclass[10pt]{amsart}

\usepackage[utf8x]{inputenc}
\usepackage[english]{babel}
\usepackage{amsmath,amsfonts,amssymb,amsthm}
\usepackage[T1]{fontenc}
\usepackage{lmodern}

\usepackage[top=3.5cm,bottom=3.5cm,left=3.6cm,right=3.6cm]{geometry}

\usepackage[dvipsnames]{xcolor}
\usepackage[hyperindex=true,frenchlinks=true,colorlinks=true,
citecolor=NavyBlue,linkcolor=Mulberry,urlcolor=LimeGreen,linktocpage]{hyperref}

\usepackage{tikz}

\usepackage{enumitem}
\usepackage{shuffle}
\usepackage{cite}
\usepackage{subfigure}
\usepackage{array}
\usepackage{multicol}
\usepackage{multirow}

\author{Samuele Giraudo}
\date{\today}
\address{Laboratoire d'Informatique Gaspard-Monge,
université Paris-Est Marne-la-Vallée, 5 Boulevard Descartes,
Champs-sur-Marne, 77454, Marne-la-Vallée cedex 2, France}
\email{samuele.giraudo@univ-mlv.fr}
\title{Combinatorial operads from monoids}

\renewcommand{\arraystretch}{1.1}

\newtheorem{Theoreme}{Theorem}[section]
\newtheorem{Proposition}[Theoreme]{Proposition}
\newtheorem{Lemme}[Theoreme]{Lemma}

\numberwithin{equation}{subsection}

\renewcommand{\leq}{\leqslant}
\renewcommand{\geq}{\geqslant}

\newcommand{\EnsPermu}{\mathfrak{S}}
\newcommand{\EnsNat}{\mathbb{N}}
\newcommand{\EnsNatM}{\mathbb{M}}
\newcommand{\EnsRel}{\mathbb{Z}}
\newcommand{\K}{\mathbb{K}}
\newcommand{\CalB}{\mathcal{B}}
\newcommand{\CalC}{\mathcal{C}}
\newcommand{\CalF}{\mathcal{F}}
\newcommand{\CalP}{\mathcal{P}}
\newcommand{\CalQ}{\mathcal{Q}}
\newcommand{\CalT}{\mathcal{T}}
\newcommand{\CalS}{\mathcal{S}}
\newcommand{\T}{{\sf T}}
\newcommand{\TT}{{\it T}}

\newcommand{\Unite}{{\bf 1}}
\newcommand{\Permu}{{\sf As}}
\newcommand{\End}{{\sf End}}
\newcommand{\FP}{{\sf PF}}
\newcommand{\MT}{{\sf PW}}
\newcommand{\Per}{{\sf Per}}
\newcommand{\APE}{{\sf PRT}}
\newcommand{\FCat}[1]{{\sf FCat^{(#1)}}}
\newcommand{\Schr}{{\sf Schr}}
\newcommand{\Motz}{{\sf Motz}}
\newcommand{\Comp}{{\sf Comp}}
\newcommand{\AnD}{{\sf DA}}
\newcommand{\SComp}{{\sf SComp}}
\newcommand{\DD}{{\sf Di}}
\newcommand{\Dias}{{\sf Dias}}
\newcommand{\Tr}{{\sf Tr}}
\newcommand{\Trias}{{\sf Trias}}
\newcommand{\Dendr}{{\sf Dendr}}
\newcommand{\TDendr}{{\sf TDendr}}
\newcommand{\Quad}{{\sf Quad}}
\newcommand{\Ennea}{{\sf Ennea}}
\newcommand{\NAP}{{\sf NAP}}
\newcommand{\Com}{{\sf Com}}
\newcommand{\La}{{\tt a}}
\newcommand{\Lb}{{\tt b}}
\newcommand{\Lc}{{\tt c}}
\newcommand{\Ld}{{\tt d}}
\newcommand{\Alphab}{\operatorname{Alph}}

\newcommand{\Deg}{\operatorname{d}}
\newcommand{\ArbCons}{\bigwedge}
\newcommand{\Eval}{\operatorname{ev}}
\newcommand{\Etiq}{\operatorname{lbl}}
\newcommand{\Poids}{\operatorname{w}}
\newcommand{\NbNoeuds}{\operatorname{n}}
\newcommand{\NbFeuilles}{\ell}

\newcommand{\Feuille}{\,\scalebox{.25}{
\begin{tikzpicture}
    \node[Feuille](0)at(0,0){};
    \node(1)at(0,1){};
    \draw[Arete](1)--(0);
\end{tikzpicture}}\,}

\newcommand{\SchrOpA}{\,\scalebox{.18}{
\begin{tikzpicture}
    \node[NoeudSr](Schr1)at(0,0){};
    \node[Feuille](Schr2)at(-1,-1){};
    \node[Feuille](Schr3)at(0,-1){};
    \node[Feuille](Schr4)at(1,-1){};
    \draw[Arete](Schr1)--(Schr2);
    \draw[Arete](Schr1)--(Schr3);
    \draw[Arete](Schr1)--(Schr4);
\end{tikzpicture}}\,}

\newcommand{\SchrOpB}{\,\scalebox{.18}{
\begin{tikzpicture}
    \node[NoeudSr](Schr1)at(0,0){};
    \node[Feuille](Schr2)at(-1,-1){};
    \node[NoeudSr](Schr3)at(.5,-.5){};
    \node[Feuille](Schr4)at(0,-1){};
    \node[Feuille](Schr5)at(1,-1){};
    \draw[Arete](Schr1)--(Schr2);
    \draw[Arete](Schr1)--(Schr3);
    \draw[Arete](Schr3)--(Schr4);
    \draw[Arete](Schr3)--(Schr5);
\end{tikzpicture}}\,}

\newcommand{\SchrOpC}{\,\scalebox{.18}{
\begin{tikzpicture}
    \node[NoeudSr](Schr1)at(0,0){};
    \node[NoeudSr](Schr2)at(-.5,-.5){};
    \node[Feuille](Schr3)at(-1,-1){};
    \node[Feuille](Schr4)at(0,-1){};
    \node[Feuille](Schr5)at(1,-1){};
    \draw[Arete](Schr1)--(Schr2);
    \draw[Arete](Schr2)--(Schr3);
    \draw[Arete](Schr2)--(Schr4);
    \draw[Arete](Schr1)--(Schr5);
\end{tikzpicture}}\,}

\newcommand{\MotzOpA}{\,\scalebox{.18}{
\begin{tikzpicture}
    \node[NoeudDyck](Motz1)at(0,0){};
    \node[NoeudDyck](Motz2)at(1,0){};
    \draw[PasDyck](Motz1)--(Motz2);
\end{tikzpicture}}\,}

\newcommand{\MotzOpB}{\,\scalebox{.18}{
\begin{tikzpicture}
    \node[NoeudDyck](Motz1)at(0,0){};
    \node[NoeudDyck](Motz2)at(1,1){};
    \node[NoeudDyck](Motz3)at(2,0){};
    \draw[PasDyck](Motz1)--(Motz2);
    \draw[PasDyck](Motz2)--(Motz3);
\end{tikzpicture}}\,}

\newcommand{\AnDOpA}{\,\scalebox{.18}{
\begin{tikzpicture}
    \node[NoeudDyck](AnD1)at(0,0){};
    \node[NoeudDyck](AnD2)at(1,0){};
    \draw[PasDyck](AnD1)--(AnD2);
\end{tikzpicture}}\,}

\newcommand{\AnDOpB}{\,\scalebox{.18}{
\begin{tikzpicture}
    \node[NoeudDyck](AnD1)at(0,0){};
    \node[NoeudDyck](AnD2)at(0,1){};
    \draw[PasDyck](AnD1)--(AnD2);
\end{tikzpicture}}\,}

\newcommand{\CompOpA}{\,\scalebox{.15}{
\begin{tikzpicture}
    \node[Boite]at(0,0){};
    \node[Boite]at(1,0){};
\end{tikzpicture}}\,}

\newcommand{\CompOpB}{\,\scalebox{.15}{
\begin{tikzpicture}
    \node[Boite,Marque2]at(0,0){};
    \node[Boite,Marque2]at(0,-1){};
\end{tikzpicture}}\,}

\newcommand{\SCompOpA}{\,\scalebox{.15}{
\begin{tikzpicture}
    \node[Boite,Marque3]at(0,0){};
    \node[Boite,Marque3]at(1.75,0){};
\end{tikzpicture}}\,}

\newcommand{\SCompOpB}{\CompOpA}

\newcommand{\SCompOpC}{\CompOpB}

\newcommand{\Sloane}[1]{\href{http://oeis.org/#1}{{\bf #1}}}

\definecolor{Noir}{RGB}{0,0,0}
\definecolor{Blanc}{RGB}{255,255,255}
\definecolor{Rouge}{RGB}{205,35,38}
\definecolor{Bleu}{RGB}{2,60,195}
\definecolor{Vert}{RGB}{23,163,1}
\definecolor{Violet}{RGB}{181,18,225}
\definecolor{Orange}{RGB}{255,113,15}

\tikzstyle{Noeud} = [circle,draw=Bleu!80,fill=Bleu!20,thick,inner sep=0pt,
minimum size=10mm,line width=4pt,font=\Huge]
\tikzstyle{NoeudSr} = [Noeud,minimum size=4mm]
\tikzstyle{Arete} = [Rouge!80,thick,draw,line width=4pt]
\tikzstyle{Feuille} = [rectangle,draw=Noir!70,fill=Noir!20,thick,
inner sep=0pt,minimum size=3mm,line width=2pt]
\tikzstyle{Marque1} = [draw=Vert!100,fill=Vert!30]
\tikzstyle{Marque2} = [draw=Orange!100,fill=Orange!40]
\tikzstyle{Marque3} = [draw=Rouge!100,fill=Rouge!50]
\tikzstyle{EtiqClair} = [fill=Blanc!100]
\tikzstyle{NoeudDyck} = [circle,draw=Violet!90,fill=Bleu!60,thick,
inner sep=0pt,minimum size=5mm]
\tikzstyle{PasDyck} = [color=Violet!60,line width=3pt]
\tikzstyle{Grille} = [color=Noir!30]
\tikzstyle{Boite} = [rectangle,rounded corners,draw=Bleu!100,fill=Bleu!20,
thick,inner sep=0pt,minimum size=10mm,line width=3pt,font=\Huge]
\tikzstyle{Injection} = [Noir!100,draw,>->]
\tikzstyle{Surjection} = [Noir!100,draw,->>]
\tikzstyle{SommetAnimal} = [circle,draw=Violet!90,fill=Bleu!60,thick,
inner sep=0pt,minimum size=5mm]

\begin{document}

\begin{abstract}
    We introduce a functorial construction which, from a monoid, produces
    a set-operad. We obtain new (symmetric or not) operads as suboperads
    or quotients of the operads obtained from usual monoids such as the
    additive and multiplicative monoids of integers and cyclic monoids.
    They involve various familiar combinatorial objects: endofunctions,
    parking functions, packed words, permutations, planar rooted trees,
    trees with a fixed arity, Schröder trees, Motzkin words, integer
    compositions, directed animals, and segmented integer compositions.
    We also recover some already known (symmetric or not) operads: the
    magmatic operad, the associative commutative operad, the diassociative
    operad, and the triassociative operad. We provide presentations by
    generators and relations of all constructed nonsymmetric operads.
\end{abstract}

\maketitle

\tableofcontents

\section*{Introduction} \label{sec:Intro}

Operads are algebraic structures introduced in the 1970s by Boardman and
Vogt~\cite{BV73} and by May~\cite{May72} in the context of algebraic topology.
Informally, an operad is a structure containing operators with~$n$ inputs
and~$1$ output, for all positive integer~$n$. Two operators~$x$ and~$y$
can be composed at the $i$th position by grafting the output of~$y$ on the $i$th
input of~$x$. The new operator thus obtained is denoted by~$x \circ_i y$.
In an operad, one can also switch the inputs of an operator~$x$
by letting a permutation~$\sigma$ act to obtain a new operator denoted
by~$x \cdot \sigma$. One of the main relishes of operads comes from the
fact that they offer a general theory to study in an unifying way different
types of algebras, such as associative algebras and Lie algebras.
\smallskip

In recent years, the importance of operads in combinatorics has increased
and several new (nonsymmetric) operads were defined on combinatorial objects
(see {\em e.g.},~\cite{Lod01,CL01,Liv06,Cha08}). The structure thereby added
on combinatorial families enables us to see these in a new light and offers
original ways to solve some combinatorial problems. For example, the dendriform
operad~\cite{Lod01} is a nonsymmetric operad on binary trees that plays
an interesting role in describing the Hopf algebra of Loday-Ronco of binary
trees~\cite{LR98,HNT05}. Besides, this nonsymmetric operad is a key ingredient
for the enumeration of intervals of the Tamari lattice~\cite{Cha06,Cha08}. There
is also a very rich link connecting combinatorial Hopf algebra theory and
operad theory: various constructions produce combinatorial Hopf algebras
from operads~\cite{Vdl04,CL07,Fra08,LV12,ML13}.
\smallskip

In this paper, we propose a new generic method to build combinatorial operads.
The starting point is to pick a monoid~$M$. We then consider the set of
words whose letters are elements of~$M$. The arity of such words are their
length, the composition of two words is expressed from the product of~$M$,
and permutations act on words by permuting letters. In this way, we associate
an operad denoted by~$\T M$ with any monoid~$M$. This construction is rich
from a combinatorial point of view since it allows us, by considering
suboperads and quotients of~$\T M$, to get new (symmetric or not) operads
on various combinatorial objects. Our construction is related to two
previous ones.
\smallskip

The first one is a construction of Méndez and Nava~\cite{MN93} emerging from
the context of the species theory~\cite{Joy81}. Roughly speaking, a species
is a combinatorial construction~$U$ which takes an underlying set~$E$ as
input and produces a set~$U[E]$ of objects by adding some structure on the
elements of~$E$ (see~\cite{BLL94}). This theory has many links with the
theory of operads since an operad is a monoid with respect to the operation
of substitution of species. In~\cite{MN93}, the authors defined the plethystic
species, that are species taking as input sets where any element has a colour
picked from a fixed monoid~$M$. This monoid has to satisfy some precise
conditions (as to be left cancellable and without proper divisor of the unity,
and such that any element has finitely many factorizations). It appears that
the elements of the so-called uniform plethystic species can be seen as
words of colours and hence, as elements of~$\T M$. Moreover, the composition
of this operad is the one of~$\T M$. The main difference between the
construction of Méndez and Nava and ours lies in the fact that the
construction~$\T$ can be applied on any monoid.
\smallskip

The second one, introduced by Berger and Moerdijk~\cite{BM03}, is a
construction which allows to obtain, from a commutative bialgebra~$\CalB$,
a cooperad~$\TT \CalB$. Our construction~$\T$ and the construction~$\TT$
of these two authors are different but coincide in many cases.
For instance, when~$(M, \bullet)$ is a monoid such that for any~$x \in M$,
the set of pairs~$(y, z) \in M^2$ satisfying~$y \bullet z = x$ is finite,
the operad~$\T M$ is the dual of the cooperad $\TT \CalB$ where~$\CalB$
is the dual bialgebra of~$\K [M]$ endowed with the diagonal coproduct ($\K$
is a field). On the other hand, there are operads that we can build by
the construction~$\T$ but not by the construction~$\TT$, and conversely.
For example, the operad~$\T \EnsRel$, where~$\EnsRel$ is the additive monoid
of integers, cannot be obtained as the dual of a cooperad built by the
construction~$\TT$ of Berger and Moerdijk.
\smallskip

Furthermore, our construction is defined in the category of sets and
computations are explicit. It is therefore possible given a monoid~$M$,
to make experiments on the operad~$\T M$, using if necessary a computer.
In this paper, we study many applications of the construction~$\T$
focusing on its combinatorial aspect. More precisely, we define, by
starting from very simple monoids like the additive or multiplicative
monoids of integers, or cyclic monoids, various nonsymmetric operads
involving well-known combinatorial objects.
\smallskip

This paper is organized as follows. In Section~\ref{sec:Operades}, we
set some notations about syntax trees and rewriting systems on syntax trees.
We then briefly recall the basics about operads. We also prove in this
section two important lemmas used in the sequel of the paper: the first
one deals with the form of elements of nonsymmetric operads generated by
a set of generators and the second one is a tool to prove presentations
by generators and relations of nonsymmetric operads using rewrite rules
on syntax trees. Section~\ref{sec:Foncteur} defines the construction~$\T$,
associating an operad with a monoid and establishes its first properties.
We show that this construction is a functor from the category
of monoids to the category of operads which respects injections and surjections.
Finally we apply this construction in Section~\ref{sec:ConstructionOperades}
on various monoids and obtain several new (symmetric or not) operads. We
construct in this way some operads on combinatorial objects which were not
provided with such a structure: planar rooted trees with a fixed arity,
Motzkin words, integer compositions, directed animals, and segmented integer
compositions. We also obtain new operads on objects which are already provided
with such a structure: endofunctions, parking functions, packed words,
permutations, planar rooted trees, and Schröder trees. By using the
construction~$\T$, we also give an alternative construction for the
diassociative operad~\cite{Lod01} and for the triassociative operad~\cite{LR04}.
\medskip

This paper is an extended version of~\cite{Gir12} and~\cite{Gir12b}. It
contains all proofs and new results like the presentations by generators and
relations of the considered nonsymmetric operads.
\medskip

{\it Acknowledgments.}
The author would like to thank N.~Bergeron, F.~Chapoton, F.~Hivert,
Y.~Le~Borgne, M. Méndez, J.-C.~Novelli, F.~Saliola, and Y.~Vargas for
their suggestions and interesting discussions. The computations of this
work have been done with the open-source mathematical software
Sage~\cite{Sage} and one of its extensions, Sage-Combinat~\cite{SageC}.
The author also wishes to thank the anonymous referees for their
highly valuable comments and suggestions.
\medskip

\section{Syntax trees and operads} \label{sec:Operades}
In this section, we set some notations about syntax trees and operads.
All our operads are reduced, {\em i.e.}, all of its elements have at
least one input. In the same way, all our syntax trees are reduced, {\em i.e.},
all of its internal nodes have at least one child. We also present in this
section some notions about rewriting systems on syntax trees that we
shall use to prove presentations of operads throughout the rest of the paper.
\medskip

\subsection{Rewriting systems on syntax trees}
In the same way as the elements of free monoids can be seen as words
and also are good objects to study monoids, the elements of free operads
can be seen as syntax trees are useful objects to manipulate these algebraic
structures.
\medskip

\subsubsection{Syntax trees}
We use in the sequel the standard terminology ({\em i.e.}, {\em node},
{\em edge}, {\em root}, {\em subtree}, {\em parent}, {\em child},
{\em path}, {\em etc.}) about trees (see for instance~\cite{AU94}). In
our graphical representations, internal nodes are represented by
circles~\scalebox{.25}{\tikz{\node[Noeud]{};}}, leaves by
squares~\scalebox{.25}{\tikz{\node[Feuille]{};}}, and edges by
segments~\scalebox{.25}{\tikz{\draw[Arete](0,0)--(0,1);}}. Each tree
is depicted so that its root is the uppermost node.
\medskip

Let~$S$ be a nonempty set. A {\em syntax tree on~$S$}, or simply a
{\em syntax tree} if~$S$ is fixed, is a planar rooted tree such that
internal nodes are labeled on~$S$. We shall denote by~$\CalT^S_n$ the set
of syntax trees on~$S$ with~$n$ leaves and by~$\CalT^S$ the
set~$\sqcup_{n \geq 1} \CalT^S_n$.
\medskip

Let~$T$ be a syntax tree. We denote by~$\NbNoeuds(T)$ (resp. $\NbFeuilles(T)$)
the number of internal nodes (resp. leaves) of~$T$. The {\em arity} of an
internal node is the number of its children. The {\em depth-first traversal}
of~$T$ consists in visiting the root of~$T$ and then, recursively visiting
by a depth-first traversal the subtrees of~$T$, from left to right. The
{\em $i$th internal node} (resp. {\em $i$th leaf}) of~$T$ is the $i$th
internal node (resp. $i$th leaf) of~$T$ visited by a depth-first traversal.
The {\em depth} of a node~$x$ of~$T$ is the length of the unique path
connecting~$x$ with the root of~$T$. Note that the depth of the root of~$T$
is~$0$. The {\em weight}~$\Poids(T)$ of~$T$ is the sum, for all internal
nodes~$x$ of~$T$, of the number of internal nodes of the rightmost
subtree of~$x$.
\medskip

For example,
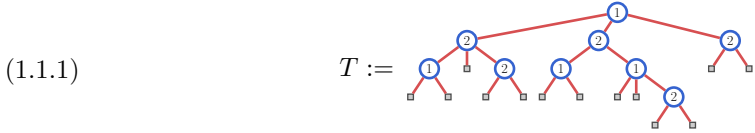
\begin{equation}
    \begin{split}T := \enspace \end{split}
    \begin{split}\scalebox{.25}{\begin{tikzpicture}
        \node[Feuille](0)at(0.00,-4.50){};
        \node[Feuille](11)at(9.00,-4.50){};
        \node[Feuille](13)at(11.00,-4.50){};
        \node[Feuille](15)at(12.00,-4.50){};
        \node[Feuille](16)at(13.00,-6.00){};
        \node[Feuille](18)at(15.00,-6.00){};
        \node[Feuille](19)at(16.00,-3.00){};
        \node[Feuille](2)at(2.00,-4.50){};
        \node[Feuille](21)at(18.00,-3.00){};
        \node[Feuille](4)at(3.00,-3.00){};
        \node[Feuille](5)at(4.00,-4.50){};
        \node[Feuille](7)at(6.00,-4.50){};
        \node[Feuille](9)at(7.00,-4.50){};
        \node[Noeud,EtiqClair](1)at(1.00,-3.00){$1$};
        \node[Noeud,EtiqClair](10)at(8.00,-3.00){$1$};
        \node[Noeud,EtiqClair](12)at(10.00,-1.50){$2$};
        \node[Noeud,EtiqClair](14)at(12.00,-3.00){$1$};
        \node[Noeud,EtiqClair](17)at(14.00,-4.50){$2$};
        \node[Noeud,EtiqClair](20)at(17.00,-1.50){$2$};
        \node[Noeud,EtiqClair](3)at(3.00,-1.50){$2$};
        \node[Noeud,EtiqClair](6)at(5.00,-3.00){$2$};
        \node[Noeud,EtiqClair](8)at(11.00,0.00){$1$};
        \draw[Arete](0)--(1);
        \draw[Arete](1)--(3);
        \draw[Arete](10)--(12);
        \draw[Arete](11)--(10);
        \draw[Arete](12)--(8);
        \draw[Arete](13)--(14);
        \draw[Arete](14)--(12);
        \draw[Arete](15)--(14);
        \draw[Arete](16)--(17);
        \draw[Arete](17)--(14);
        \draw[Arete](18)--(17);
        \draw[Arete](19)--(20);
        \draw[Arete](2)--(1);
        \draw[Arete](20)--(8);
        \draw[Arete](21)--(20);
        \draw[Arete](3)--(8);
        \draw[Arete](4)--(3);
        \draw[Arete](5)--(6);
        \draw[Arete](6)--(3);
        \draw[Arete](7)--(6);
        \draw[Arete](9)--(10);
    \end{tikzpicture}}
    \end{split}
\end{equation}
is a syntax tree on the set~$\{1, 2\}$. It has~$9$ internal nodes,~$13$
leaves, its weight is~$5$, and the sequence of the labels of its internal
nodes visited by the depth-first traversal is $121221122$.
\medskip

\subsubsection{Syntax tree patterns}
Let~$T$ and~$\alpha$ be two syntax trees. We say that~$T$
{\em admits an occurrence} of~$\alpha$ {\em at the root} if one of the
following two conditions is satisfied:
\begin{enumerate}[label = ({\it \roman*})]
    \item the tree~$\alpha$ consists in exactly one leaf and no internal
    node;
    \item the roots of~$T$ and~$\alpha$ have same arities and same labels,
    and, by denoting by~$(T_1, \dots, T_k)$ and~$(\alpha_1, \dots, \alpha_k)$
    the sequence of the subtrees of~$T$ and~$\alpha$ from left to
    right,~$T_i$ admits an occurrence of~$\alpha_i$ at the root, for
    any~$i \in [k]$.
\end{enumerate}
We say that~$T$ {\em admits an occurrence} of~$\alpha$ if there is a node~$x$
of~$T$ such that the syntax tree rooted on~$x$ admits an occurrence
of~$\alpha$ at the root.
\medskip

For example, set
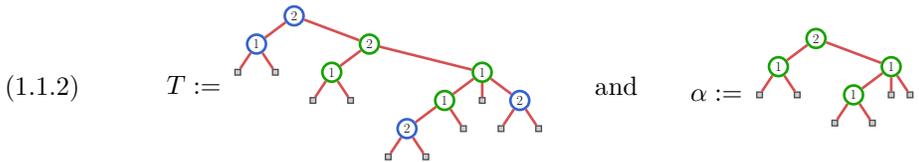
\begin{equation}
    \begin{split}T := \enspace \end{split}
    \begin{split}\scalebox{.25}{\begin{tikzpicture}
        \node[Feuille](0)at(0.00,-3.00){};
        \node[Feuille](10)at(10.00,-7.50){};
        \node[Feuille](12)at(12.00,-6.00){};
        \node[Feuille](14)at(13.00,-4.50){};
        \node[Feuille](15)at(14.00,-6.00){};
        \node[Feuille](17)at(16.00,-6.00){};
        \node[Feuille](2)at(2.00,-3.00){};
        \node[Feuille](4)at(4.00,-4.50){};
        \node[Feuille](6)at(6.00,-4.50){};
        \node[Feuille](8)at(8.00,-7.50){};
        \node[Noeud,EtiqClair](1)at(1.00,-1.50){$1$};
        \node[Noeud,Marque1,EtiqClair](11)at(11.00,-4.50){$1$};
        \node[Noeud,Marque1,EtiqClair](13)at(13.00,-3.00){$1$};
        \node[Noeud,EtiqClair](16)at(15.00,-4.50){$2$};
        \node[Noeud,EtiqClair](3)at(3.00,0.00){$2$};
        \node[Noeud,Marque1,EtiqClair](5)at(5.00,-3.00){$1$};
        \node[Noeud,Marque1,EtiqClair](7)at(7.00,-1.50){$2$};
        \node[Noeud,EtiqClair](9)at(9.00,-6.00){$2$};
        \draw[Arete](0)--(1);
        \draw[Arete](1)--(3);
        \draw[Arete](10)--(9);
        \draw[Arete](11)--(13);
        \draw[Arete](12)--(11);
        \draw[Arete](13)--(7);
        \draw[Arete](14)--(13);
        \draw[Arete](15)--(16);
        \draw[Arete](16)--(13);
        \draw[Arete](17)--(16);
        \draw[Arete](2)--(1);
        \draw[Arete](4)--(5);
        \draw[Arete](5)--(7);
        \draw[Arete](6)--(5);
        \draw[Arete](7)--(3);
        \draw[Arete](8)--(9);
        \draw[Arete](9)--(11);
    \end{tikzpicture}}
    \end{split}
    \qquad \mbox{and} \qquad
    \begin{split}\alpha := \enspace \end{split}
    \begin{split}\scalebox{.25}{\begin{tikzpicture}
        \node[Feuille](0)at(0.00,-3.00){};
        \node[Feuille](2)at(2.00,-3.00){};
        \node[Feuille](4)at(4.00,-4.50){};
        \node[Feuille](6)at(6.00,-4.50){};
        \node[Feuille](8)at(7.00,-3.00){};
        \node[Feuille](9)at(8.00,-3.00){};
        \node[Noeud,Marque1,EtiqClair](1)at(1.00,-1.50){$1$};
        \node[Noeud,Marque1,EtiqClair](3)at(3.00,0.00){$2$};
        \node[Noeud,Marque1,EtiqClair](5)at(5.00,-3.00){$1$};
        \node[Noeud,Marque1,EtiqClair](7)at(7.00,-1.50){$1$};
        \draw[Arete](0)--(1);
        \draw[Arete](1)--(3);
        \draw[Arete](2)--(1);
        \draw[Arete](4)--(5);
        \draw[Arete](5)--(7);
        \draw[Arete](6)--(5);
        \draw[Arete](7)--(3);
        \draw[Arete](8)--(7);
        \draw[Arete](9)--(7);
    \end{tikzpicture}}
    \end{split}\,.
\end{equation}
Then~$T$ admits one occurrence of~$\alpha$.
\medskip

\subsubsection{Rewrite rules on syntax trees}
Let~$S$ be a nonempty set and~$\CalT$ be a subset of~$\CalT^S$.
A {\em rewrite rule} on~$\CalT$ is a binary relation~$\mapsto$ on~$\CalT$
such that
\begin{equation}
    \alpha \mapsto \beta
    \quad \mbox{ implies } \quad
    \NbFeuilles(\alpha) = \NbFeuilles(\beta).
\end{equation}
A syntax tree~$T_0$ is {\em rewritable} into a syntax tree~$T_1$
by~$\mapsto$ if:
\begin{enumerate}[label = ({\it \roman*})]
    \item there are two syntax trees~$\alpha$ and~$\beta$ such
    that~$\alpha \mapsto \beta$;
    \item there is in~$T_0$ an occurrence of~$\alpha$;
    \item we obtain~$T_1$ by replacing an occurrence of~$\alpha$
    in~$T_0$ by~$\beta$.
\end{enumerate}
We denote this property by~$T_0 \to T_1$ and we call the pair~$(T_0, T_1)$
a {\em rewriting}. Moreover, if there is a sequence~$S_1, \dots, S_k$ of
syntax trees such that
\begin{equation}
    T_0 \to S_1 \to \dots \to S_k \to T_1,
\end{equation}
we say that~$T_0$ is {\em rewritable} into~$T_1$ by~$\to$ and we denote
this property by~$T_0 \xrightarrow{*} T_1$.
\medskip

If there is no infinite chain
\begin{equation}
    T_0 \to T_1 \to \cdots,
\end{equation}
we say that~$\mapsto$ is {\em terminating}. In this case, a syntax tree~$T$
that cannot be rewritten is a {\em normal form} for~$\mapsto$. When~$\mapsto$
is terminating and there are for all~$n \geq 1$ finitely many normal forms
for~$\mapsto$ with~$n$ leaves, we denote by~$\#(\mapsto_n)$ the number of
normal forms for~$\mapsto$ with~$n$ leaves. In this case, the
{\em generating series} of~$\mapsto$ is
\begin{equation}
    F_{\mapsto}(t) := \sum_{n \geq 1} \#(\mapsto_n) t^n.
\end{equation}
In the sequel, we shall make use of {\em regular specifications} to describe
normal forms of rewrite rules and obtain their generating series. Regular
specifications are formal grammars explaining how to build combinatorial objects
(see~\cite{FS09} for an introduction on regular specifications).
\medskip

To give an example, set~$S := \{1, 2, 3\}$ and consider the rewrite
rule~$\mapsto$ on~$\CalT^S$ defined by
\begin{equation}
    \begin{split}\scalebox{.25}{\begin{tikzpicture}
        \node[Feuille](0)at(0.,-3.){};
        \node[Noeud,EtiqClair](1)at(1.,-1.5){$1$};
        \node[Feuille](2)at(2.,-3.){};
        \node[Noeud,EtiqClair](3)at(3.,0.){$1$};
        \node[Feuille](4)at(4.,-1.5){};
        \draw[Arete](0)--(1);
        \draw[Arete](1)--(3);
        \draw[Arete](2)--(1);
        \draw[Arete](4)--(3);
    \end{tikzpicture}}\end{split}
    \begin{split} \quad \mapsto \quad \end{split}
    \begin{split}\scalebox{.25}{\begin{tikzpicture}
        \node[Feuille](0)at(0.,-1.5){};
        \node[Noeud,EtiqClair](1)at(1.,0.){$3$};
        \node[Feuille](2)at(1.,-1.5){};
        \node[Feuille](3)at(2.,-1.5){};
        \draw[Arete](0)--(1);
        \draw[Arete](2)--(1);
        \draw[Arete](3)--(1);
    \end{tikzpicture}}\end{split}
    \begin{split}\qquad \mbox{and} \qquad \end{split}
    \begin{split}\scalebox{.25}{\begin{tikzpicture}
        \node[Feuille](0)at(0.,-1.5){};
        \node[Noeud,EtiqClair](1)at(1.,0.){$1$};
        \node[Feuille](2)at(2.,-1.5){};
        \draw[Arete](0)--(1);
        \draw[Arete](2)--(1);
    \end{tikzpicture}}\end{split}
    \begin{split} \quad \mapsto \quad \end{split}
    \begin{split}\scalebox{.25}{\begin{tikzpicture}
        \node[Feuille](0)at(0.,-1.5){};
        \node[Noeud,EtiqClair](1)at(1.,0.){$2$};
        \node[Feuille](2)at(2.,-1.5){};
        \draw[Arete](0)--(1);
        \draw[Arete](2)--(1);
    \end{tikzpicture}}\end{split}\,.
\end{equation}
Here is a sequence of rewritings:
\begin{equation}
    \begin{split}\scalebox{.25}{\begin{tikzpicture}
        \node[Feuille](0)at(0.,-4.5){};
        \node[Noeud,EtiqClair](1)at(1.,-3.){$1$};
        \node[Feuille](2)at(2.,-4.5){};
        \node[Noeud,Marque1,EtiqClair](3)at(3.,-1.5){$1$};
        \node[Feuille](4)at(4.,-3.){};
        \node[Noeud,Marque1,EtiqClair](5)at(5.,0.){$1$};
        \node[Feuille](6)at(6.,-4.5){};
        \node[Noeud,EtiqClair](7)at(7.,-3.){$2$};
        \node[Feuille](8)at(8.,-4.5){};
        \node[Noeud,EtiqClair](9)at(9.,-1.5){$1$};
        \node[Feuille](10)at(10.,-4.5){};
        \node[Noeud,EtiqClair](11)at(11.,-3.){$3$};
        \node[Feuille](12)at(11.,-4.5){};
        \node[Feuille](13)at(12.,-4.5){};
        \draw[Arete](0)--(1);
        \draw[Arete](1)--(3);
        \draw[Arete](2)--(1);
        \draw[Arete](3)--(5);
        \draw[Arete](4)--(3);
        \draw[Arete](6)--(7);
        \draw[Arete](7)--(9);
        \draw[Arete](8)--(7);
        \draw[Arete](9)--(5);
        \draw[Arete](10)--(11);
        \draw[Arete](11)--(9);
        \draw[Arete](12)--(11);
        \draw[Arete](13)--(11);
    \end{tikzpicture}}\end{split}
    \begin{split} \quad \to \quad \end{split}
    \begin{split}\scalebox{.25}{\begin{tikzpicture}
        \node[Feuille](0)at(0.,-3.){};
        \node[Noeud,Marque1,EtiqClair](1)at(1.,-1.5){$1$};
        \node[Feuille](2)at(2.,-3.){};
        \node[Noeud,EtiqClair](3)at(3.,0.){$3$};
        \node[Feuille](4)at(3.,-1.5){};
        \node[Feuille](5)at(4.,-4.5){};
        \node[Noeud,EtiqClair](6)at(5.,-3.){$2$};
        \node[Feuille](7)at(6.,-4.5){};
        \node[Noeud,EtiqClair](8)at(7.,-1.5){$1$};
        \node[Feuille](9)at(8.,-4.5){};
        \node[Noeud,EtiqClair](10)at(9.,-3.){$3$};
        \node[Feuille](11)at(9.,-4.5){};
        \node[Feuille](12)at(10.,-4.5){};
        \draw[Arete](0)--(1);
        \draw[Arete](1)--(3);
        \draw[Arete](2)--(1);
        \draw[Arete](4)--(3);
        \draw[Arete](5)--(6);
        \draw[Arete](6)--(8);
        \draw[Arete](7)--(6);
        \draw[Arete](8)--(3);
        \draw[Arete](9)--(10);
        \draw[Arete](10)--(8);
        \draw[Arete](11)--(10);
        \draw[Arete](12)--(10);
    \end{tikzpicture}}\end{split}
    \begin{split} \quad \to \quad \end{split}
    \begin{split}\scalebox{.25}{\begin{tikzpicture}
        \node[Feuille](0)at(0.,-3.){};
        \node[Noeud,EtiqClair](1)at(1.,-1.5){$2$};
        \node[Feuille](2)at(2.,-3.){};
        \node[Noeud,EtiqClair](3)at(3.,0.){$3$};
        \node[Feuille](4)at(3.,-1.5){};
        \node[Feuille](5)at(4.,-4.5){};
        \node[Noeud,EtiqClair](6)at(5.,-3.){$2$};
        \node[Feuille](7)at(6.,-4.5){};
        \node[Noeud,Marque1,EtiqClair](8)at(7.,-1.5){$1$};
        \node[Feuille](9)at(8.,-4.5){};
        \node[Noeud,EtiqClair](10)at(9.,-3.){$3$};
        \node[Feuille](11)at(9.,-4.5){};
        \node[Feuille](12)at(10.,-4.5){};
        \draw[Arete](0)--(1);
        \draw[Arete](1)--(3);
        \draw[Arete](2)--(1);
        \draw[Arete](4)--(3);
        \draw[Arete](5)--(6);
        \draw[Arete](6)--(8);
        \draw[Arete](7)--(6);
        \draw[Arete](8)--(3);
        \draw[Arete](9)--(10);
        \draw[Arete](10)--(8);
        \draw[Arete](11)--(10);
        \draw[Arete](12)--(10);
    \end{tikzpicture}}\end{split}
    \begin{split} \quad \to \quad \end{split}
    \begin{split}\scalebox{.25}{\begin{tikzpicture}
        \node[Feuille](0)at(0.,-3.){};
        \node[Noeud,EtiqClair](1)at(1.,-1.5){$2$};
        \node[Feuille](2)at(2.,-3.){};
        \node[Noeud,EtiqClair](3)at(3.,0.){$3$};
        \node[Feuille](4)at(3.,-1.5){};
        \node[Feuille](5)at(4.,-4.5){};
        \node[Noeud,EtiqClair](6)at(5.,-3.){$2$};
        \node[Feuille](7)at(6.,-4.5){};
        \node[Noeud,EtiqClair](8)at(7.,-1.5){$2$};
        \node[Feuille](9)at(8.,-4.5){};
        \node[Noeud,EtiqClair](10)at(9.,-3.){$3$};
        \node[Feuille](11)at(9.,-4.5){};
        \node[Feuille](12)at(10.,-4.5){};
        \draw[Arete](0)--(1);
        \draw[Arete](1)--(3);
        \draw[Arete](2)--(1);
        \draw[Arete](4)--(3);
        \draw[Arete](5)--(6);
        \draw[Arete](6)--(8);
        \draw[Arete](7)--(6);
        \draw[Arete](8)--(3);
        \draw[Arete](9)--(10);
        \draw[Arete](10)--(8);
        \draw[Arete](11)--(10);
        \draw[Arete](12)--(10);
    \end{tikzpicture}}\end{split}.
\end{equation}

Since for any rewriting $T_0 \to T_1$,~$T_1$ has less nodes labeled by~$1$
than $T_0$,~$\mapsto$ is terminating. The normal forms of~$\mapsto$
are the syntax trees on~$S$ with no internal node labeled by~$1$ that has
two children.
\medskip

\subsection{Operads}
Let us now set, in our context, some definitions and notations about operads.
We shall use in the next sections the previous notions about syntax trees to
handle elements of free nonsymmetric operads and establish presentations by
generators and relations of nonsymmetric operads.
\medskip

\subsubsection{Nonsymmetric operads}
A {\em nonsymmetric operad}, or a {\em ns operad} for short, is a collection
\begin{equation}
    \CalP := \bigsqcup_{n \geq 1} \CalP(n),
\end{equation}
together with {\em partial composition maps}
\begin{equation}
    \circ_i : \CalP(n) \times \CalP(m) \to \CalP(n + m - 1),
    \qquad n, m \geq 1, i \in [n],
\end{equation}
and a distinguished element $\Unite \in \CalP(1)$, the {\em unit} of $\CalP$.
The above data has to satisfy the following relations:
\begin{equation} \label{eq:AssocSerie}
    (x \circ_i y) \circ_{i + j - 1} z = x \circ_i (y \circ_j z),
    \qquad x \in \CalP(n), y \in \CalP(m),
    z \in \CalP(k), i \in [n], j \in [m],
\end{equation}
\begin{equation} \label{eq:AssocParallele}
    (x \circ_i y) \circ_{j + m - 1} z = (x \circ_j z) \circ_i y,
    \qquad x \in \CalP(n), y \in \CalP(m),
    z \in \CalP(k), 1 \leq i < j \leq n,
\end{equation}
\begin{equation} \label{eq:Unite}
    \Unite \circ_1 x = x = x \circ_i \Unite,
    \qquad x \in \CalP(n), i \in [n].
\end{equation}
\medskip

One of the simplest ns operads is the {\em associative commutative operad}~$\Com$.
It is defined for all~$n \geq 1$ by
\begin{equation}
    \Com(n) := \left\{\alpha_n\right\},
\end{equation}
and the partial composition maps are defined by
\begin{equation}
    \alpha_n \circ_i \alpha_m := \alpha_{n + m - 1},
\end{equation}
for all~$n, m \geq 1$ and~$i \in [n]$.
\medskip

Let us now fix some terminology about ns operads and recall basic definitions.
The {\em arity} of an element~$x$ of~$\CalP(n)$ is~$n$ and is denoted
by~$|x|$. If~$\CalQ$ is a ns operad, a map~$\phi : \CalP \to \CalQ$ is a
{\em ns operad morphism} if it maps elements of arity~$n$ of~$\CalP$ to elements
of arity~$n$ of~$\CalQ$ and commutes with partial composition maps. The ns
operad~$\CalQ$ is a {\em ns suboperad} of~$\CalP$ if~$\CalQ \subseteq \CalP$,
$\Unite \in \CalQ$, and~$\CalQ$ is closed for the partial composition maps.
\medskip

For any set~$G \subseteq \CalP$, the {\em ns operad generated by~$G$} is the
smallest ns suboperad of~$\CalP$ which contains every element of~$G$. When the
ns operad generated by~$G$ is~$\CalP$ itself and~$G$ is minimal with respect to
inclusion among the subsets of~$\CalP$ satisfying this property, $G$ is a
{\em generating set} of~$\CalP$ and its elements are {\em generators} of~$\CalP$.
We say that~$\CalP$ is {\em finitely generated} if it admits a finite generating
set.
\medskip

The {\em Hilbert series} of a ns operad~$\CalP$ containing, for all~$n \geq 1$,
finitely many elements of arity~$n$ is the series
\begin{equation}
    F_\CalP(t) := \sum_{n \geq 1} \# \CalP(n) \, t^n.
\end{equation}
A {\em combinatorial ns operad} is a ns operad which admits an Hilbert series
and such that its only element of arity~$1$ is its unit.
\medskip

The {\em composition map} of~$\CalP$ is the mapping
\begin{equation}
    \circ : \CalP(n) \times \CalP(m_1) \times \dots \times \CalP(m_n) \to
    \CalP(m_1 + \dots + m_n),
\end{equation}
defined using partial composition maps~$\circ_i$ by
\begin{equation}
    x \circ [y_1, \dots, y_n] :=
    (\dots ((x \circ_n y_n) \circ_{n - 1} y_{n - 1}) \dots) \circ_1 y_1.
\end{equation}
The ns operad~$\CalP$ is {\em basic} if for all $y_1, \dots, y_n \in \CalP$,
the maps
\begin{equation}
    \gamma_{y_1, \dots, y_n} : \CalP(n) \to \CalP(|y_1| + \dots + |y_n|),
\end{equation}
defined by
\begin{equation}
    \gamma_{y_1, \dots, y_n}(x) := x \circ [y_1, \dots, y_n]
\end{equation}
are injective.
\medskip

\subsubsection{Symmetric operads}
Let~$\EnsPermu_n$ be the group of permutations of~$[n]$. Any permutation~$\sigma$
is denoted as a word~$\sigma_1 \dots \sigma_n$ in such a way that the
$i$th letter $\sigma_i$ is the image of~$i$. For instance, the word~$312$
represents the bijection~$1 \mapsto 3$, $2 \mapsto 1$, $3 \mapsto 2$.
\medskip

To define what is an operad, we need the following definition. Let~$\Permu$
be the ns operad satisfying for all~$n \geq 1$,
\begin{equation}
    \Permu(n) := \EnsPermu_n,
\end{equation}
and, for all~$\sigma \in \Permu(n)$, $\nu \in \Permu(m)$, and~$i \in [n]$,
\begin{equation}
    \sigma \circ_i \nu :=
    \sigma'_1 \dots \sigma'_{i - 1}
    \nu'_1 \dots \nu'_m
    \sigma'_{i + 1} \dots \sigma'_n,
\end{equation}
where
\begin{equation}
    \nu'_j := \nu_j + \sigma_i - 1, \qquad \mbox{for all $j \in [m]$},
\end{equation}
and
\begin{equation}
    \sigma'_j :=
    \begin{cases}
        \sigma_j & \mbox{if $\sigma_j < \sigma_i$}, \\
        \sigma_j + m - 1 & \mbox{otherwise},
    \end{cases}
    \qquad \mbox{for all $j \in [n]$.}
\end{equation}
This ns operad is known as the {\em associative noncommutative operad}.
\medskip

For instance, here are two examples of compositions in $\Permu$
\begin{equation}
    \textcolor{Bleu}{1}{\bf 2}\textcolor{Bleu}{3}
    \circ_2
    \textcolor{Rouge}{12}
    = \textcolor{Bleu}{1}\textcolor{Rouge}{23}\textcolor{Bleu}{4},
\end{equation}
\begin{equation}
    \textcolor{Bleu}{741}{\bf 5}\textcolor{Bleu}{623}
    \circ_4
    \textcolor{Rouge}{231}
    = \textcolor{Bleu}{941}\textcolor{Rouge}{675}\textcolor{Bleu}{823}.
\end{equation}
\medskip

A {\em symmetric operad}, or an {\em operad} for short, is a ns
operad~$\CalP$ together with a map
\begin{equation}
    \cdot : \CalP(n) \times \Permu(n) \to \CalP(n), \qquad n \geq 1,
\end{equation}
which satisfies the following relation:
\begin{equation} \label{eq:Equivariance}
    (x \cdot \sigma) \circ_i (y \cdot \nu) =
    \left(x \circ_{\sigma_i} y\right) \cdot
    \left(\sigma \circ_i \nu \right),
    \qquad x \in \CalP(n), y \in \CalP(m),
    \sigma \in \Permu(n), \nu \in \Permu(m), i \in [n],
\end{equation}
in such a way that~$\cdot$ also is a symmetric group action. Note that any
operad~$\CalP$ is also (and thus can be seen as) a ns operad by forgetting its
action of $\Permu$.
\medskip

If~$\CalQ$ is an operad, a map~$\phi : \CalP \to \CalQ$ is an
{\em operad morphism} if it is a ns operad morphism that commutes
with~$\cdot$. The operad~$\CalQ$ is a {\em suboperad} of~$\CalP$
if~$\CalQ$ is a ns suboperad of~$\CalP$ and~$\CalQ$ is closed
for~$\cdot$.
\medskip

For any set~$G \subseteq \CalP$, the {\em operad generated by~$G$} is the
smallest suboperad of~$\CalP$ which contains every element of~$G$. When the
operad generated by~$G$ is~$\CalP$ itself and~$G$ is minimal with respect to
inclusion among the subsets of~$\CalP$ satisfying this property, $G$ is a
{\em generating set} of~$\CalP$ and its elements are {\em generators} of~$\CalP$.
We say that~$\CalP$ is {\em finitely generated} if it admits a finite generating
set.
\medskip

\subsection{Presentation of nonsymmetric operads}
We now focus on ns operads and present the tools we will need to establish
presentations by generators and relations of ns operads.
\medskip

\subsubsection{Free ns operads}
Let~$S := \sqcup_{n \geq 1} S(n)$ be a set. The {\em free ns operad}
over~$S$ is the ns operad~$\CalF(S)$ defined as follows. For any~$n \geq 1$,
the set~$\CalF(S)(n)$ is the set of syntax trees on~$S$ with~$n$ leaves
and where internal nodes which have~$k$ children are labeled on~$S(k)$.
The composition~$x \circ_i y$ of two elements of~$\CalF(S)$ consists in
grafting the root of~$y$ on the $i$th leaf of~$x$. The unit~$\Unite$
of~$\CalF(S)$ is the tree with no internal node and hence exactly one leaf.
The {\em degree}~$\Deg(x)$ of an element~$x$ of~$\CalF(S)$ is its
number~$\NbNoeuds(x)$ of internal nodes and its {\em arity}~$|x|$
is its number~$\NbFeuilles(x)$ of leaves. If~$\Deg(x) = 1$, the
{\em label}~$\Etiq(x)$ of~$x$ is the element of~$S$ which labels the only
internal node of~$x$.
\medskip

Let~$S := S(1) \sqcup S(2) \sqcup S(3)$, where $S(1) := \{\La\}$,
$S(2) := \{\Lb, \Lc\}$, and $S(3) := \{\Ld\}$. Then,
\begin{equation}
    \begin{split} x := \end{split}
    \begin{split}\scalebox{.25}{
    \begin{tikzpicture}
        \node[Noeud,EtiqClair](0)at(0.,-3.){$\La$};
        \node[Feuille](1)at(0.,-4.5){};
        \node[Noeud,EtiqClair](2)at(1.,-1.5){$\Lc$};
        \node[Feuille](3)at(2.,-3.){};
        \node[Noeud,EtiqClair](4)at(3.,0.){$\Lb$};
        \node[Feuille](5)at(4.,-4.5){};
        \node[Noeud,EtiqClair](6)at(5.,-3.){$\Lc$};
        \node[Feuille](7)at(6.,-4.5){};
        \node[Noeud,EtiqClair](8)at(7.,-1.5){$\Ld$};
        \node[Noeud,EtiqClair](9)at(7.,-3.){$\La$};
        \node[Feuille](10)at(7.,-4.5){};
        \node[Feuille](11)at(8.,-4.5){};
        \node[Noeud,EtiqClair](12)at(9.,-3.){$\Ld$};
        \node[Feuille](13)at(9.,-4.5){};
        \node[Feuille](14)at(10.,-4.5){};
        \draw[Arete](0)--(2);
        \draw[Arete](1)--(0);
        \draw[Arete](2)--(4);
        \draw[Arete](3)--(2);
        \draw[Arete](5)--(6);
        \draw[Arete](6)--(8);
        \draw[Arete](7)--(6);
        \draw[Arete](8)--(4);
        \draw[Arete](9)--(8);
        \draw[Arete](10)--(9);
        \draw[Arete](11)--(12);
        \draw[Arete](12)--(8);
        \draw[Arete](13)--(12);
        \draw[Arete](14)--(12);
    \end{tikzpicture}}\end{split}
    \qquad \mbox{and} \qquad
    \begin{split} y := \end{split}
    \begin{split}\scalebox{.25}{
    \begin{tikzpicture}
        \node[Feuille](0)at(0.,-3.){};
        \node[Noeud,Marque1,EtiqClair](1)at(1.,-1.5){$\Lb$};
        \node[Feuille](2)at(2.,-3.){};
        \node[Noeud,Marque1,EtiqClair](3)at(3.,0.){$\Ld$};
        \node[Feuille](4)at(3.,-1.5){};
        \node[Noeud,Marque1,EtiqClair](5)at(5.,-1.5){$\La$};
        \node[Feuille](6)at(5.,-3.){};
        \draw[Arete](0)--(1);
        \draw[Arete](1)--(3);
        \draw[Arete](2)--(1);
        \draw[Arete](4)--(3);
        \draw[Arete](5)--(3);
        \draw[Arete](6)--(5);
    \end{tikzpicture}}\end{split}
\end{equation}
are two elements of~$\CalF(S)$. The arity of~$x$ is~$8$ and its degree is~$7$.
The arity of~$y$ is~$4$ and its degree is~$3$. Moreover, one has in~$\CalF(S)$
the following composition
\begin{equation}
    \begin{split} x \circ_2 y = \enspace \end{split}
    \begin{split}\scalebox{.25}{
    \begin{tikzpicture}
        \node[Noeud,EtiqClair](0)at(0.,-3.){$\La$};
        \node[Feuille](1)at(0.,-4.5){};
        \node[Noeud,EtiqClair](2)at(1.,-1.5){$\Lc$};
        \node[Feuille](3)at(2.,-6.){};
        \node[Noeud,Marque1,EtiqClair](4)at(3.,-4.5){$\Lb$};
        \node[Feuille](5)at(4.,-6.){};
        \node[Noeud,Marque1,EtiqClair](6)at(5.,-3.){$\Ld$};
        \node[Feuille](7)at(5.,-4.5){};
        \node[Noeud,Marque1,EtiqClair](8)at(7.,-4.5){$\La$};
        \node[Feuille](9)at(7.,-6.){};
        \node[Noeud,EtiqClair](10)at(7.,0.){$\Lb$};
        \node[Feuille](11)at(8.,-4.5){};
        \node[Noeud,EtiqClair](12)at(9.,-3.){$\Lc$};
        \node[Feuille](13)at(10.,-4.5){};
        \node[Noeud,EtiqClair](14)at(11.,-1.5){$\Ld$};
        \node[Noeud,EtiqClair](15)at(11.,-3.){$\La$};
        \node[Feuille](16)at(11.,-4.5){};
        \node[Feuille](17)at(12.,-4.5){};
        \node[Noeud,EtiqClair](18)at(13.,-3.){$\Ld$};
        \node[Feuille](19)at(13.,-4.5){};
        \node[Feuille](20)at(14.,-4.5){};
        \draw[Arete](0)--(2);
        \draw[Arete](1)--(0);
        \draw[Arete](2)--(10);
        \draw[Arete](3)--(4);
        \draw[Arete](4)--(6);
        \draw[Arete](5)--(4);
        \draw[Arete](6)--(2);
        \draw[Arete](7)--(6);
        \draw[Arete](8)--(6);
        \draw[Arete](9)--(8);
        \draw[Arete](11)--(12);
        \draw[Arete](12)--(14);
        \draw[Arete](13)--(12);
        \draw[Arete](14)--(10);
        \draw[Arete](15)--(14);
        \draw[Arete](16)--(15);
        \draw[Arete](17)--(18);
        \draw[Arete](18)--(14);
        \draw[Arete](19)--(18);
        \draw[Arete](20)--(18);
    \end{tikzpicture}}\end{split}.
\end{equation}
\medskip

\subsubsection{Evaluations}
Let $\left(\CalP, \circ_i^\CalP, \Unite_\CalP\right)$ be a ns operad
and~$X := \sqcup_{n \geq 1} X(n)$ be a set of elements of~$\CalP$,
where~$X(n) \subseteq \CalP(n)$ for all~$n \geq 1$. The
{\em evaluation}~$\Eval$ is the mapping
\begin{equation}
    \Eval : \CalF(X) \to \CalP,
\end{equation}
recursively defined by~$\Eval(\Unite) := \Unite_\CalP$, $\Eval(x) := \Etiq(x)$
if~$\Deg(x) = 1$, and
\begin{equation}
    \Eval(x) := \Eval(y) \circ_i^\CalP \Eval(z)
\end{equation}
if~$\Deg(x) \geq 2$, where~$y, z \in \CalF(X) \setminus \{\Unite\}$,
$i \in \left[|y|\right]$, and~$x = y \circ_i z$.
\medskip

In other words, we can see an element~$x$ of~$\CalF(X)$ as a tree-like
expression for an element~$\Eval(x)$ of~$\CalP$. Moreover, it is easy, by
induction on the degree, to prove that~$\Eval$ is a well-defined surjective
mapping and hence, is a ns operad morphism.
\medskip

\subsubsection{Ns operadic congruences}
An {\em ns operadic congruence} over a ns operad~$\CalP$ is an equivalence
relation~$\equiv$ on its elements such that~$x \equiv x'$ implies~$|x| = |x'|$,
and for all~$x, x', y, y' \in \CalP$ and~$i \in [|x|]$,
\begin{equation}
    x \equiv x' \mbox{ and } y \equiv y'
    \quad \mbox{ implies } \quad
    x \circ_i y \equiv x' \circ_i y'.
\end{equation}
\medskip

Given a set~$S := \sqcup_{n \geq 1} S(n)$ and an operadic congruence~$\equiv$
over the free ns operad~$\CalF(S)$, one can construct a {\em ns quotient operad}
$\CalF(S)/_\equiv$ of~$\CalF(S)$ defined as follows. We set
\begin{equation}
    \CalF(S)/_\equiv(n) := \left\{[x]_\equiv : x \in \CalF(S)(n)\right\},
    \qquad n \geq 1,
\end{equation}
where~$[x]_\equiv$ is the $\equiv$-equivalence class of~$x$, and
\begin{equation}
    [x]_\equiv \circ_i [y]_\equiv := \left[x \circ_i y\right]_\equiv,
\end{equation}
where~$x$ and~$y$ are any elements of~$\CalF(S)$ such that~$x \in [x]_\equiv$
and~$y \in [y]_\equiv$.
\medskip

\subsubsection{Presentation by generators and relations}
In the sequel, we shall define ns operadic congruences~$\equiv$ over free
ns operads~$\CalF(S)$ through equivalence relations~$\leftrightarrow$ on the
set~$\CalF(S)$. The {\em congruence generated by}~$\leftrightarrow$ is
the most refined ns operadic congruence~$\equiv$ containing~$\leftrightarrow$.
\medskip

Besides, we say that a relation~$\mapsto$ on~$\CalF(S)$ is an
{\em orientation} of~$\leftrightarrow$ if~$\mapsto$ is the finest
relation such that its reflexive, symmetric, and transitive closure
is~$\leftrightarrow$. The link between rewrite rules on syntax trees
and ns operads relies on the fact that orientations can be regarded
as rewrite rules.
\medskip

A {\em presentation} of a ns operad~$\CalP$ consists in a
set~$S := \sqcup_{n \geq 1} S(n)$ and a ns operadic congruence~$\equiv$
over~$\CalF(S)$ such that~$\CalP = \CalF(S)/_\equiv$. When~$S(2) \ne \emptyset$
and~$S(n) = \emptyset$ for all~$n \ne 2$, $\CalP$ is called {\em binary}.
When~$\equiv$ can be generated as a ns operadic congruence by an equivalence
relation~$\leftrightarrow$ on~$\CalF(S)$ only involving elements of
degree~$2$, $\CalP$ is called {\em quadratic}.
\medskip

The following lemma presents a description of the elements of ns operads
generated by a set of generators.
\begin{Lemme} \label{lem:GenerationOpNSEns}
    Let~$\CalP$ be a ns operad generated by a set~$G$ of generators.
    Then any element~$x$ of~$\CalP$ different from the unit of~$\CalP$
    can be written as
    \begin{equation}
        x = y \circ_i g,
    \end{equation}
    where~$y \in \CalP(n)$, $n \geq 1$, $g \in G$, and~$i \in [n]$.
\end{Lemme}
\begin{proof}
    Since the map~$\Eval : \CalF(G) \to \CalP$ is surjective, $x$ admits
    a tree-like expression~$x' \in \CalF(G)$ satisfying~$\Eval(x') = x$.
    Since~$x$ is different from the unit of~$\CalP$, one has~$x' = y' \circ_i g'$
    for two syntax trees~$y'$ and~$g'$ of~$\CalF(G)$ such that~$g'$ has exactly
    one internal node. Then, by setting~$y := \Eval(y')$ and~$g := \Eval(g')$,
    we have, since~$\Eval$ is a ns operad morphism, $x = y \circ_i g$.
\end{proof}
\medskip

We shall use in the sequel Lemma~\ref{lem:GenerationOpNSEns} to study ns
operads~$\CalQ$ generated by a subset of elements of a bigger ns operad~$\CalP$.
It allows us to describe~$\CalQ$ arity by arity because any element of~$\CalQ$
can be obtained by composing an element of a smaller arity with a generator.
\medskip

The following lemma presents a tool for showing that a given combinatorial
ns operad admits a specified presentation.
\begin{Lemme} \label{lem:PresentationReecriture}
    Let $\CalP$ be a combinatorial ns operad generated by a set~$G$ of
    generators and~$\equiv$ be a ns operadic congruence over~$\CalF(G)$
    generated by an equivalence relation~$\leftrightarrow$ on~$\CalF(G)$.
    If the following two conditions are satisfied together:
    \begin{enumerate}[label = ({\it \roman*})]
        \item for all $x, x' \in \CalF(G)$, $x \leftrightarrow x'$ implies
        $\Eval(x) = \Eval(x')$;
        \label{item:ReecritureConditionEquiv}
        \item there exists an orientation~$\mapsto$ of~$\leftrightarrow$
        such that~$\mapsto$ is terminating and has
        as many normal forms of arity~$n$ as elements of~$\CalP$ of arity~$n$;
        \label{item:ReecritureConditionFN}
    \end{enumerate}
    then~$\CalP$ admits the presentation~$\CalP = \CalF(G)/_\equiv$.
\end{Lemme}
\begin{proof}
    The definition of the evaluation map $\Eval$
    and~\ref{item:ReecritureConditionEquiv} imply that the map
    \begin{equation}
        \phi : \CalF(G)/_\equiv \to \CalP
    \end{equation}
    defined for any~$x \in \CalF(G)$
    by~$\phi\left([x]_\equiv\right) := \Eval(x)$ is a surjective ns operad
    morphism.
    \smallskip

    Since~$\mapsto$ is an orientation of~$\leftrightarrow$, for
    any~$x \in \CalF(G)$, there is at least one normal form for~$\mapsto$
    in~$[x]_\equiv$ and
    by~\ref{item:ReecritureConditionFN}, for all~$n \geq 1$, we have
    \begin{equation} \label{eq:CardinauxReecriture}
        \# \CalP(n) = \#(\mapsto_n) \; \geq \# \CalF(G)/_\equiv(n).
    \end{equation}
    This, together with the fact that~$\phi$ is surjective, implies
    that~$\phi$ also is an isomorphism. Hence, $\CalP$ admits the claimed
    presentation.
\end{proof}
\medskip

\section{A combinatorial functor from monoids to operads}
\label{sec:Foncteur}
We describe in this section the main ingredient of this paper, namely the
{\em construction~$\T$}. This functorial construction associates an
operad~$\T M$ with any monoid~$M$ and an operad
morphism~$\T \theta : \T M \to \T N$ with any monoid
morphism~$\theta : M \to N$.

\subsection{The construction} \label{subsec:Construction}

\subsubsection{From monoids to operads}
Let $(M, \bullet, 1)$ be a monoid. Let us denote by~$\T M$ the collection
\begin{equation}
    \T M := \bigsqcup_{n \geq 1} \T M(n),
\end{equation}
where for all~$n \geq 1$,
\begin{equation}
    \T M(n) := \left\{(x_1, \dots, x_n) : x_i \in M
    \mbox{ for all $i \in [n]$}\right\}.
\end{equation}
We endow the set~$\T M$ with maps
\begin{equation} \label{eq:TDomaineSubs}
    \circ_i : \T M(n) \times \T M(m) \to \T M(n + m - 1),
    \qquad n, m \geq 1, i \in [n],
\end{equation}
defined as follows: for all $x \in \T M(n)$, $y \in \T M(m)$, and~$i \in [n]$,
\begin{equation} \label{eq:TSub}
    x \circ_i y := (x_1, \dots, x_{i-1},
    x_i \bullet y_1, \dots, x_i \bullet y_m, x_{i+1}, \dots, x_n).
\end{equation}
Let us also set~$\Unite := (1)$ as a distinguished element of~$\T M(1)$.
We endow finally each set~$\T M(n)$ with a map
\begin{equation}
    \cdot : \T M(n) \times \Permu(n) \to \T M(n), \qquad n \geq 1,
\end{equation}
defined as follows: for all $x \in \T M(n)$ and $\sigma \in \Permu(n)$,
\begin{equation}
    x \cdot \sigma := \left(x_{\sigma_1}, \dots, x_{\sigma_n}\right).
\end{equation}

The elements of~$\T M$ are words over~$M$ regarded as an alphabet. The
arity~$|x|$ of an element~$x$ of~$\T M(n)$ is~$n$. For the sake of
readability, we shall denote in some cases an element $(x_1, \dots, x_n)$
of~$\T M(n)$ by its {\em word notation}~$x_1 \dots x_n$.
\medskip

\begin{Proposition} \label{prop:TOperade}
    If~$M$ is a monoid, then~$\T M$ is an operad.
\end{Proposition}
\begin{proof}
    This is a straightforward checking of the relations of operads:
    \eqref{eq:AssocSerie} comes from the fact that the product of~$M$
    is associative, \eqref{eq:AssocParallele} comes from the fact
    that the elements of~$\T M$ are words, \eqref{eq:Unite} comes from the
    fact that~$M$ has a unit, and~\eqref{eq:Equivariance} comes from the fact
    that~$\cdot$ acts by permuting the letters of the words.
\end{proof}
\medskip

\subsubsection{From monoids morphisms to operads morphisms}
Let~$M$ and~$N$ be two monoids and~$\theta : M \to N$ be a monoid morphism.
Let us denote by~$\T \theta$ the map
\begin{equation}
    \T \theta : \T M \to \T N,
\end{equation}
defined for all $(x_1, \dots, x_n) \in \T M(n)$ by
\begin{equation}
    \T \theta\left(x_1, \dots, x_n\right) :=
    \left(\theta(x_1), \dots, \theta(x_n)\right).
\end{equation}
\medskip

\begin{Proposition} \label{prop:TOperadeMorph}
    If~$M$ and~$N$ are two monoids and~$\theta : M \to N$ is a monoid
    morphism, then the map~$\T \theta : \T M \to \T N$ is an
    operad morphism.
\end{Proposition}
\begin{proof}
    This is a straightforward checking: the fact that~$\theta$ is a monoid
    morphism implies the statement of the proposition.
\end{proof}
\medskip

\subsection{Main properties of the construction} \label{subsec:ConstructionProp}

\subsubsection{Functoriality of $\T$}

\begin{Proposition} \label{prop:TInjSur}
    Let~$M$ and~$N$ be two monoids and $\theta : M \to N$ be a monoid
    morphism. If~$\theta$ is injective (resp. surjective), then~$\T \theta$
    is injective (resp. surjective).
\end{Proposition}
\begin{proof}
    This is a straightforward checking: the fact that~$\T \theta$ acts
    letter by letter implies the statement of the proposition.
\end{proof}
\medskip

\begin{Theoreme} \label{thm:TFonct}
    The construction~$\T$ is a functor from the category of monoids with
    monoid morphisms to the category of operads with operad morphisms.
    Moreover, $\T$ respects injections and surjections.
\end{Theoreme}
\begin{proof}
    By Proposition~\ref{prop:TOperade},~$\T$ constructs an operad from a
    monoid, and by Proposition~\ref{prop:TOperadeMorph}, an operad morphism
    from a monoid morphism. Now, since $\T$ sends identity monoid morphisms
    to identity operad morphisms and $\T$ commutes with map composition,
    $\T$ is a functor. Finally, by Proposition~\ref{prop:TInjSur},~$\T$ also
    respects injections and surjections, whence the statement of the theorem.
\end{proof}
\medskip

\subsubsection{Miscellaneous properties}

Recall that a monoid~$(M, \bullet)$ is {\em right cancellable} if for any
$x, y, z \in M$, $y \bullet x = z \bullet x$ implies~$y = z$.
\medskip

\begin{Proposition} \label{prop:OperadeBasique}
    Let~$M$ be a monoid. The operad~$\T M$ is basic if and only
    if~$M$ is a right cancellable monoid.
\end{Proposition}
\begin{proof}
    Let us denote by~$\bullet$ the product of~$M$.
    \smallskip

    Assume first that~$M$ is a right cancellable monoid.
    Let~$y^{(1)}, \dots, y^{(n)} \in \T M$, $x, x' \in \T M(n)$, and
    assume that
    \begin{equation}
        \gamma_{y^{(1)}, \dots, y^{(n)}}(x) =
        \gamma_{y^{(1)}, \dots, y^{(n)}}(x').
    \end{equation}
    Then, for any~$i \in [n]$ and~$j \in \left[\left|y^{(i)}\right|\right]$,
    \begin{equation}
        x_i \bullet y^{(i)}_j = x'_i \bullet y^{(i)}_j.
    \end{equation}
    Since~$M$ is right cancellable,~$x_i = x'_i$ and then,~$x = x'$. This implies
    that~$\gamma_{y^{(1)}, \dots, y^{(n)}}$ is injective and that~$\T M$
    is basic.
    \smallskip

    Conversely, assume now that~$\T M$ is basic. In particular, for
    any~$y \in \T M(1)$, the map~$\gamma_y$ is injective. Hence, for
    any~$x, x' \in \T M(1)$, the equality~$\gamma_y(x) = \gamma_y(x')$
    implies~$x = x'$. This is equivalent to say
    that~$x \bullet y = x' \bullet y$ implies~$x = x'$. This amounts
    exactly to say that~$M$ is a right cancellable monoid.
\end{proof}
\medskip

\begin{Proposition} \label{prop:GenerateursTM}
    Let~$M$ be a monoid generated by a set~$G$. The ns operad~$\T M$
    is generated by the set
    \begin{equation} \label{eq:GensTM}
        \{(g) : g \in G\} \cup \{(1, 1)\},
    \end{equation}
    where~$(1, 1) \in \T M(2)$ and~$1$ is the unit of~$M$.
\end{Proposition}
\begin{proof}
    Any element~$x := (x_1, \dots, x_n)$ of~$\T M$ can be generated by the
    elements of~\eqref{eq:GensTM} in the following way. First, generate
    the element~$y := (1, \dots, 1)$ of arity~$n$ by composing~$(1, 1)$ with
    itself~$n - 1$ times. Next, change each letter~$y_i$ of~$y$ by composing $y$
    with a sequence of generators of~$G$ to reach~$x_i$. This is possible
    since~$M$ is generated by~$G$.
\end{proof}
\medskip

\begin{Theoreme} \label{thm:algebre_sur_TM_ns}
    Let $(M, \bullet)$ be a monoid generated by a set $G := \{g_1, g_2, \dots\}$
    of generators satisfying a set $R$ of nontrivial relations. Then, any
    algebra $S$ over the ns operad $\T M$ is a set equipped with maps
    \begin{equation} \label{equ:gen_TM_1}
        \star : S \times S \to S
    \end{equation}
    and
    \begin{equation} \label{equ:gen_TM_2}
        \uparrow_g : S \to S, \qquad g \in G
    \end{equation}
    satisfying for all~$a, b, c \in S$, $g \in G$, and all relations
    $g_{i_1} \bullet \dots \bullet g_{i_n} = g_{j_1} \bullet \dots \bullet g_{j_m}$
    of~$R$, the equalities
    \begin{equation} \label{equ:relation_TM_1}
        (a \star b) \star c = a \star (b \star c),
    \end{equation}
    \begin{equation} \label{equ:relation_TM_2}
        (a \star b) \uparrow_g = a \uparrow_g \; \star \; b \uparrow_g,
    \end{equation}
    \begin{equation} \label{equ:relation_TM_3}
        a \uparrow_{g_{i_1}} \dots \; \uparrow_{g_{i_n}} =
        a \uparrow_{g_{j_1}} \dots \; \uparrow_{g_{j_m}}.
    \end{equation}
\end{Theoreme}
\begin{proof}
    Proving the statement of the theorem is equivalent to prove that the
    ns operad~$\T M$ admits the presentation by generators and relations
    obtained by traducing \eqref{equ:gen_TM_1}, \eqref{equ:gen_TM_2},
    \eqref{equ:relation_TM_1}, \eqref{equ:relation_TM_2},
    and \eqref{equ:relation_TM_3} in operadic terms. Thereby, this ns
    operad~$\CalP$ is the quotient of the free operad generated by a
    binary generator~$\star$ and unary generators~$\uparrow_g$, $g \in G$,
    submitted to the relations
    \begin{equation}
        \star \circ_1 \star = \star \circ_2 \star,
    \end{equation}
    \begin{equation}
        \uparrow_g \circ_1 \star = \star \circ [\uparrow_g, \uparrow_g],
        \qquad g \in G,
    \end{equation}
    \begin{equation}
        \uparrow_{g_{i_1}} \circ_1 \dots \circ_1 \uparrow_{g_{i_n}} =
        \uparrow_{g_{j_1}} \circ_1 \dots \circ_1 \uparrow_{g_{j_m}}
    \end{equation}
    for all relations
    $g_{i_1} \bullet \dots \bullet g_{i_n} = g_{j_1} \bullet \dots \bullet g_{j_m}$
    of~$R$.
    \smallskip

    Let~$\phi : \CalP \to \T M$ be ns operad morphism defined
    by~$\phi(\star) := (1, 1)$ and~$\phi(\uparrow_g) := (g)$ for any~$g \in G$,
    where~$1$ denotes the unit of~$M$. This morphism is well-defined since
    the elements~$(1, 1)$ and~$(g)$ of~$\T M$ satisfy the above relations
    by replacing~$\star$ by~$(1, 1)$ and~$\uparrow_g$ by~$(g)$.
    Proposition~\ref{prop:GenerateursTM} implies that~$\phi$ is surjective
    since, as a ns operad,~$\T M$ is generated by~$(1, 1)$ and~$(g)$,
    $g \in G$.
    \smallskip

    Now, since the equivalence classes of~$\CalP$ are clearly in bijection
    with the elements of~$\T M$, this shows that~$\phi$ is an isomorphism.
\end{proof}
\medskip

\section{Constructing operads} \label{sec:ConstructionOperades}
Through this section, we consider examples of applications of the functor~$\T$.
We shall mainly consider, given a monoid~$M$, some suboperads of~$\T M$,
symmetric or not, which have for all~$n \geq 1$ finitely many elements of
arity~$n$.
\medskip

For the most part of the constructed operads~$\CalP$, we shall establish
for all arities~$n \geq 1$, bijections $\phi : \CalP(n) \to \CalC_n$ between
the elements of~$\CalP$ of arity~$n$ and elements of size~$n$ of a
set~$\CalC := \sqcup_{n \geq 1} \CalC_n$ of combinatorial objects.
These bijections, in addition to show that~$\CalP$ are operads involving
the objects of~$\CalC$, allow us to define composition operations on~$\CalC$ by
interpreting the partial composition maps of~$\CalP$ on the elements of~$\CalC$.
\medskip

Moreover, we shall also establish presentations by generators and relations
of the constructed ns operads by using the tools provided by
Section~\ref{sec:Operades}.
\medskip

\subsection{Operads from the additive monoid} \label{subsec:MonoideAdditif}
We shall denote by~$\EnsNat$ the additive monoid of integers, and for
all~$\ell \geq 1$, by~$\EnsNat_\ell$ the quotient of~$\EnsNat$ consisting
in the set $\{0, 1, \dots, \ell - 1\}$ with the addition modulo $\ell$ as
the operation of $\EnsNat_\ell$.
\medskip

Note that since, by Theorem~\ref{thm:TFonct},~$\T$ is a functor which
respects surjective maps, $\T \EnsNat_\ell$ is a quotient operad of~$\T \EnsNat$.
Besides, since the monoids~$\EnsNat$ and~$\EnsNat_\ell$ are right cancellable,
by Proposition~\ref{prop:OperadeBasique}, the operads~$\T \EnsNat$
and~$\T \EnsNat_\ell$ are basic, and since any suboperad of a basic operad
is basic, all operads constructed in this section are basic.
\medskip

The ns operads constructed in this section fit into the diagram
of ns operads represented by Figure~\ref{fig:DiagrammeOperades}.
Table~\ref{tab:Operades} summarizes some information about these ns operads.
\begin{figure}[ht]
    \centering
    \begin{tikzpicture}[scale=.55]
        \node(TN)at(2,0){$\T \EnsNat$};
        \node(TN2)at(-4,-2){$\T \EnsNat_2$};
        \node(TN3)at(8,-2){$\T \EnsNat_3$};
        \node(End)at(-.5,-2){$\End$};
        \node(FP)at(-1.25,-4){$\FP$};
        \node(MT)at(-1.25,-6){$\MT$};
        \node(Per)at(-2,-8){$\Per$};
        \node(Schr)at(-.5,-8){$\Schr$};
        \node(FCat1)at(2,-10){$\FCat{1}$};
        \node(FCat2)at(2,-8){$\FCat{2}$};
        \node(FCat3)at(2,-6){$\FCat{3}$};
        \node(SComp)at(8,-10){$\SComp$};
        \node(AnD)at(8,-12){$\AnD$};
        \node(APE)at(5,-12){$\APE$};
        \node(Motz)at(-1,-12){$\Motz$};
        \node(Comp)at(-4,-12){$\Comp$};
        \node(FCat0)at(2,-14){$\FCat{0}$};
        \draw[Surjection](TN)--(TN2);
        \draw[Surjection](TN)--(TN3);
        \draw[Injection](End)--(TN);
        \draw[Injection](FP)--(End);
        \draw[Injection](MT)--(FP);
        \draw[Surjection](MT)--(Per);
        \draw[Injection](Schr)--(MT);
        \draw[Injection](FCat1)--(Schr);
        \draw[Injection](FCat1)--(FCat2);
        \draw[Injection](FCat2)--(FCat3);
        \draw[Injection,dashed](FCat3)--(TN);
        \draw[Surjection](FCat2)--(SComp);
        \draw[Injection](SComp)--(TN3);
        \draw[Surjection](FCat1)--(AnD);
        \draw[Injection](AnD)--(SComp);
        \draw[Injection](APE)--(FCat1);
        \draw[Injection](Motz)--(FCat1);
        \draw[Surjection](FCat1)--(Comp);
        \draw[Injection](Comp)--(TN2);
        \draw[Surjection](Comp)--(FCat0);
        \draw[Surjection](APE)--(FCat0);
        \draw[Surjection](AnD)--(FCat0);
        \draw[Injection](FCat0)--(Motz);
    \end{tikzpicture}
    \caption{The diagram of ns suboperads and quotients of~$\T \EnsNat$.
    Arrows~$\rightarrowtail$ (resp.~$\twoheadrightarrow$) are injective
    (resp. surjective) ns operad morphisms.}
    \label{fig:DiagrammeOperades}
\end{figure}
\begin{table}[ht]
    \centering
    \begin{tabular}{c|c|c|c|c}
        Monoid & Ns operad & Generators & First dimensions
            & Combinatorial objects \\ \hline \hline
        \multirow{8}{*}{$\EnsNat$} & $\End$ & --- & $1, 4, 27, 256, 3125$
            & Endofunctions \\
        & $\FP$ & --- & $1, 3, 16, 125, 1296$ & Parking functions \\
        & $\MT$ & --- & $1, 3, 13, 75, 541$ & Packed words \\
        & $\Per$ & --- & $1, 2, 6, 24, 120$ & Permutations \\
        & $\APE$ & $01$  & $1, 1, 2, 5, 14, 42$ & Planar rooted trees \\
        & $\FCat{k}$ & $00$, $01$, \dots, $0k$ & Fu\ss-Catalan numbers
            & $k$-leafy trees \\
        & $\Schr$ & $00$, $01$, $10$ & $1, 3, 11, 45, 197$ & Schröder trees \\
        & $\Motz$ & $00$, $010$ & $1, 1, 2, 4, 9, 21, 51$
            & Motzkin words \\ \hline
        $\EnsNat_2$ & $\Comp$ & $00$, $01$ & $1, 2, 4, 8, 16, 32$
            & Int. compo. \\ \hline
        \multirow{2}{*}{$\EnsNat_3$} & $\AnD$ & $00$, $01$
            & $1, 2, 5, 13, 35, 96$ & Directed animals \\
        & $\SComp$ & $00$, $01$, $02$ & $1, 3, 27, 81, 243$
            & Seg. int. compo.
    \end{tabular} \vspace{.5em}
    \caption{Ground monoids, generators, first dimensions, and combinatorial
    objects involved in the ns suboperads and quotients of~$\T \EnsNat$.}
    \label{tab:Operades}
\end{table}
\medskip

\subsubsection{Operads on endofunctions, parking functions, packed words,
and permutations}
Recall that an {\em endofunction} of size~$n$ is a word $x$ of length~$n$
on the alphabet~$\{1, \dots, n\}$. A {\em parking function} of size~$n$ is
an endofunction~$x$ of size~$n$ such that the nondecreasing rearrangement~$y$
of~$x$ satisfies~$y_i \leq i$ for all~$i \in [n]$. A {\em packed word}
of size~$n$ is an endofunction~$x$ of size~$n$ such that for any
letter~$x_i \geq 2$ of~$x$, there is in~$x$ a letter~$x_j = x_i - 1$.
\medskip

Note that neither the set of endofunctions nor the set of parking functions,
packed words, and permutations are suboperads of~$\T \EnsNat$. Indeed, one
has the following counterexample:
\begin{equation}
    \textcolor{Bleu}{1}{\bf 2} \circ_2 \textcolor{Rouge}{12} =
    \textcolor{Bleu}{1}\textcolor{Rouge}{34},
\end{equation}
and, even if~$12$ is a permutation,~$134$ is not an endofunction.
\medskip

Therefore, let us call a word~$x$ a {\em twisted} endofunction (resp.
parking function, packed word, permutation) if the word
$(x_1 + 1, x_2 + 1, \dots, x_n + 1)$ is an endofunction (resp. parking
function, packed word, permutation). For example, the word~$2300$ is a
twisted endofunction since~$3411$ is an endofunction. Let us denote by~$\End$
(resp.~$\FP$, $\MT$, $\Per$) the set of twisted endofunctions (resp. parking
functions, packed words, permutations). Under this reformulation, one has
the following result:
\begin{Proposition} \label{prop:OpEndFPMT}
    The sets~$\End$, $\FP$, and~$\MT$ form suboperads of~$\T \EnsNat$.
\end{Proposition}
\begin{proof}
    First, by definition of the partial composition map of~$\T \EnsNat$, the set
    of twisted endofunctions forms a suboperad of~$\T \EnsNat$.
    \smallskip

    Let~$x$ and~$y$ be two twisted parking functions (resp. packed words)
    and~$i \in [|x|]$. Since~$x$ and~$y$ have by definition at least one
    occurrence of~$0$, we have in~$\T \EnsNat$,
    \begin{equation}
        \Alphab(x \circ_i y) =
        \Alphab(x) \cup \{x_i + a : a \in \Alphab(y)\},
    \end{equation}
    where~$\Alphab(u)$ is the set~$\{u_j : j \in [|u|]\}$.
    This, in addition to the fact that any permutation of a twisted
    parking function (resp. packed word) is still a twisted parking functions
    (resp. packed word), shows that the partial composition maps of~$\T \EnsNat$
    and the map~$\cdot$ are still well-defined in~$\FP$ (resp. $\MT$).
\end{proof}
\medskip

For example, we have in~$\End$ the following composition
\begin{equation}
    \textcolor{Bleu}{2}{\bf 1} \textcolor{Bleu}{23}
    \circ_2
    \textcolor{Rouge}{30313} =
    \textcolor{Bleu}{2}\textcolor{Rouge}{41424}\textcolor{Bleu}{23},
\end{equation}
and the following application of the map~$\cdot$
\begin{equation}
    11210 \cdot 23514 = 12011.
\end{equation}

Note that~$\End$ is not a finitely generated operad. Indeed,
the twisted endofunctions~$x$ of size~$n$ satisfying~$x_i := n - 1$ for
all~$i \in [n]$ cannot be obtained by compositions involving elements
of~$\End$ of arity smaller than~$n$. Similarly,~$\FP$ is not a finitely
generated operad since the twisted parking functions~$x$ of
size~$n$ satisfying~$x_i := 0$ for all~$i \in [n - 1]$ and~$x_n := n - 1$
cannot be obtained by compositions involving elements of~$\FP$ of arity
smaller than~$n$.
\medskip

However, the operad~$\MT$ is a finitely generated operad:
\begin{Proposition}  \label{prop:GenerationMT}
    The operad~$\MT$ is the suboperad of~$\T \EnsNat$ generated
    by the elements~$00$ and~$01$.
\end{Proposition}
\begin{proof}
    Let~$\CalP$ be the suboperad of~$\T \EnsNat$ generated by
    the elements~$00$ and~$01$, and let us show that~$\CalP = \MT$.
    \smallskip

    First, by Proposition~\ref{prop:OpEndFPMT}, since~$00$ and~$01$ are
    twisted packed words, the elements of~$\CalP$ also are twisted packed
    words.
    \smallskip

    Now let~$x$ be a nondecreasing twisted packed word and let us show by
    induction on the size of~$x$ that~$x \in \CalP$. If~$|x| = 1$,
    since~$x$ is a twisted packed word, one has~$x = 0$ and since~$0$ is
    the unit of~$\T \EnsNat$, $x \in \CalP$. Otherwise, let~$y$ be the
    prefix of size~$n - 1$ of~$x$. Since~$x$ is a nondecreasing word,
    there are two possibilities to express the last letter~$x_n$ of~$x$
    from the letter~$x_{n - 1}$. If $x_n = x_{n - 1}$, we have
    $x = y \circ_{n - 1} 00$, and if $x_n = x_{n - 1} + 1$, we have
    $x = y \circ_{n - 1} 01$. Hence, since by induction hypothesis~$\CalP$
    contains~$y$, $\CalP$ also contains~$x$. Finally, since any twisted
    packed word~$z$ can be obtained from a nondecreasing packed word~$x$
    by permuting its letters, we have~$z = x \cdot \sigma$ for a certain
    permutation~$\sigma$ of~$\Permu(n)$, and hence,~$\CalP = \MT$.
\end{proof}
\medskip

Let~$\K$ be a field and let us from now consider that~$\MT$ is an
operad in the category of $\K$-vector spaces, {\em i.e.},~$\MT$ is the
free $\K$-vector space over the set of twisted packed words with partial
composition maps and the map~$\cdot$ extended by linearity. For more details on
operads in the category of vector spaces, we redirect the reader
to~\cite{LV12}.
\medskip

Let~$I$ be the free $\K$-vector space over the set of twisted packed words
having multiple occurrences of a same letter.
\medskip

\begin{Proposition} \label{prop:IdealDeMT}
    The vector space~$I$ is an operadic ideal of~$\MT$. Moreover, the
    operadic quotient~$\MT/_I$ is the free vector space over the set of
    twisted permutations~$\Per$ and, for all twisted permutations~$x$
    and~$y$, the partial composition map in~$\Per$ is expressed as
    \begin{equation} \label{eq:SubsPartiellePer}
        x \circ_i y =
        \begin{cases}
            x \circ_i y & \mbox{if $x_i = |x|$,} \\
            0_\K         & \mbox{otherwise,}
        \end{cases}
    \end{equation}
    where~$0_\K$ is the null vector of~$\Per$ and the partial composition
    map~$\circ_i$ in the right member of~\eqref{eq:SubsPartiellePer} is the
    partial composition map of~$\MT$.
\end{Proposition}
\begin{proof}
    Let~$x$ be a twisted packed word and~$y$ be a twisted packed word having
    multiple occurrences of a same letter. Since~$x$ and~$y$ have at least
    one occurrence of~$0$, any composition involving~$x$ and~$y$ also has
    multiple occurrences of a same letter. Moreover, for any permutation~$\sigma$
    of~$\Permu$ of size~$|y|$, $y \cdot \sigma$ also has multiple occurrences
    of a same letter. Hence,~$I$ is an operadic ideal of~$\MT$ and one
    can consider the operadic quotient~$\MT/_I$.
    \smallskip

    Since twisted packed words with no multiple occurrence of a same letter
    are twisted permutations,~$\MT/_I$ can be identified with the
    $\K$-vector space over the set of twisted permutations~$\Per$
    and~\eqref{eq:SubsPartiellePer} follows from the fact that the
    composition~$x \circ_i y$ of two twisted permutations~$x$ and~$y$ is
    still a twisted permutation if and only if~$x_i$ is the greatest letter
    of~$x$.
\end{proof}
\medskip

Here are two examples of compositions in~$\Per$
\begin{equation}
    {\bf 2}\textcolor{Bleu}{0431} \circ_1 \textcolor{Rouge}{102} = 0_\K,
\end{equation}
\begin{equation}
    \textcolor{Bleu}{20}{\bf 4}\textcolor{Bleu}{31} \circ_3 \textcolor{Rouge}{102}
        = \textcolor{Bleu}{20}\textcolor{Rouge}{546}\textcolor{Bleu}{31}.
\end{equation}
\medskip

\subsubsection{A ns operad on planar rooted trees} \label{subsubsec:APE}
Let~$\APE$ be the ns suboperad of~$\T \EnsNat$ generated by~$01$.
The following table shows the first elements of~$\APE$.
\begin{center}
    \begin{tabular}{c|p{11cm}}
        Arity & Elements of~$\APE$ \\ \hline \hline
        $1$ & $0$ \\ \hline
        $2$ & $01$ \\ \hline
        $3$ & $011$, $012$ \\ \hline
        $4$ & $0111$, $0112$, $0121$, $0122$, $0123$ \\ \hline
        $5$ & $01111$, $01112$, $01121$, $01122$, $01123$, $01211$,
              $01212$, $01221$, $01222$, $01223$, $01231$, $01232$,
              $01233$, $01234$
    \end{tabular}
\end{center}
\medskip

One has the following characterization of the elements of~$\APE$:
\begin{Proposition} \label{prop:MotsAPE}
    The elements of~$\APE$ are exactly the words~$x$ on the alphabet~$\EnsNat$
    satisfying~$x_1 = 0$ and $1 \leq x_{i + 1} \leq x_i + 1$ for
    all~$i \in [|x| - 1]$.
\end{Proposition}
\begin{proof}
    Let us first show by induction on the length of the words that any
    word~$x$ of~$\APE$ satisfies the statement. This is true
    when~$|x| = 1$. When~$|x| \geq 2$, by Lemma~\ref{lem:GenerationOpNSEns},
    there is an element~$y$ of~$\APE$ of length~$n := |x| - 1$ and an
    integer~$i \in [n]$ such that~$x = y \circ_i 01$. We have
    \begin{equation}
        x = (y_1, \dots, y_{i - 1}, y_i, y_i + 1, y_{i + 1}, \dots, y_n).
    \end{equation}
    Since~$x_{i + 1} = x_i + 1$ and since, by induction hypothesis,~$y$
    satisfies the statement,~$x$ also satisfies~it.
    \smallskip

    Let us now show by induction on the length of the words that~$\APE$
    contains any word~$x$ satisfying the statement. This is true
    when~$|x| = 1$. When~$n := |x| \geq 2$, since~$x_1 = 0$ and~$x_2 = 1$,
    $x$ has a factor~$x_i x_{i + 1}$ where~$i$ is the greatest integer such
    that~$x_{i+1} = x_i + 1$. Now, by setting
    \begin{equation}
        y := (x_1, \dots, x_i, x_{i+2}, \dots, x_n),
    \end{equation}
    we have~$x = y \circ_i 01$, and, since~$i$ is maximal, if~$i + 2 \leq n$
    we have~$x_{i+2} \leq x_{i+1}$. This implies that~$y$ satisfies the
    statement. By induction hypothesis,~$\APE$ contains~$y$ and,
    since~$x = y \circ_i 01$, $\APE$ also contains~$x$.
\end{proof}
\medskip

Recall that there are~$\frac{1}{n} \binom{2n - 2}{n - 1}$ planar rooted
trees with~$n$ nodes. There is a bijection~$\phi_\APE$
between the words of~$\APE$ of arity~$n$ and planar rooted trees with~$n$
nodes.
\medskip

To compute~$\phi_\APE(x)$ where~$x$ is an element of~$\APE$, iteratively
insert the letters of~$x$ from left to right according to the following
procedure. If~$|x| = 1$, then~$x = 0$ and~$\phi_\APE(0)$ is the only
planar rooted tree with one node. Otherwise, the insertion of a
letter~$\La \geq 1$ into a planar rooted tree~$T$ consists in grafting
in~$T$ a new node as the rightmost child of the last node of depth~$\La - 1$
for the depth-first traversal of~$T$.
\medskip

The inverse bijection is computed as follows. Given a planar rooted tree~$T$
of size~$n$, one computes an element of~$\APE$ of arity~$n$ by labelling
each node of~$T$ by its depth and then, by reading its labels following
a depth-first traversal of~$T$.
\medskip

Since the elements of~$\APE$ satisfy
Proposition~\ref{prop:MotsAPE},~$\phi_\APE$ is well-defined. Hence, we can
regard the elements of arity~$n$ of~$\APE$ as planar rooted trees with~$n$
nodes. Figure~\ref{fig:InterpretationAPE} shows an example of this bijection.
\begin{figure}[ht]
    \centering
    \begin{equation*}
        \begin{split}0112333212 \quad \xrightarrow{\phi_\APE} \quad \end{split}
        \begin{split}\scalebox{.3}{\begin{tikzpicture}
            \node[Noeud,EtiqClair](1)at(0,0){\Huge $0$};
            \node[Noeud,EtiqClair](2)at(-2,-2){\Huge $1$};
            \node[Noeud,EtiqClair](3)at(0,-2){\Huge $1$};
            \node[Noeud,EtiqClair](4)at(-1,-4){\Huge $2$};
            \node[Noeud,EtiqClair](5)at(-2.5,-6){\Huge $3$};
            \node[Noeud,EtiqClair](6)at(-1,-6){\Huge $3$};
            \node[Noeud,EtiqClair](7)at(0.5,-6){\Huge $3$};
            \node[Noeud,EtiqClair](8)at(1,-4){\Huge $2$};
            \node[Noeud,EtiqClair](9)at(2,-2){\Huge $1$};
            \node[Noeud,EtiqClair](10)at(3,-4){\Huge $2$};
            \draw[Arete](1)--(2);
            \draw[Arete](1)--(3);
            \draw[Arete](3)--(4);
            \draw[Arete](4)--(5);
            \draw[Arete](4)--(6);
            \draw[Arete](4)--(7);
            \draw[Arete](3)--(8);
            \draw[Arete](1)--(9);
            \draw[Arete](9)--(10);
        \end{tikzpicture}}\end{split}
        \quad \longleftrightarrow \quad
        \begin{split}\scalebox{.3}{\begin{tikzpicture}
            \node[Noeud](1)at(0,0){};
            \node[Noeud](2)at(-2,-2){};
            \node[Noeud](3)at(0,-2){};
            \node[Noeud](4)at(-1,-4){};
            \node[Noeud](5)at(-2.5,-6){};
            \node[Noeud](6)at(-1,-6){};
            \node[Noeud](7)at(0.5,-6){};
            \node[Noeud](8)at(1,-4){};
            \node[Noeud](9)at(2,-2){};
            \node[Noeud](10)at(3,-4){};
            \draw[Arete](1)--(2);
            \draw[Arete](1)--(3);
            \draw[Arete](3)--(4);
            \draw[Arete](4)--(5);
            \draw[Arete](4)--(6);
            \draw[Arete](4)--(7);
            \draw[Arete](3)--(8);
            \draw[Arete](1)--(9);
            \draw[Arete](9)--(10);
        \end{tikzpicture}}\end{split}
    \end{equation*}
    \caption{Interpretation of an element of the ns operad~$\APE$ in terms
    of planar rooted trees via the bijection~$\phi_\APE$. The nodes of
    the planar rooted tree in the middle are labeled by their depth.}
    \label{fig:InterpretationAPE}
\end{figure}
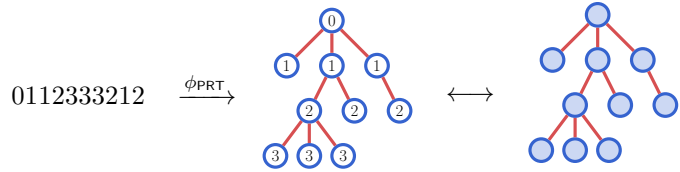
\medskip

The bijection~$\phi_\APE$ between elements of~$\APE$ and planar rooted
trees offers an alternative way to compute the composition of elements of~$\APE$:
\begin{Proposition} \label{prop:SubsAPE}
    Let~$S$ and~$T$ be two planar rooted trees and~$s$ be the $i$th node
    for the depth-first traversal of~$S$. The composition~$S \circ_i T$ in~$\APE$
    amounts to replace~$s$ by the root of~$T$ and graft the children
    of~$s$ as rightmost sons of the root of~$T$.
\end{Proposition}
\begin{proof}
    Let~$x \in \APE(n)$ and~$y \in \APE(m)$ such that~$S := \phi_\APE(x)$
    and~$T := \phi_\APE(y)$. Let $U := \phi_\APE(x \circ_i y)$. By
    definition of~$\phi_\APE$ and the partial composition maps of~$\APE$,~$U$ is
    obtained by inserting the prefix of length~$i - 1$ of~$x$, then the
    letters of~$y$ incremented by~$x_i$, and finally, the suffix of
    length~$n - i$ of~$x$. Since by Proposition~\ref{prop:MotsAPE},~$y$
    starts by~$0$, the nodes created by inserting the letters of~$y$
    incremented by~$x_i$ are descendants of the node created by
    inserting~$x_i = y_1$. Moreover, the nodes corresponding to the letters
    of the suffix of length~$n - i$ of~$x$ have same parents as they have
    in~$T$. This implies the statement.
\end{proof}
\medskip

Figure~\ref{fig:SubsAPE} shows an example of composition in~$\APE$.
\begin{figure}[ht]
    \centering
    \begin{equation*}
        \begin{split}\scalebox{.3}{\begin{tikzpicture}
            \node[Noeud](1)at(0,0){};
            \node[Noeud,Marque2](2)at(-1,-2){};
            \node[Noeud](3)at(-1,-4){};
            \node[Noeud](4)at(1,-2){};
            \draw[Arete](1)--(2);
            \draw[Arete](2)--(3);
            \draw[Arete](1)--(4);
        \end{tikzpicture}}\end{split}
        \enspace \circ_2 \enspace
        \begin{split}\scalebox{.3}{\begin{tikzpicture}
            \node[Noeud,Marque1](1)at(0,0){};
            \node[Noeud,Marque1](2)at(-1.5,-2){};
            \node[Noeud,Marque1](3)at(0,-2){};
            \node[Noeud,Marque1](4)at(0,-4){};
            \node[Noeud,Marque1](5)at(1.5,-2){};
            \draw[Arete](1)--(2);
            \draw[Arete](1)--(3);
            \draw[Arete](3)--(4);
            \draw[Arete](1)--(5);
        \end{tikzpicture}}\end{split}
        \enspace = \enspace
        \begin{split}\scalebox{.3}{\begin{tikzpicture}
            \node[Noeud](1)at(0,0){};
            \node[Noeud,Marque1](2)at(-1,-2){};
            \node[Noeud,Marque1](3)at(-3.5,-4){};
            \node[Noeud,Marque1](4)at(-2,-4){};
            \node[Noeud,Marque1](5)at(-2,-6){};
            \node[Noeud,Marque1](6)at(-.5,-4){};
            \node[Noeud](7)at(1,-4){};
            \node[Noeud](8)at(1,-2){};
            \draw[Arete](1)--(2);
            \draw[Arete](2)--(3);
            \draw[Arete](2)--(4);
            \draw[Arete](4)--(5);
            \draw[Arete](2)--(6);
            \draw[Arete](2)--(7);
            \draw[Arete](1)--(8);
        \end{tikzpicture}}\end{split}
    \end{equation*}
    \caption{Interpretation of the partial composition map of the ns operad~$\APE$
    in terms of planar rooted trees.}
    \label{fig:SubsAPE}
\end{figure}
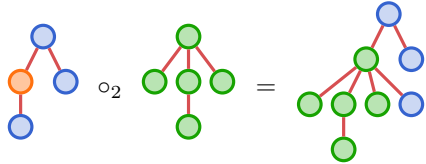
\medskip

\begin{Proposition} \label{prop:PresentationAPE}
    The ns operad~$\APE$ is isomorphic to the free ns operad generated by one
    element of arity~$2$.
\end{Proposition}
\begin{proof}
    By the characterization of its elements given by
    Proposition~\ref{prop:MotsAPE} and the bijection $\phi_\APE$, there are as
    many elements in~$\APE$ of arity~$n$ than elements of arity~$n$ of the free
    ns operad generated by one element of arity~$2$. These two ns operads
    are hence isomorphic.
\end{proof}
\medskip

Proposition~\ref{prop:PresentationAPE} also says that~$\APE$ is isomorphic
to the magmatic operad and hence, that~$\APE$ is a realization of the magmatic
operad. This result is already known since in~\cite{MY91}, Méndez and Yang
point out that the species of parenthesizations (binary trees) and the species
of planar rooted trees are isomorphic. This isomorphism implies that these
species are also isomorphic as ns operads. Moreover,~$\APE$ can be seen as
a planar version of the {\em non-associative permutative operad}~$\NAP$
\cite{MY91} (see also~\cite{Liv06}) seen as a ns operad, which is an operad
involving labeled non-planar rooted trees.
\medskip

\subsubsection{A ns operad on leafy trees with a fixed arity}
Let~$k \geq 0$ be an integer and~$\FCat{k}$ be the ns suboperad
of~$\T \EnsNat$ generated by~$00$,~$01$,~\dots,~$0k$. The following tables,
respectively, show the first elements of~$\FCat{1}$ and~$\FCat{2}$.
\begin{center}
    \begin{tabular}{c|p{11cm}}
        Arity & Elements of~$\FCat{1}$ \\ \hline \hline
        $1$ & $0$ \\ \hline
        $2$ & $00$, $01$ \\ \hline
        $3$ & $000$, $001$, $010$, $011$, $012$ \\ \hline
        $4$ & $0000$, $0001$, $0010$, $0011$, $0012$, $0100$, $0101$,
              $0110$, $0111$, $0112$, $0120$, $0121$, $0122$, $0123$
    \end{tabular}
\end{center}
\begin{center}
    \begin{tabular}{c|p{11cm}}
        Arity & Elements of~$\FCat{2}$ \\ \hline \hline
        $1$ & $0$ \\ \hline
        $2$ & $00$, $01$, $02$ \\ \hline
        $3$ & $000$, $001$, $002$, $010$, $011$, $012$, $013$,
              $020$, $021$, $022$, $023$, $024$ \\ \hline
        $4$ & $0000$, $000$1, $0002$, $0010$, $0011$, $0012$, $0013$, $0020$,
              $0021$, $0022$, $0023$, $0024$, $0100$, $0101$, $0102$, $0110$,
              $0111$, $0112$, $0113$, $0120$, $0121$, $0122$, $0123$, $0124$,
              $0130$, $0131$, $0132$, $0133$, $0134$, $0135$, $0200$, $0201$,
              $0202$, $0210$, $0211$, $0212$, $0213$, $0220$, $0221$, $0222$,
              $0223$, $0224$, $0230$, $0231$, $0232$, $0233$, $0234$, $0235$,
              $0240$, $0241$, $0242$, $0243$, $0244$, $0245$, $0246$
    \end{tabular}
\end{center}
\medskip

It is immediate from the definition of~$\FCat{k}$ that for any~$k \geq 0$,
$\FCat{k}$ is a ns suboperad of~$\FCat{k + 1}$. Hence, the ns operads~$\FCat{k}$
form an increasing sequence (for inclusion) of ns operads. Note that~$\FCat{0}$
is isomorphic to the associative commutative operad~$\Com$. Note also that
since~$\FCat{1}$ is generated by~$00$ and~$01$ and since~$\APE$ is generated
by~$01$,~$\APE$ is a ns suboperad of~$\FCat{1}$. Moreover,~$\FCat{0}$ is a
quotient of~$\APE$ by the ns operadic congruence~$\equiv$ defined for
all~$x, y \in \APE(n)$ by~$x \equiv y$.
\medskip

One has the following characterization of the elements of~$\FCat{k}$:
\begin{Proposition} \label{prop:MotsFCat}
    The elements of~$\FCat{k}$ are exactly the words~$x$ on the
    alphabet~$\EnsNat$ satisfying~\mbox{$x_1 = 0$} and
    $0 \leq x_{i + 1} \leq x_i + k$ for all $i \in [|x| - 1]$.
\end{Proposition}
\begin{proof}
    Let us first show by induction on the length of the words that any
    word~$x$ of~$\FCat{k}$ satisfies the statement. This is
    true when~$|x| = 1$. When~$|x| \geq 2$, by
    Lemma~\ref{lem:GenerationOpNSEns}, there is an element~$y$ of~$\FCat{k}$
    of length~$n := |x| - 1$, an integer~$i \in [n]$, and~$0 \leq h \leq k$
    such that $x = y \circ_i 0h$. We have
    \begin{equation}
        x = (y_1, \dots, y_{i - 1}, y_i, y_i + h, y_{i + 1}, \dots, y_n).
    \end{equation}
    Since~$x_{i + 1} = x_i + h$ and~$0 \leq h \leq k$, and since, by
    induction hypothesis,~$y$ satisfies the statement,~$x$ also satisfies it.
    \smallskip

    Let us now show by induction on the length of the words that~$\FCat{k}$
    contains any word~$x$ satisfying the statement. This is
    true when~$|x| = 1$. When~$n := |x| \geq 2$, since~$x_1 = 0$
    and~$0 \leq x_2 \leq k$, $x$ has a factor~$x_i x_{i + 1}$ where~$i$
    is the greatest integer such that~$x_i \leq x_{i + 1}$. Now, by
    setting~$h := x_{i + 1} - x_i$ and
    \begin{equation}
        y := (x_1, \dots, x_i, x_{i + 2}, \dots, x_n),
    \end{equation}
    we have~$x = y \circ_i 0h$. Since~$i$ is maximal, if~$i + 2 \leq n$
    we have~$x_{i + 2} < x_{i + 1}$. This implies that~$y$ satisfies
    the statement. By induction hypothesis,~$\FCat{k}$ contains~$y$
    and, since~$x = y \circ_i 0h$, $\FCat{k}$ also contains~$x$.
\end{proof}
\medskip

A {\em $k$-leafy tree} is a planar rooted tree such that each internal node
has exactly~$k + 1$ children. The size~$|T|$ of a $k$-leafy tree~$T$ is
the number of its internal nodes. It is well-known that there
are~$\frac{1}{kn + 1}\binom{kn + n}{n}$ $k$-leafy trees of size~$n$. We say
that an internal node~$x$ is {\em smaller} than an internal node~$y$ of~$T$
if, in the depth-first traversal of~$T$, $x$ appears before~$y$. We also
say that a $k$-leafy tree $T$ is {\em well-labeled} if its root is labeled
by~$0$, and, for each internal node~$x$ of~$T$ labeled by~$\La$, the children
of~$x$ are labeled, from left to right, by~$\La + k$, \dots, $\La + 1$,
$\La$. There is a unique way to label a $k$-leafy tree so that it is
well-labeled. There is a bijection~$\phi_\FCat{k}$ between the words
of~$\FCat{k}$ of arity~$n$ and well-labeled $k$-leafy trees of size~$n$.
\medskip

To compute~$\phi_\FCat{k}(x)$ where~$x$ is an element of~$\FCat{k}$,
iteratively insert the letters of~$x$ from left to right according to the
following procedure. If~$|x| = 1$, then~$x = 0$ and~$\phi_\FCat{k}(x)$
is the only well-labeled $k$-leafy tree of size~$1$. Otherwise, the insertion
of a letter~$\La \geq 0$ into a well-labeled $k$-leafy tree~$T$ consists
in replacing a leaf of~$T$ by the $k$-leafy tree~$S$ of size~$1$ labeled
by~$\La$ so that~$S$ is the child of the greatest internal node such that
the obtained $k$-tree is still well-labeled.
\medskip

The inverse bijection is computed as follows. Given a well-labeled $k$-leafy
tree~$T$, one computes an element of~$\FCat{k}$ of arity~$n$ by reading
its labels following a depth-first traversal of~$T$.
\medskip

Since the elements of~$\FCat{k}$ satisfy
Proposition~\ref{prop:MotsFCat},~$\phi_\FCat{k}$ is well-defined. Hence,
we can regard the elements of arity~$n$ of~$\FCat{k}$ as $k$-leafy trees of
size~$n$. Figure~\ref{fig:BijFCatKArbres} shows an example of this bijection.
\begin{figure}[ht]
    \centering
    \begin{equation*}
        \begin{split}024021121 \quad \xrightarrow{\phi_\FCat{2}} \quad \end{split}
        \begin{split}\scalebox{.25}{\begin{tikzpicture}
            \node[Feuille](0)at(0.,-4.5){};
            \node[Noeud,EtiqClair](1)at(1.,-3.){$4$};
            \node[Feuille](2)at(1.,-4.5){};
            \node[Feuille](3)at(2.,-4.5){};
            \node[Noeud,EtiqClair](4)at(3.,-1.5){$2$};
            \node[Feuille](5)at(3.,-3.){};
            \node[Feuille](6)at(4.,-3.){};
            \node[Noeud,EtiqClair](7)at(5.,0.){$0$};
            \node[Feuille](8)at(5.,-1.5){};
            \node[Feuille](9)at(6.,-4.5){};
            \node[Noeud,EtiqClair](10)at(7.,-3.){$2$};
            \node[Feuille](11)at(7.,-4.5){};
            \node[Feuille](12)at(8.,-4.5){};
            \node[Noeud,EtiqClair](13)at(9.,-1.5){$0$};
            \node[Feuille](14)at(9.,-4.5){};
            \node[Noeud,EtiqClair](15)at(10.,-3.){$1$};
            \node[Feuille](16)at(10.,-4.5){};
            \node[Feuille](17)at(11.,-6.){};
            \node[Noeud,EtiqClair](18)at(12.,-4.5){$1$};
            \node[Feuille](19)at(12.,-7.5){};
            \node[Noeud,EtiqClair](20)at(13.,-6.){$2$};
            \node[Feuille](21)at(13.,-7.5){};
            \node[Feuille](22)at(14.,-7.5){};
            \node[Feuille](23)at(15.,-7.5){};
            \node[Noeud,EtiqClair](24)at(16.,-6.){$1$};
            \node[Feuille](25)at(16.,-7.5){};
            \node[Feuille](26)at(17.,-7.5){};
            \node[Feuille](27)at(18.,-3.){};
            \draw[Arete](0)--(1);
            \draw[Arete](1)--(4);
            \draw[Arete](2)--(1);
            \draw[Arete](3)--(1);
            \draw[Arete](4)--(7);
            \draw[Arete](5)--(4);
            \draw[Arete](6)--(4);
            \draw[Arete](8)--(7);
            \draw[Arete](9)--(10);
            \draw[Arete](10)--(13);
            \draw[Arete](11)--(10);
            \draw[Arete](12)--(10);
            \draw[Arete](13)--(7);
            \draw[Arete](14)--(15);
            \draw[Arete](15)--(13);
            \draw[Arete](16)--(15);
            \draw[Arete](17)--(18);
            \draw[Arete](18)--(15);
            \draw[Arete](19)--(20);
            \draw[Arete](20)--(18);
            \draw[Arete](21)--(20);
            \draw[Arete](22)--(20);
            \draw[Arete](23)--(24);
            \draw[Arete](24)--(18);
            \draw[Arete](25)--(24);
            \draw[Arete](26)--(24);
            \draw[Arete](27)--(13);
        \end{tikzpicture}}\end{split}
        \quad \longleftrightarrow \quad
        \begin{split}\scalebox{.25}{\begin{tikzpicture}
            \node[Feuille](0)at(0.,-4.5){};
            \node[Noeud](1)at(1.,-3.){};
            \node[Feuille](2)at(1.,-4.5){};
            \node[Feuille](3)at(2.,-4.5){};
            \node[Noeud](4)at(3.,-1.5){};
            \node[Feuille](5)at(3.,-3.){};
            \node[Feuille](6)at(4.,-3.){};
            \node[Noeud](7)at(5.,0.){};
            \node[Feuille](8)at(5.,-1.5){};
            \node[Feuille](9)at(6.,-4.5){};
            \node[Noeud](10)at(7.,-3.){};
            \node[Feuille](11)at(7.,-4.5){};
            \node[Feuille](12)at(8.,-4.5){};
            \node[Noeud](13)at(9.,-1.5){};
            \node[Feuille](14)at(9.,-4.5){};
            \node[Noeud](15)at(10.,-3.){};
            \node[Feuille](16)at(10.,-4.5){};
            \node[Feuille](17)at(11.,-6.){};
            \node[Noeud](18)at(12.,-4.5){};
            \node[Feuille](19)at(12.,-7.5){};
            \node[Noeud](20)at(13.,-6.){};
            \node[Feuille](21)at(13.,-7.5){};
            \node[Feuille](22)at(14.,-7.5){};
            \node[Feuille](23)at(15.,-7.5){};
            \node[Noeud](24)at(16.,-6.){};
            \node[Feuille](25)at(16.,-7.5){};
            \node[Feuille](26)at(17.,-7.5){};
            \node[Feuille](27)at(18.,-3.){};
            \draw[Arete](0)--(1);
            \draw[Arete](1)--(4);
            \draw[Arete](2)--(1);
            \draw[Arete](3)--(1);
            \draw[Arete](4)--(7);
            \draw[Arete](5)--(4);
            \draw[Arete](6)--(4);
            \draw[Arete](8)--(7);
            \draw[Arete](9)--(10);
            \draw[Arete](10)--(13);
            \draw[Arete](11)--(10);
            \draw[Arete](12)--(10);
            \draw[Arete](13)--(7);
            \draw[Arete](14)--(15);
            \draw[Arete](15)--(13);
            \draw[Arete](16)--(15);
            \draw[Arete](17)--(18);
            \draw[Arete](18)--(15);
            \draw[Arete](19)--(20);
            \draw[Arete](20)--(18);
            \draw[Arete](21)--(20);
            \draw[Arete](22)--(20);
            \draw[Arete](23)--(24);
            \draw[Arete](24)--(18);
            \draw[Arete](25)--(24);
            \draw[Arete](26)--(24);
            \draw[Arete](27)--(13);
        \end{tikzpicture}}\end{split}
    \end{equation*}
    \caption{Interpretation of an element of the ns operad~$\FCat{2}$
    in terms of $2$-leafy trees via the bijection~$\phi_\FCat{2}$.
    The $2$-leafy tree in the middle is well-labeled.}
    \label{fig:BijFCatKArbres}
\end{figure}
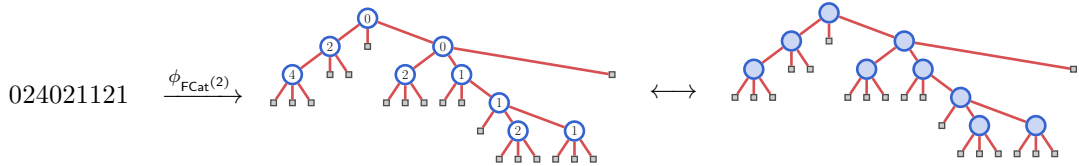
\medskip

The bijection~$\phi_\FCat{k}$ between elements of~$\FCat{k}$ and $k$-leafy
trees offers an alternative way to compute the composition of elements
of~$\FCat{k}$:
\begin{Proposition} \label{prop:SubsFCatK}
    Let~$S$ and~$T$ be two $k$-leafy trees and~$s$ be the $i$th internal
    node for the depth-first traversal of~$S$.
    The composition~$S \circ_i T$ in~$\FCat{k}$ amounts to replace~$s$ by
    the root of~$T$ and graft the children of~$s$ from right to left
    on the rightmost leaves of~$T$.
\end{Proposition}
\begin{proof}
    Let~$x \in \FCat{k}(n)$ and~$y \in \FCat{k}(m)$ such
    that~$S := \phi_{\FCat{k}}(x)$ and~$T := \phi_{\FCat{k}}(y)$.
    Let~$U := \phi_{\FCat{k}}(x \circ_i y)$. By definition of~$\phi_{\FCat{k}}$
    and the partial composition maps of~$\FCat{k}$,~$U$ is obtained by inserting
    the prefix of length~$i - 1$ of~$x$, then the letters of~$y$ incremented
    by~$x_i$, and finally, the suffix of length~$n - i$ of~$x$. Since by
    Proposition~\ref{prop:MotsFCat},~$y$ starts by~$0$, the internal nodes
    created by inserting the letters of~$y$ incremented by~$x_i$ are
    descendants of the internal node created by inserting~$x_i = y_1$.
    Since the last~$n - i$ letters of~$x \circ_i y$ are
    the same as the last~$n - i$ letters of~$x$, by definition
    of~$\phi_{\FCat{k}}$, the children of the $i$th internal node of~$S$ are
    grafted in~$U$ from right to left on the rightmost leaves of~$T$.
    This implies the statement.
\end{proof}
\medskip

Figure~\ref{fig:SubsFCatk} shows an example of composition in~$\FCat{2}$.
\begin{figure}[ht]
    \centering
    \begin{equation*}
        \begin{split}\scalebox{.3}{\begin{tikzpicture}
            \node[Feuille](0)at(0.,-3.){};
            \node[Noeud](1)at(1.,-1.5){};
            \node[Feuille](2)at(1.,-3.){};
            \node[Feuille](3)at(2.,-3.){};
            \node[Noeud,Marque2](4)at(3.,0.){};
            \node[Feuille](5)at(3.,-1.5){};
            \node[Feuille](6)at(4.,-4.5){};
            \node[Noeud](7)at(5.,-3.){};
            \node[Feuille](8)at(5.,-4.5){};
            \node[Feuille](9)at(6.,-4.5){};
            \node[Noeud](10)at(7.,-1.5){};
            \node[Feuille](11)at(7.,-3.){};
            \node[Feuille](12)at(8.,-3.){};
            \draw[Arete](0)--(1);
            \draw[Arete](1)--(4);
            \draw[Arete](2)--(1);
            \draw[Arete](3)--(1);
            \draw[Arete](5)--(4);
            \draw[Arete](6)--(7);
            \draw[Arete](7)--(10);
            \draw[Arete](8)--(7);
            \draw[Arete](9)--(7);
            \draw[Arete](10)--(4);
            \draw[Arete](11)--(10);
            \draw[Arete](12)--(10);
        \end{tikzpicture}}\end{split}
        \enspace \circ_1 \enspace
        \begin{split}\scalebox{.3}{\begin{tikzpicture}
            \node[Feuille](0)at(0.,-3.){};
            \node[Noeud,Marque1](1)at(1.,-1.5){};
            \node[Feuille](2)at(1.,-3.){};
            \node[Feuille](3)at(2.,-3.){};
            \node[Noeud,Marque1](4)at(3.,0.){};
            \node[Feuille](5)at(3.,-3.){};
            \node[Noeud,Marque1](6)at(4.,-1.5){};
            \node[Feuille](7)at(4.,-3.){};
            \node[Feuille](8)at(5.,-3.){};
            \node[Feuille](9)at(6.,-1.5){};
            \draw[Arete](0)--(1);
            \draw[Arete](1)--(4);
            \draw[Arete](2)--(1);
            \draw[Arete](3)--(1);
            \draw[Arete](5)--(6);
            \draw[Arete](6)--(4);
            \draw[Arete](7)--(6);
            \draw[Arete](8)--(6);
            \draw[Arete](9)--(4);
        \end{tikzpicture}}\end{split}
        \enspace = \enspace
        \begin{split}\scalebox{.3}{\begin{tikzpicture}
            \node[Feuille,Marque3](0)at(0.,-3.){};
            \node[Noeud,Marque1](1)at(1.,-1.5){};
            \node[Feuille,Marque3](2)at(1.,-3.){};
            \node[Feuille,Marque3](3)at(2.,-3.){};
            \node[Noeud,Marque1](4)at(3.,0.){};
            \node[Feuille,Marque3](5)at(3.,-3.){};
            \node[Noeud,Marque1](6)at(4.,-1.5){};
            \node[Feuille](7)at(4.,-4.5){};
            \node[Noeud](8)at(5.,-3.){};
            \node[Feuille](9)at(5.,-4.5){};
            \node[Feuille](10)at(6.,-4.5){};
            \node[Feuille](11)at(7.,-3.){};
            \node[Feuille](12)at(8.,-4.5){};
            \node[Noeud](13)at(9.,-3.){};
            \node[Feuille](14)at(9.,-4.5){};
            \node[Feuille](15)at(10.,-4.5){};
            \node[Noeud](16)at(11.,-1.5){};
            \node[Feuille](17)at(11.,-3.){};
            \node[Feuille](18)at(12.,-3.){};
            \draw[Arete](0)--(1);
            \draw[Arete](1)--(4);
            \draw[Arete](2)--(1);
            \draw[Arete](3)--(1);
            \draw[Arete](5)--(6);
            \draw[Arete](6)--(4);
            \draw[Arete](7)--(8);
            \draw[Arete](8)--(6);
            \draw[Arete](9)--(8);
            \draw[Arete](10)--(8);
            \draw[Arete](11)--(6);
            \draw[Arete](12)--(13);
            \draw[Arete](13)--(16);
            \draw[Arete](14)--(13);
            \draw[Arete](15)--(13);
            \draw[Arete](16)--(4);
            \draw[Arete](17)--(16);
            \draw[Arete](18)--(16);
        \end{tikzpicture}}\end{split}
    \end{equation*}
    \caption{Interpretation of the partial composition map of the ns
    operad~$\FCat{2}$ in terms of $2$-leafy trees.}
    \label{fig:SubsFCatk}
\end{figure}
\medskip

The next Theorem elucidates the structure of~$\FCat{k}$:
\begin{Theoreme} \label{thm:PresentationFCatk}
    The ns operad~$\FCat{k}$ admits the presentation
    \begin{equation}
        \FCat{k} = \CalF\left(\left\{\La_0, \dots, \La_k\right\}\right)/_\equiv,
    \end{equation}
    where the $\La_i$ are of arity~$2$ and~$\equiv$ is the ns operadic
    congruence generated by
    \begin{equation}
        \La_{i + j} \circ_1 \La_i \: \leftrightarrow \: \La_i \circ_2 \La_j,
        \qquad i, j \geq 0, i + j \leq k.
    \end{equation}
\end{Theoreme}
\begin{proof}
    First, note that by replacing~$\La_i$ by $0i \in \FCat{k}(2)$, we have
    $\Eval(x) = \Eval(y)$ for the relation $x \leftrightarrow y$ of the
    statement of the Theorem. Indeed, this equivalence class is the one
    of the element~$(0, i, i + j)$ of~$\FCat{k}$.
    \smallskip

    Consider now the orientation of~$\leftrightarrow$ into the rewrite
    rule~$\mapsto$ defined by
    \begin{equation} \label{eq:RelOrientationFCatk}
        \begin{split}\scalebox{.35}{\begin{tikzpicture}
            \node[Feuille](0)at(0.0,-2){};
            \node[Noeud,EtiqClair,minimum size=15mm](1)at(1.0,-1){$\La_i$};
            \node[Feuille](2)at(2.0,-2){};
            \draw[Arete](1)--(0);
            \draw[Arete](1)--(2);
            \node[Noeud,EtiqClair,minimum size=15mm](3)at(3.0,0){$\La_{i + j}$};
            \node[Feuille](4)at(4.0,-1){};
            \draw[Arete](3)--(1);
            \draw[Arete](3)--(4);
        \end{tikzpicture}}\end{split}
        \begin{split}\enspace \mapsto \enspace \end{split}
        \begin{split}\scalebox{.35}{\begin{tikzpicture}
            \node[Feuille](0)at(0.0,-1){};
            \node[Noeud,EtiqClair,minimum size=15mm](1)at(1.0,0){$\La_i$};
            \node[Feuille](2)at(2.0,-2){};
            \node[Noeud,EtiqClair,minimum size=15mm](3)at(3.0,-1){$\La_j$};
            \node[Feuille](4)at(4.0,-2){};
            \draw[Arete](3)--(2);
            \draw[Arete](3)--(4);
            \draw[Arete](1)--(0);
            \draw[Arete](1)--(3);
        \end{tikzpicture}}\end{split}.
    \end{equation}

    This rewrite rule is terminating. Indeed, it is plain that for any
    rewriting~$T_0 \to T_1$, we have~$\Poids(T_0) < \Poids(T_1)$.
    \smallskip

    Moreover, the normal forms of~$\mapsto$ are all elements
    of~$\CalF\left(\left\{\La_0, \dots, \La_k\right\}\right)$ such that
    for each node~$x$ labeled by~$\La_i$ which has a left child~$y$ labeled
    by~$\La_j$, one has~$i < j$. This set~$\CalS$ of syntax trees admits
    the following regular specification
    \begin{equation}
        \CalS = \Feuille \enspace + \enspace
            \bigsqcup_{0 \leq i \leq k} \CalS_i,
    \end{equation}
    where~$\CalS_i$ is the set of such syntax trees with roots labeled
    by $\La_i$. These sets satisfy the following regular specification
    \begin{equation}
        \begin{split} \CalS_i = \end{split}
        \begin{split}\scalebox{.3}{\begin{tikzpicture}
            \node[Noeud,EtiqClair,minimum size=15mm](1)at(0,0){$\La_i$};
            \node[Feuille](2)at(-1.5,-1.5){};
            \node(3)at(1.5,-1.5){\Huge $\CalS$};
            \draw[Arete](1)--(2);
            \draw[Arete](1)--(3);
        \end{tikzpicture}}\end{split}
        \enspace + \enspace
        \bigsqcup_{i + 1 \leq j \leq k}
        \begin{split}\scalebox{.3}{\begin{tikzpicture}
            \node[Noeud,EtiqClair,minimum size=15mm](1)at(0,0){$\La_i$};
            \node(2)at(-1.5,-1.5){\Huge $\CalS_j$};
            \node(3)at(1.5,-1.5){\Huge $\CalS$};
            \draw[Arete](1)--(2);
            \draw[Arete](1)--(3);
        \end{tikzpicture}}\end{split}.
    \end{equation}
    Hence, the generating series $F(t)$ of $\CalS$ and $F_i(t)$ of $\CalS_i$
    satisfy
    \begin{equation}
        F(t) = t + \sum_{0 \leq i \leq k} F_i(t),
    \end{equation}
    and
    \begin{equation}
        F_i(t) = t F(t) + F(t) \sum_{i + 1 \leq j \leq k} F_j(t).
    \end{equation}
    By basic manipulations involving binomial coefficients, we obtain
    \begin{equation}
        F_i(t) = t F(t) \sum_{0 \leq j \leq k - i} \binom{k - i}{j} {F(t)}^j,
    \end{equation}
    and then,
    \begin{equation} \label{eq:EqFoncKArbres}
        F(t) = t \sum_{0 \leq j \leq k + 1} \binom{k + 1}{j} {F(t)}^j.
    \end{equation}
    The functional equation~\eqref{eq:EqFoncKArbres} is an alternative
    functional equation for the generating series of $k$-leafy trees.
    By Proposition~\ref{prop:MotsFCat}, $F(t)$ also is the Hilbert series
    of~$\FCat{k}$.
    \smallskip

    Hence, by Lemma~\ref{lem:PresentationReecriture},~$\FCat{k}$ admits
    the claimed presentation.
\end{proof}
\medskip

\subsubsection{A ns operad on Schröder trees}
Let~$\Schr$ be the ns suboperad of~$\T \EnsNat$ generated
by~$00$, $01$, and~$10$. The following table shows the first elements
of~$\Schr$.
\begin{center}
    \begin{tabular}{c|p{11cm}}
        Arity & Elements of~$\Schr$ \\ \hline \hline
        $1$ & $0$ \\ \hline
        $2$ & $00$, $01$, $10$ \\ \hline
        $3$ & $000$, $001$, $010$, $011$, $012$, $021$, $100$, $101$,
              $110$, $120$, $210$ \\ \hline
        $4$ & $0000$, $0001$, $0010$, $0011$, $0012$, $0021$, $0100$,
              $0101$, $0110$, $0111$, $0112$, $0120$, $0121$, $0122$,
              $0123$, $0132$, $0210$, $0211$, $0212$, $0221$, $0231$,
              $0321$, $1000$, $1001$, $1010$, $1011$, $1012$, $1021$,
              $1100$, $1101$, $1110$, $1120$, $1200$, $1201$, $1210$,
              $1220$, $1230$, $1320$, $2100$, $2101$, $2110$, $2120$,
              $2210$, $2310$, $3210$
    \end{tabular}
\end{center}
\medskip

Since~$\FCat{1}$ is generated by~$00$ and~$01$,~$\FCat{1}$ is a ns suboperad
of~$\Schr$. Moreover, since~$\MT$ is, by Proposition~\ref{prop:GenerationMT},
generated as an operad by~$00$ and~$01$,~$\Schr$ is a ns suboperad
of~$\MT$.
\medskip

One has the following characterization of the elements of~$\Schr$:
\begin{Proposition} \label{prop:MotsSchr}
    The elements of~$\Schr$ are exactly the words~$x$ on the alphabet~$\EnsNat$
    having at least one occurrence of~$0$ and, for all letter~$\Lb \geq 1$
    of~$x$, there exists a letter~$\La := \Lb - 1$ such that~$x$ has a
    factor~$\La u \Lb$ or~$\Lb u \La$ where~$u$ is a word consisting in
    letters~$\Lc$ satisfying~$\Lc \geq \Lb$.
\end{Proposition}
\begin{proof}
    Let us first show by induction on the length of the words that any
    word~$x$ of~$\Schr$ satisfies the statement. This is true
    when~$|x| = 1$. When~$|x| \geq 2$, by Lemma~\ref{lem:GenerationOpNSEns},
    there is a element~$y$ of~$\Schr$ of length~$n := |x| - 1$, an integer~$i \in [n]$,
    and $g \in \{00, 01, 10\}$ such that~$x = y \circ_i g$. Then,
    \begin{equation}
        x = (y_1, \dots,
        y_{i - 1}, y_i + g_1, y_i + g_2, y_{i + 1}, \dots, y_{n - 1}).
    \end{equation}
    Since by induction hypothesis~$y$ satisfies the statement and~$x_i = y_i$
    or~$x_{i + 1} = y_i$,~$x$ also satisfies the statement.
    \smallskip

    Let us now show by induction on the length of the words that~$\Schr$
    contains any word~$x$ satisfying the statement. This is
    true when~$|x| = 1$. When~$n := |x| \geq 2$, let us observe that
    if~$x$ only consists in letters~$0$,~$\Schr$ contains~$x$ because~$x$
    can be obtained by composing the generator~$00$ with itself. Hence,
    let us assume that~$x$ contains at least one letter different from~$0$.
    Set~$\Lb$ as the greatest letter of~$x$ and~$\La := \Lb - 1$. Since~$\Lb$
    is the greatest letter of~$x$, there is a factor~$x_i x_{i + 1}$ of~$x$
    such that $x_i x_{i + 1} \in \{\Lb\Lb, \Lb\La, \La\Lb\}$. Set
    \begin{equation}
        y := (x_1, \dots, x_{i - 1}, \min \{x_i, x_{i + 1}\},
        x_{i + 2}, \dots, x_n),
    \end{equation}
    and~$g$ as the generator~$00$ if~$x_i = x_{i + 1}$, as~$01$ if
    $x_i = x_{i + 1} - 1$, or as~$10$ when $x_i = x_{i + 1} + 1$. Then,
    we have~$x = y \circ_i g$, and, since~$y$ is obtained from~$x$ by
    removing one of its greatest letter,~$y$ satisfies the statement.
    By induction hypothesis,~$\Schr$ contains~$y$, and since~$x = y \circ_i g$,
    $\Schr$ also contains~$x$.
\end{proof}
\medskip

A {\em Schröder tree} is a planar rooted tree such that no node has
exactly one child. The size~$|T|$ of a Schröder tree~$T$ is its number of
leaves. There is a bijection~$\phi_\Schr$ between the words of~$\Schr$
of arity~$n$ and Schröder trees of size~$n$.
\medskip

To compute~$\phi_\Schr(x)$ where~$x$ is an element of~$\Schr$, factorize~$x$
as~\mbox{$x = x^{(1)} \La \dots \La x^{(\ell)}$} where~$\La$ is the smallest
letter occurring in~$x$ and the $x^{(i)}$ are factors of~$x$ without~$\La$.
Then, set
\begin{equation}
    \phi_\Schr(x) :=
    \begin{cases}
        \Feuille & \mbox{if $x = \epsilon$}, \\
        \ArbCons\left(\phi_\Schr\left(x^{(1)}\right), \dots,
        \phi_\Schr\left(x^{(\ell)}\right)\right) & \mbox{otherwise},
    \end{cases}
\end{equation}
where~$\epsilon$ denotes the empty word and $\ArbCons(T_1, \dots, T_\ell)$ is
the Schröder tree consisting in a root that has~$T_1, \dots, T_\ell$ as subtrees
from left to right.
\medskip

The inverse bijection is computed as follows. Given a Schröder tree~$T$,
one computes an element of~$\Schr$ by considering each internal node~$s$
and two adjacent consecutive edges of~$s$ and by assigning to these the
depth of~$s$. The element of~$\Schr$ is obtained by reading the labels
from left to right.
\medskip

Since the elements of~$\Schr$ satisfy
Proposition~\ref{prop:MotsSchr},~$\phi_\Schr$ is well-defined.
Figure~\ref{fig:BijSchrMots} shows an example of this bijection.
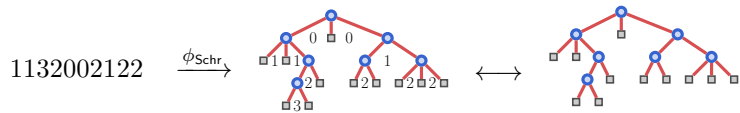
\begin{figure}[ht]
    \centering
    \begin{equation*}
        \begin{split}
            1132002122
            \quad \xrightarrow{\phi_\Schr} \quad
        \end{split}
        \begin{split}\scalebox{.3}{\begin{tikzpicture}
            \node[NoeudSr](n1)at(0,0){};
            \node[NoeudSr](n2)at(-2,-1){};
            \node[NoeudSr](n3)at(-1,-2){};
            \node[NoeudSr](n4)at(-1.5,-3){};
            \node[NoeudSr](n5)at(2.5,-1){};
            \node[NoeudSr](n6)at(1.5,-2){};
            \node[NoeudSr](n7)at(4,-2){};
            \node[Feuille](f1)at(-3,-2){};
            \node[Feuille](f2)at(-2,-2){};
            \node[Feuille](f3)at(-2,-4){};
            \node[Feuille](f4)at(-1,-4){};
            \node[Feuille](f5)at(-.5,-3){};
            \node[Feuille](f6)at(0,-1){};
            \node[Feuille](f7)at(1,-3){};
            \node[Feuille](f8)at(2,-3){};
            \node[Feuille](f9)at(3,-3){};
            \node[Feuille](f10)at(4,-3){};
            \node[Feuille](f11)at(5,-3){};
            \node[left of=f2,node distance=.5cm]{\Huge $1$};
            \node[right of=f2,node distance=.5cm]{\Huge $1$};
            \node[right of=f3,node distance=.5cm]{\Huge $3$};
            \node[left of=f5,node distance=.5cm]{\Huge $2$};
            \node[left of=f6,node distance=.8cm]{\Huge $0$};
            \node[right of=f6,node distance=.8cm]{\Huge $0$};
            \node[right of=f7,node distance=.5cm]{\Huge $2$};
            \node[right of=n6,node distance=1cm]{\Huge $1$};
            \node[right of=f9,node distance=.5cm]{\Huge $2$};
            \node[right of=f10,node distance=.5cm]{\Huge $2$};
            \draw[Arete](n1)--(n2);
            \draw[Arete](n2)--(n3);
            \draw[Arete](n3)--(n4);
            \draw[Arete](n1)--(n5);
            \draw[Arete](n5)--(n6);
            \draw[Arete](n5)--(n7);
            \draw[Arete](n1)--(f6);
            \draw[Arete](n2)--(f1);
            \draw[Arete](n2)--(f2);
            \draw[Arete](n3)--(f5);
            \draw[Arete](n4)--(f3);
            \draw[Arete](n4)--(f4);
            \draw[Arete](n6)--(f7);
            \draw[Arete](n6)--(f8);
            \draw[Arete](n7)--(f9);
            \draw[Arete](n7)--(f10);
            \draw[Arete](n7)--(f11);
        \end{tikzpicture}}
        \end{split}
        \begin{split} \quad \longleftrightarrow \quad \end{split}
        \begin{split}\scalebox{.3}{\begin{tikzpicture}
            \node[NoeudSr](n1)at(0,0){};
            \node[NoeudSr](n2)at(-2,-1){};
            \node[NoeudSr](n3)at(-1,-2){};
            \node[NoeudSr](n4)at(-1.5,-3){};
            \node[NoeudSr](n5)at(2.5,-1){};
            \node[NoeudSr](n6)at(1.5,-2){};
            \node[NoeudSr](n7)at(4,-2){};
            \node[Feuille](f1)at(-3,-2){};
            \node[Feuille](f2)at(-2,-2){};
            \node[Feuille](f3)at(-2,-4){};
            \node[Feuille](f4)at(-1,-4){};
            \node[Feuille](f5)at(-.5,-3){};
            \node[Feuille](f6)at(0,-1){};
            \node[Feuille](f7)at(1,-3){};
            \node[Feuille](f8)at(2,-3){};
            \node[Feuille](f9)at(3,-3){};
            \node[Feuille](f10)at(4,-3){};
            \node[Feuille](f11)at(5,-3){};
            \draw[Arete](n1)--(n2);
            \draw[Arete](n2)--(n3);
            \draw[Arete](n3)--(n4);
            \draw[Arete](n1)--(n5);
            \draw[Arete](n5)--(n6);
            \draw[Arete](n5)--(n7);
            \draw[Arete](n1)--(f6);
            \draw[Arete](n2)--(f1);
            \draw[Arete](n2)--(f2);
            \draw[Arete](n3)--(f5);
            \draw[Arete](n4)--(f3);
            \draw[Arete](n4)--(f4);
            \draw[Arete](n6)--(f7);
            \draw[Arete](n6)--(f8);
            \draw[Arete](n7)--(f9);
            \draw[Arete](n7)--(f10);
            \draw[Arete](n7)--(f11);
        \end{tikzpicture}}
        \end{split}
    \end{equation*}
    \caption{Interpretation of an element of the ns operad~$\Schr$
    in terms of Schröder trees via the bijection~$\phi_\Schr$.}
    \label{fig:BijSchrMots}
\end{figure}
\medskip

\begin{Theoreme} \label{thm:PresentationSchr}
    The ns operad~$\Schr$ admits the presentation
    \begin{equation}
        \Schr =
        \CalF\left(\left\{\SchrOpA, \SchrOpB, \SchrOpC\right\}\right)/_\equiv,
    \end{equation}
    where $\SchrOpA$, $\SchrOpB$, and $\SchrOpC$ are of arity~$2$,
    and~$\equiv$ is the ns operadic congruence generated by
    \begin{minipage}{.45\textwidth}
    \begin{equation}
        \SchrOpA \circ_1 \SchrOpA
            \: \leftrightarrow \: \SchrOpA \circ_2 \SchrOpA,
    \end{equation}
    \begin{equation}
        \SchrOpB \circ_1 \SchrOpC
            \: \leftrightarrow \: \SchrOpC \circ_2 \SchrOpB,
    \end{equation}
    \begin{equation}
        \SchrOpA \circ_1 \SchrOpB
            \: \leftrightarrow \: \SchrOpA \circ_2 \SchrOpC,
    \end{equation}
    \begin{equation}
        \SchrOpB \circ_1 \SchrOpA
            \: \leftrightarrow \: \SchrOpA \circ_2 \SchrOpB,
    \end{equation}
    \end{minipage}
    \begin{minipage}{.45\textwidth}
    \begin{equation}
        \SchrOpA \circ_1 \SchrOpC
            \: \leftrightarrow \: \SchrOpC \circ_2 \SchrOpA,
    \end{equation}
    \begin{equation}
        \SchrOpB \circ_1 \SchrOpB
            \: \leftrightarrow \: \SchrOpB \circ_2 \SchrOpA,
    \end{equation}
    \begin{equation}
        \SchrOpC \circ_1 \SchrOpA
            \: \leftrightarrow \: \SchrOpC \circ_2 \SchrOpC.
    \end{equation}
    \end{minipage}
\end{Theoreme}
\begin{proof}
    First, note that by replacing~$\SchrOpA$ by~$00 \in \Schr(2)$,
    $\SchrOpB$ by~$01 \in \Schr(2)$, and~$\SchrOpC$ by~$10 \in \Schr(2)$,
    we have~$\Eval(x) = \Eval(y)$ for the seven relations~$x \leftrightarrow y$
    of the statement of the Theorem. Indeed, then seven equivalence classes
    are, respectively, the ones of the elements~$000$, $101$, $010$,
    $001$, $100$, $011$, and~$110$ of~$\Schr$.
    \smallskip

    Consider now the orientation of~$\leftrightarrow$ into the rewrite
    rule~$\mapsto$ defined by

    \begin{minipage}{.45\textwidth}
    \begin{equation}
        \begin{split}\scalebox{.35}{\begin{tikzpicture}
            \node[Feuille](0)at(0.0,-2){};
            \node[Noeud,EtiqClair](1)at(1.0,-1){\scalebox{2}{\SchrOpA}};
            \node[Feuille](2)at(2.0,-2){};
            \draw[Arete](1)--(0);
            \draw[Arete](1)--(2);
            \node[Noeud,EtiqClair](3)at(3.0,0){\scalebox{2}{\SchrOpA}};
            \node[Feuille](4)at(4.0,-1){};
            \draw[Arete](3)--(1);
            \draw[Arete](3)--(4);
        \end{tikzpicture}}\end{split}
        \begin{split}\enspace \mapsto \enspace \end{split}
        \begin{split}\scalebox{.35}{\begin{tikzpicture}
            \node[Feuille](0)at(0.0,-1){};
            \node[Noeud,EtiqClair](1)at(1.0,0){\scalebox{2}{\SchrOpA}};
            \node[Feuille](2)at(2.0,-2){};
            \node[Noeud,EtiqClair](3)at(3.0,-1){\scalebox{2}{\SchrOpA}};
            \node[Feuille](4)at(4.0,-2){};
            \draw[Arete](3)--(2);
            \draw[Arete](3)--(4);
            \draw[Arete](1)--(0);
            \draw[Arete](1)--(3);
        \end{tikzpicture}}\end{split},
    \end{equation}
    \begin{equation}
        \begin{split}\scalebox{.35}{\begin{tikzpicture}
            \node[Feuille](0)at(0.0,-2){};
            \node[Noeud,EtiqClair](1)at(1.0,-1){\scalebox{2}{\SchrOpC}};
            \node[Feuille](2)at(2.0,-2){};
            \draw[Arete](1)--(0);
            \draw[Arete](1)--(2);
            \node[Noeud,EtiqClair](3)at(3.0,0){\scalebox{2}{\SchrOpB}};
            \node[Feuille](4)at(4.0,-1){};
            \draw[Arete](3)--(1);
            \draw[Arete](3)--(4);
        \end{tikzpicture}}\end{split}
        \begin{split}\enspace \mapsto \enspace \end{split}
        \begin{split}\scalebox{.35}{\begin{tikzpicture}
            \node[Feuille](0)at(0.0,-1){};
            \node[Noeud,EtiqClair](1)at(1.0,0){\scalebox{2}{\SchrOpC}};
            \node[Feuille](2)at(2.0,-2){};
            \node[Noeud,EtiqClair](3)at(3.0,-1){\scalebox{2}{\SchrOpB}};
            \node[Feuille](4)at(4.0,-2){};
            \draw[Arete](3)--(2);
            \draw[Arete](3)--(4);
            \draw[Arete](1)--(0);
            \draw[Arete](1)--(3);
        \end{tikzpicture}}\end{split},
    \end{equation}
    \begin{equation}
        \begin{split}\scalebox{.35}{\begin{tikzpicture}
            \node[Feuille](0)at(0.0,-2){};
            \node[Noeud,EtiqClair](1)at(1.0,-1){\scalebox{2}{\SchrOpB}};
            \node[Feuille](2)at(2.0,-2){};
            \draw[Arete](1)--(0);
            \draw[Arete](1)--(2);
            \node[Noeud,EtiqClair](3)at(3.0,0){\scalebox{2}{\SchrOpA}};
            \node[Feuille](4)at(4.0,-1){};
            \draw[Arete](3)--(1);
            \draw[Arete](3)--(4);
        \end{tikzpicture}}\end{split}
        \begin{split}\enspace \mapsto \enspace \end{split}
        \begin{split}\scalebox{.35}{\begin{tikzpicture}
            \node[Feuille](0)at(0.0,-1){};
            \node[Noeud,EtiqClair](1)at(1.0,0){\scalebox{2}{\SchrOpA}};
            \node[Feuille](2)at(2.0,-2){};
            \node[Noeud,EtiqClair](3)at(3.0,-1){\scalebox{2}{\SchrOpC}};
            \node[Feuille](4)at(4.0,-2){};
            \draw[Arete](3)--(2);
            \draw[Arete](3)--(4);
            \draw[Arete](1)--(0);
            \draw[Arete](1)--(3);
        \end{tikzpicture}}\end{split},
    \end{equation}
    \begin{equation}
        \begin{split}\scalebox{.35}{\begin{tikzpicture}
            \node[Feuille](0)at(0.0,-2){};
            \node[Noeud,EtiqClair](1)at(1.0,-1){\scalebox{2}{\SchrOpA}};
            \node[Feuille](2)at(2.0,-2){};
            \draw[Arete](1)--(0);
            \draw[Arete](1)--(2);
            \node[Noeud,EtiqClair](3)at(3.0,0){\scalebox{2}{\SchrOpB}};
            \node[Feuille](4)at(4.0,-1){};
            \draw[Arete](3)--(1);
            \draw[Arete](3)--(4);
        \end{tikzpicture}}\end{split}
        \begin{split}\enspace \mapsto \enspace \end{split}
        \begin{split}\scalebox{.35}{\begin{tikzpicture}
            \node[Feuille](0)at(0.0,-1){};
            \node[Noeud,EtiqClair](1)at(1.0,0){\scalebox{2}{\SchrOpA}};
            \node[Feuille](2)at(2.0,-2){};
            \node[Noeud,EtiqClair](3)at(3.0,-1){\scalebox{2}{\SchrOpB}};
            \node[Feuille](4)at(4.0,-2){};
            \draw[Arete](3)--(2);
            \draw[Arete](3)--(4);
            \draw[Arete](1)--(0);
            \draw[Arete](1)--(3);
        \end{tikzpicture}}\end{split},
    \end{equation}
    \end{minipage}
    \begin{minipage}{.45\textwidth}
    \begin{equation}
        \begin{split}\scalebox{.35}{\begin{tikzpicture}
            \node[Feuille](0)at(0.0,-2){};
            \node[Noeud,EtiqClair](1)at(1.0,-1){\scalebox{2}{\SchrOpC}};
            \node[Feuille](2)at(2.0,-2){};
            \draw[Arete](1)--(0);
            \draw[Arete](1)--(2);
            \node[Noeud,EtiqClair](3)at(3.0,0){\scalebox{2}{\SchrOpA}};
            \node[Feuille](4)at(4.0,-1){};
            \draw[Arete](3)--(1);
            \draw[Arete](3)--(4);
        \end{tikzpicture}}\end{split}
        \begin{split}\enspace \mapsto \enspace \end{split}
        \begin{split}\scalebox{.35}{\begin{tikzpicture}
            \node[Feuille](0)at(0.0,-1){};
            \node[Noeud,EtiqClair](1)at(1.0,0){\scalebox{2}{\SchrOpC}};
            \node[Feuille](2)at(2.0,-2){};
            \node[Noeud,EtiqClair](3)at(3.0,-1){\scalebox{2}{\SchrOpA}};
            \node[Feuille](4)at(4.0,-2){};
            \draw[Arete](3)--(2);
            \draw[Arete](3)--(4);
            \draw[Arete](1)--(0);
            \draw[Arete](1)--(3);
        \end{tikzpicture}}\end{split},
    \end{equation}
    \begin{equation}
        \begin{split}\scalebox{.35}{\begin{tikzpicture}
            \node[Feuille](0)at(0.0,-2){};
            \node[Noeud,EtiqClair](1)at(1.0,-1){\scalebox{2}{\SchrOpB}};
            \node[Feuille](2)at(2.0,-2){};
            \draw[Arete](1)--(0);
            \draw[Arete](1)--(2);
            \node[Noeud,EtiqClair](3)at(3.0,0){\scalebox{2}{\SchrOpB}};
            \node[Feuille](4)at(4.0,-1){};
            \draw[Arete](3)--(1);
            \draw[Arete](3)--(4);
        \end{tikzpicture}}\end{split}
        \begin{split}\enspace \mapsto \enspace \end{split}
        \begin{split}\scalebox{.35}{\begin{tikzpicture}
            \node[Feuille](0)at(0.0,-1){};
            \node[Noeud,EtiqClair](1)at(1.0,0){\scalebox{2}{\SchrOpB}};
            \node[Feuille](2)at(2.0,-2){};
            \node[Noeud,EtiqClair](3)at(3.0,-1){\scalebox{2}{\SchrOpA}};
            \node[Feuille](4)at(4.0,-2){};
            \draw[Arete](3)--(2);
            \draw[Arete](3)--(4);
            \draw[Arete](1)--(0);
            \draw[Arete](1)--(3);
        \end{tikzpicture}}\end{split},
    \end{equation}
    \begin{equation}
        \begin{split}\scalebox{.35}{\begin{tikzpicture}
            \node[Feuille](0)at(0.0,-1){};
            \node[Noeud,EtiqClair](1)at(1.0,0){\scalebox{2}{\SchrOpC}};
            \node[Feuille](2)at(2.0,-2){};
            \node[Noeud,EtiqClair](3)at(3.0,-1){\scalebox{2}{\SchrOpC}};
            \node[Feuille](4)at(4.0,-2){};
            \draw[Arete](3)--(2);
            \draw[Arete](3)--(4);
            \draw[Arete](1)--(0);
            \draw[Arete](1)--(3);
        \end{tikzpicture}}\end{split}
        \begin{split}\enspace \mapsto \enspace \end{split}
        \begin{split}\scalebox{.35}{\begin{tikzpicture}
            \node[Feuille](0)at(0.0,-2){};
            \node[Noeud,EtiqClair](1)at(1.0,-1){\scalebox{2}{\SchrOpA}};
            \node[Feuille](2)at(2.0,-2){};
            \draw[Arete](1)--(0);
            \draw[Arete](1)--(2);
            \node[Noeud,EtiqClair](3)at(3.0,0){\scalebox{2}{\SchrOpC}};
            \node[Feuille](4)at(4.0,-1){};
            \draw[Arete](3)--(1);
            \draw[Arete](3)--(4);
        \end{tikzpicture}}\end{split}.
    \end{equation}
    \end{minipage}

    This rewrite rule is terminating. Indeed, let~$T$ be an element
    of~$\CalF\left(\left\{\SchrOpA, \SchrOpB, \SchrOpC\right\}\right)$.
    By associating with~$T$ the pair~$(k_T, \Poids(T))$ where~$k_T$ is the
    number of nodes labeled by~$\SchrOpA$ in~$T$, it is plain that for any
    rewriting~$T_0 \to T_1$, one has $k_{T_0} < k_{T_1}$
    or $k_{T_0} = k_{T_1}$ and $\Poids(T_0) < \Poids(T_1)$.
    \smallskip

    Moreover, the normal forms of~$\mapsto$ are all elements
    of~$\CalF\left(\left\{\SchrOpA, \SchrOpB, \SchrOpC\right\}\right)$
    such that nodes labeled by~$\SchrOpA$ or~$\SchrOpB$ have no left
    child and nodes labeled by~$\SchrOpC$ have no right child labeled
    by~$\SchrOpC$. This set~$\CalS$ of syntax trees admits the following
    regular specification
    \begin{equation}\begin{split}
        \begin{split} \CalS = \Feuille \end{split}
        \enspace + \enspace
        \begin{split}\scalebox{.3}{\begin{tikzpicture}
            \node[Noeud,EtiqClair,minimum size=15mm](1)at(0,0)
                {\scalebox{2}{\SchrOpA}};
            \node[Feuille](2)at(-1.5,-1.5){};
            \node(3)at(1.5,-1.5){\Huge $\CalS$};
            \draw[Arete](1)--(2);
            \draw[Arete](1)--(3);
        \end{tikzpicture}}\end{split}
        \enspace + \enspace
        \begin{split}\scalebox{.3}{\begin{tikzpicture}
            \node[Noeud,EtiqClair,minimum size=15mm](1)at(0,0)
                {\scalebox{2}{\SchrOpB}};
            \node[Feuille](2)at(-1.5,-1.5){};
            \node(3)at(1.5,-1.5){\Huge $\CalS$};
            \draw[Arete](1)--(2);
            \draw[Arete](1)--(3);
        \end{tikzpicture}}\end{split}
        \enspace + \enspace
        \begin{split}\scalebox{.3}{\begin{tikzpicture}
            \node[Noeud,EtiqClair,minimum size=15mm](1)at(0,0)
                {\scalebox{2}{\SchrOpC}};
            \node[Feuille](2)at(1.5,-1.5){};
            \node(3)at(-1.5,-1.5){\Huge $\CalS$};
            \draw[Arete](1)--(2);
            \draw[Arete](1)--(3);
        \end{tikzpicture}}\end{split}
        \enspace + \enspace
        \begin{split}\scalebox{.3}{\begin{tikzpicture}
            \node[Noeud,EtiqClair,minimum size=15mm](1)at(0,0)
                {\scalebox{2}{\SchrOpC}};
            \node(2)at(-1.5,-1.5){\Huge $\CalS$};
            \node[Noeud,EtiqClair,minimum size=15mm](3)at(2,-2)
                {\scalebox{2}{\SchrOpA}};
            \node[Feuille](4)at(1,-3.5){};
            \node(5)at(3,-3.5){\Huge $\CalS$};
            \draw[Arete](1)--(2);
            \draw[Arete](1)--(3);
            \draw[Arete](3)--(4);
            \draw[Arete](3)--(5);
        \end{tikzpicture}}\end{split}
        \enspace + \enspace
        \begin{split}\scalebox{.3}{\begin{tikzpicture}
            \node[Noeud,EtiqClair,minimum size=15mm](1)at(0,0)
                {\scalebox{2}{\SchrOpC}};
            \node(2)at(-1.5,-1.5){\Huge $\CalS$};
            \node[Noeud,EtiqClair,minimum size=15mm](3)at(2,-2)
                {\scalebox{2}{\SchrOpB}};
            \node[Feuille](4)at(1,-3.5){};
            \node(5)at(3,-3.5){\Huge $\CalS$};
            \draw[Arete](1)--(2);
            \draw[Arete](1)--(3);
            \draw[Arete](3)--(4);
            \draw[Arete](3)--(5);
        \end{tikzpicture}}\end{split}.
    \end{split}\end{equation}
    Hence, the generating series~$F(t)$ of~$\CalS$ satisfies
    \begin{equation}
        F(t) = t + 3 t F(t) + 2 t {F(t)}^2,
    \end{equation}
    that is the generating series of Schröder trees. By
    Proposition~\ref{prop:MotsSchr},~$F(t)$ also is the Hilbert series
    of~$\Schr$.
    \smallskip

    Hence, by Lemma~\ref{lem:PresentationReecriture},~$\Schr$ admits
    the claimed presentation.
\end{proof}
\medskip

\subsubsection{A ns operad on Motzkin words}
Let~$\Motz$ be the ns suboperad of~$\T \EnsNat$ generated
by~$00$ and~$010$. The following table shows the first elements of~$\Motz$.
\begin{center}
    \begin{tabular}{c|p{11cm}}
        Arity & Elements of~$\Motz$ \\ \hline \hline
        $1$ & $0$ \\ \hline
        $2$ & $00$ \\ \hline
        $3$ & $000$, $010$ \\ \hline
        $4$ & $0000$, $0010$, $0100$, $0110$ \\ \hline
        $5$ & $00000$, $00010$, $00100$, $00110$, $01000$, $01010$,
              $01100$, $01110$, $01210$ \\ \hline
        $6$ & $000000$, $000010$, $000100$, $000110$, $001000$, $001010$,
              $001100$, $001110$, $001210$, $010000$, $010010$, $010100$,
              $010110$, $011000$, $011010$, $011100$, $011110$, $011210$,
              $012100$, $012110$, $012210$
    \end{tabular}
\end{center}
\medskip

Since~$00$ and~$01$ generate~$\FCat{1}$ and since~$010 = 00 \circ_1 01$,
$\Motz$ is a ns suboperad of~$\FCat{1}$. Moreover, since~$\FCat{0}$ is
generated by~$00$, $\FCat{0}$ is a ns suboperad of~$\Motz$.
\medskip

One has the following characterization of the elements of~$\Motz$:
\begin{Proposition} \label{prop:ElemMotz}
    The elements of~$\Motz$ are exactly the words~$x$ on the
    alphabet~$\EnsNat$ beginning and start by~$0$ and such
    that~$|x_i - x_{i + 1}| \leq 1$ for all~$i \in [|x| - 1]$.
\end{Proposition}
\begin{proof}
    Let us first show by induction on the length of the words that any
    word~$x$ of~$\Motz$ satisfies the statement. This is true
    when~$|x| = 1$. When~$|x| \geq 2$, by
    Lemma~\ref{lem:GenerationOpNSEns}, there is an element~$y$
    of~$\Motz$ of length~$n < |x|$, an integer~$i \in [n]$, and~$g \in \{00, 010\}$
    such that~$x = y \circ_i g$. If~$g = 00$, then one has
    \begin{equation}
        x = (y_1, \dots, y_{i - 1}, y_i, y_i, y_{i + 1}, \dots, y_n).
    \end{equation}
    Otherwise, we have~$g = 010$ and
    \begin{equation}
        x = (y_1, \dots, y_{i - 1}, y_i, y_i + 1, y_i, y_{i + 1},
        \dots, y_n).
    \end{equation}
    Since by induction hypothesis~$y$ satisfies the statement, it is
    immediate that in both cases,~$x$ also satisfies the statement.
    \smallskip

    Let us now show by induction on the length of the words that~$\Motz$
    contains any word~$x$ satisfying the statement. This is
    true when~$|x| = 1$. When~$n := |x| \geq 2$, one has two cases to
    consider. If~$x$ contains a factor~$x_i x_{i + 1}$ such
    that~$x_i = x_{i + 1}$, by setting
    \begin{equation}
        y := (x_1, \dots, x_i, x_{i + 2}, \dots, x_n),
    \end{equation}
    we have~$x = y \circ_i 00$. Otherwise, let~$\Lb$ be the greatest
    letter of~$x$. Since~$x$ satisfies the statement, there is in~$x$
    a factor~$x_i x_{i + 1} x_{i + 2}$ where~$x_{i + 1} = \Lb$
    and~$x_i = x_{i + 2} = \Lb - 1$. By setting
    \begin{equation}
        y := (x_1, \dots, x_i, x_{i + 3}, \dots, x_n),
    \end{equation}
    we have~$x = y \circ_i 010$. Now, for both cases, since by induction
    hypothesis,~$\Motz$ contains~$y$,~$\Motz$ also contains~$x$.
\end{proof}
\medskip

A {\em Motzkin word} is a word~$u$ on the
alphabet~$\left\{-1, 0, 1\right\}$ such that the sum of all letters of~$u$
is~$0$ and, for any prefix~$u'$ of~$u$, the sum of all letters of~$u'$
is a nonnegative integer. The size~$|u|$ of a Motzkin word~$u$ is its
length plus one. In the sequel, we shall denote by~$\bar 1$ the
letter~$-1$. We can represent a Motzkin word~$u$ graphically by a
{\em Motzkin path} that is the path in~$\EnsNat^2$ connecting the
points~$(0, 0)$ and~$(n, 0)$ obtained by drawing a step~$(1, -1)$
(resp.~$(1, 0)$, $(1, 1)$) for each letter~$\bar 1$ (resp.~$0$, $1$) of~$u$.
There is a bijection~$\phi_\Motz$ between the words of~$\Motz$ of arity~$n$
and Motzkin words of size~$n$.
\medskip

To compute~$\phi_\Motz(x)$ where~$x$ is an element of~$\Motz(n)$, build
the word~$u$ of length~$n - 1$ satisfying~$u_i := x_{i + 1} - x_i$
for all~$i \in [n - 1]$.
\medskip

The inverse bijection is computed as follows. The element of~$\Motz$ in
bijection with a Motzkin word~$u$ is the word~$x$ such that~$x_i$ is the
sum of the letters of the prefix~$u_1 \dots u_{i - 1}$  of~$u$,
for all~$i \in [n]$.
\medskip

Since the elements of~$\Motz$ satisfy
Proposition~\ref{prop:ElemMotz},~$\phi_\Motz$ is well-defined.
Figure~\ref{fig:BijMotzMots} shows an example of this bijection.
\begin{figure}[ht]
    \centering
    \begin{equation*}
    \begin{split} 001123221010 \quad \xrightarrow{\phi_\Motz} \quad \end{split}
    \begin{split} 0 1 0 1 1 \bar1 0 \bar1 \bar1 1 \bar1 \end{split}
    \quad \longleftrightarrow \quad
    \begin{split}
    \scalebox{.3}{{\begin{tikzpicture}
        \draw[Grille] (0,0) grid (11,3);
        \node[NoeudDyck](0)at(0,0){};
        \node[NoeudDyck](1)at(1,0){};
        \node[NoeudDyck](2)at(2,1){};
        \node[NoeudDyck](3)at(3,1){};
        \node[NoeudDyck](4)at(4,2){};
        \node[NoeudDyck](5)at(5,3){};
        \node[NoeudDyck](6)at(6,2){};
        \node[NoeudDyck](7)at(7,2){};
        \node[NoeudDyck](8)at(8,1){};
        \node[NoeudDyck](9)at(9,0){};
        \node[NoeudDyck](10)at(10,1){};
        \node[NoeudDyck](11)at(11,0){};
        \draw[PasDyck](0)--(1);
        \draw[PasDyck](1)--(2);
        \draw[PasDyck](2)--(3);
        \draw[PasDyck](3)--(4);
        \draw[PasDyck](4)--(5);
        \draw[PasDyck](5)--(6);
        \draw[PasDyck](6)--(7);
        \draw[PasDyck](7)--(8);
        \draw[PasDyck](8)--(9);
        \draw[PasDyck](9)--(10);
        \draw[PasDyck](10)--(11);
    \end{tikzpicture}}}
    \end{split}
    \end{equation*}
    \caption{Interpretation of an element of the ns operad~$\Motz$
    in terms of Motzkin words and Motzkin paths via the
    bijection~$\phi_\Motz$.}
    \label{fig:BijMotzMots}
\end{figure}
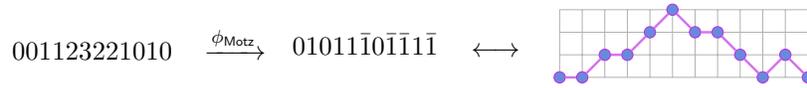
\medskip

The bijection~$\phi_\Motz$ between elements of $\Motz$ and Motzkin words
offers an alternative way to compute the composition of elements of~$\Motz$:
\begin{Proposition} \label{prop:SubsMotz}
    Let~$u$ and~$v$ be two Motzkin words where~$u$ is of size~$n$,
    and~$i \in [n]$ be an integer. Then the composition~$u \circ_i v$ in~$\Motz$
    amounts to insert~$v$ at the~$i$th position into~$u$.
\end{Proposition}
\begin{proof}
    Let~$y$ be the element of~$\Motz$ in bijection by~$\phi_\Motz$
    with~$v$. The statement is a direct consequence of the fact that, by
    Proposition~\ref{prop:ElemMotz},~$y$ starts and ends by~$0$.
\end{proof}
\medskip

Figure~\ref{fig:SubsMotz} shows an example of composition in $\Motz$.
\begin{figure}[ht]
    \centering
    \begin{equation*}
        \begin{split}\scalebox{.3}{\begin{tikzpicture}
            \draw[Grille] (0,0) grid (7,3);
            \node[NoeudDyck](0)at(0,0){};
            \node[NoeudDyck](1)at(1,1){};
            \node[NoeudDyck](2)at(2,1){};
            \node[NoeudDyck,Marque2](3)at(3,2){};
            \node[NoeudDyck](4)at(4,3){};
            \node[NoeudDyck](5)at(5,2){};
            \node[NoeudDyck](6)at(6,1){};
            \node[NoeudDyck](7)at(7,0){};
            \draw[PasDyck](0)--(1);
            \draw[PasDyck](1)--(2);
            \draw[PasDyck](2)--(3);
            \draw[PasDyck](3)--(4);
            \draw[PasDyck](4)--(5);
            \draw[PasDyck](5)--(6);
            \draw[PasDyck](6)--(7);
        \end{tikzpicture}}\end{split}
        \enspace \circ_4 \enspace
        \begin{split}\scalebox{.3}{\begin{tikzpicture}
            \draw[Grille] (0,0) grid (6,2);
            \node[NoeudDyck,Marque1](0)at(0,0){};
            \node[NoeudDyck,Marque1](1)at(1,1){};
            \node[NoeudDyck,Marque1](2)at(2,2){};
            \node[NoeudDyck,Marque1](3)at(3,2){};
            \node[NoeudDyck,Marque1](4)at(4,1){};
            \node[NoeudDyck,Marque1](5)at(5,1){};
            \node[NoeudDyck,Marque1](6)at(6,0){};
            \draw[PasDyck](0)--(1);
            \draw[PasDyck](1)--(2);
            \draw[PasDyck](2)--(3);
            \draw[PasDyck](3)--(4);
            \draw[PasDyck](4)--(5);
            \draw[PasDyck](5)--(6);
        \end{tikzpicture}}\end{split}
        \enspace = \enspace
        \begin{split}\scalebox{.3}{\begin{tikzpicture}
            \draw[Grille] (0,0) grid (13,4);
            \node[NoeudDyck](0)at(0,0){};
            \node[NoeudDyck](1)at(1,1){};
            \node[NoeudDyck](2)at(2,1){};
            \node[NoeudDyck,Marque1](3)at(3,2){};
            \node[NoeudDyck,Marque1](4)at(4,3){};
            \node[NoeudDyck,Marque1](5)at(5,4){};
            \node[NoeudDyck,Marque1](6)at(6,4){};
            \node[NoeudDyck,Marque1](7)at(7,3){};
            \node[NoeudDyck,Marque1](8)at(8,3){};
            \node[NoeudDyck,Marque1](9)at(9,2){};
            \node[NoeudDyck](10)at(10,3){};
            \node[NoeudDyck](11)at(11,2){};
            \node[NoeudDyck](12)at(12,1){};
            \node[NoeudDyck](13)at(13,0){};
            \draw[PasDyck](0)--(1);
            \draw[PasDyck](1)--(2);
            \draw[PasDyck](2)--(3);
            \draw[PasDyck](3)--(4);
            \draw[PasDyck](4)--(5);
            \draw[PasDyck](5)--(6);
            \draw[PasDyck](6)--(7);
            \draw[PasDyck](7)--(8);
            \draw[PasDyck](8)--(9);
            \draw[PasDyck](9)--(10);
            \draw[PasDyck](10)--(11);
            \draw[PasDyck](11)--(12);
            \draw[PasDyck](12)--(13);
        \end{tikzpicture}}\end{split}
    \end{equation*}
    \caption{Interpretation of the partial composition map of the ns operad~$\Motz$
    in terms of Motzkin paths.}
    \label{fig:SubsMotz}
\end{figure}
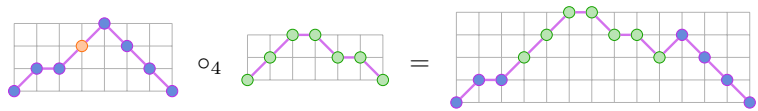
\medskip

\begin{Theoreme} \label{prop:RelationsGenMotz}
    The ns operad~$\Motz$ admits the presentation
    \begin{equation}
        \Motz = \CalF\left(\left\{\MotzOpA, \MotzOpB\right\}\right)/_\equiv,
    \end{equation}
    where~$\MotzOpA$ is of arity~$2$,~$\MotzOpB$ of arity~$3$,
    and~$\equiv$ is the ns operadic congruence generated by

    \begin{minipage}{.45\textwidth}
    \begin{equation}
        \MotzOpA \circ_1 \MotzOpA
            \: \leftrightarrow \: \MotzOpA \circ_2 \MotzOpA,
    \end{equation}
    \begin{equation}
        \MotzOpB \circ_1 \MotzOpA
            \: \leftrightarrow \: \MotzOpA \circ_2 \MotzOpB,
    \end{equation}
    \end{minipage}
    \begin{minipage}{.45\textwidth}
    \begin{equation}
        \MotzOpA \circ_1 \MotzOpB
            \: \leftrightarrow \: \MotzOpB \circ_3 \MotzOpA,
    \end{equation}
    \begin{equation}
        \MotzOpB \circ_1 \MotzOpB
            \: \leftrightarrow \: \MotzOpB \circ_3 \MotzOpB.
    \end{equation}
    \end{minipage}
\end{Theoreme}
\begin{proof}
    First, note that by replacing~$\MotzOpA$ by~$00 \in \Motz(2)$ and~$\MotzOpB$
    by~$010 \in \Motz(3)$, we have $\Eval(x) = \Eval(y)$ for the four
    relations~$x \leftrightarrow y$ of the statement of the Theorem.
    Indeed, the four equivalence classes are, respectively, the ones of the
    elements~$000$, $0010$, $0100$, and~$01010$ of~$\Motz$.
    \smallskip

    Consider now the orientation of~$\leftrightarrow$ into the rewrite
    rule~$\mapsto$ defined by

    \begin{minipage}{.45\textwidth}
    \begin{equation}
        \begin{split}\scalebox{.35}{\begin{tikzpicture}
            \node[Feuille](0)at(0.0,-2){};
            \node[Noeud,EtiqClair,minimum size=15mm](1)at(1.0,-1){\scalebox{2}{\MotzOpA}};
            \node[Feuille](2)at(2.0,-2){};
            \draw[Arete](1)--(0);
            \draw[Arete](1)--(2);
            \node[Noeud,EtiqClair,minimum size=15mm](3)at(3.0,0){\scalebox{2}{\MotzOpA}};
            \node[Feuille](4)at(4.0,-1){};
            \draw[Arete](3)--(1);
            \draw[Arete](3)--(4);
        \end{tikzpicture}}\end{split}
        \begin{split}\enspace \mapsto \enspace \end{split}
        \begin{split}\scalebox{.35}{\begin{tikzpicture}
            \node[Feuille](0)at(0.0,-1){};
            \node[Noeud,EtiqClair,minimum size=15mm](1)at(1.0,0){\scalebox{2}{\MotzOpA}};
            \node[Feuille](2)at(2.0,-2){};
            \node[Noeud,EtiqClair,minimum size=15mm](3)at(3.0,-1){\scalebox{2}{\MotzOpA}};
            \node[Feuille](4)at(4.0,-2){};
            \draw[Arete](3)--(2);
            \draw[Arete](3)--(4);
            \draw[Arete](1)--(0);
            \draw[Arete](1)--(3);
        \end{tikzpicture}}\end{split},
    \end{equation}
    \begin{equation}
        \begin{split}\scalebox{.35}{\begin{tikzpicture}
            \node[Feuille](0)at(0.0,-2){};
            \node[Noeud,EtiqClair,minimum size=15mm](1)at(1.0,-1){\scalebox{2}{\MotzOpA}};
            \node[Feuille](2)at(2.0,-2){};
            \draw[Arete](1)--(0);
            \draw[Arete](1)--(2);
            \node[Noeud,EtiqClair,minimum size=15mm](3)at(3.0,0){\scalebox{2}{\MotzOpB}};
            \node[Feuille](4)at(4.0,-1.5){};
            \node[Feuille](5)at(3,-1.5){};
            \draw[Arete](3)--(1);
            \draw[Arete](3)--(4);
            \draw[Arete](3)--(5);
        \end{tikzpicture}}\end{split}
        \begin{split}\enspace \mapsto \enspace \end{split}
        \begin{split}\scalebox{.35}{\begin{tikzpicture}
            \node[Feuille](0)at(0.0,-1){};
            \node[Noeud,EtiqClair,minimum size=15mm](1)at(1.0,0){\scalebox{2}{\MotzOpA}};
            \node[Feuille](2)at(2.0,-2.5){};
            \node[Noeud,EtiqClair,minimum size=15mm](3)at(3.0,-1){\scalebox{2}{\MotzOpB}};
            \node[Feuille](4)at(4.0,-2.5){};
            \node[Feuille](5)at(3,-2.5){};
            \draw[Arete](3)--(2);
            \draw[Arete](3)--(4);
            \draw[Arete](1)--(0);
            \draw[Arete](1)--(3);
            \draw[Arete](3)--(5);
        \end{tikzpicture}}\end{split},
    \end{equation}
    \end{minipage}
    \begin{minipage}{.45\textwidth}
    \begin{equation}
        \begin{split}\scalebox{.35}{\begin{tikzpicture}
            \node[Feuille](0)at(0.0,-2.5){};
            \node[Noeud,EtiqClair,minimum size=15mm](1)at(1.0,-1){\scalebox{2}{\MotzOpB}};
            \node[Feuille](2)at(2.0,-2.5){};
            \draw[Arete](1)--(0);
            \draw[Arete](1)--(2);
            \node[Noeud,EtiqClair,minimum size=15mm](3)at(3.0,0){\scalebox{2}{\MotzOpA}};
            \node[Feuille](4)at(4.0,-1){};
            \node[Feuille](5)at(1,-2.5){};
            \draw[Arete](3)--(1);
            \draw[Arete](3)--(4);
            \draw[Arete](1)--(5);
        \end{tikzpicture}}\end{split}
        \begin{split}\enspace \mapsto \enspace \end{split}
        \begin{split}\scalebox{.35}{\begin{tikzpicture}
            \node[Feuille](0)at(0.0,-1.5){};
            \node[Noeud,EtiqClair,minimum size=15mm](1)at(1.0,0){\scalebox{2}{\MotzOpB}};
            \node[Feuille](2)at(2.0,-2){};
            \node[Noeud,EtiqClair,minimum size=15mm](3)at(3.0,-1){\scalebox{2}{\MotzOpA}};
            \node[Feuille](4)at(4.0,-2){};
            \node[Feuille](5)at(1,-1.5){};
            \draw[Arete](3)--(2);
            \draw[Arete](3)--(4);
            \draw[Arete](1)--(0);
            \draw[Arete](1)--(3);
            \draw[Arete](1)--(5);
        \end{tikzpicture}}\end{split},
    \end{equation}
    \begin{equation}
        \begin{split}\scalebox{.35}{\begin{tikzpicture}
            \node[Feuille](0)at(0.0,-2.5){};
            \node[Noeud,EtiqClair,minimum size=15mm](1)at(1.0,-1){\scalebox{2}{\MotzOpB}};
            \node[Feuille](2)at(2.0,-2.5){};
            \draw[Arete](1)--(0);
            \draw[Arete](1)--(2);
            \node[Noeud,EtiqClair,minimum size=15mm](3)at(3.0,0){\scalebox{2}{\MotzOpB}};
            \node[Feuille](4)at(4.0,-1.5){};
            \node[Feuille](5)at(1,-2.5){};
            \node[Feuille](6)at(3,-1.5){};
            \draw[Arete](3)--(1);
            \draw[Arete](3)--(4);
            \draw[Arete](1)--(5);
            \draw[Arete](3)--(6);
        \end{tikzpicture}}\end{split}
        \begin{split}\enspace \mapsto \enspace \end{split}
        \begin{split}\scalebox{.35}{\begin{tikzpicture}
            \node[Feuille](0)at(0.0,-1.5){};
            \node[Noeud,EtiqClair,minimum size=15mm](1)at(1.0,0){\scalebox{2}{\MotzOpB}};
            \node[Feuille](2)at(2.0,-2.5){};
            \node[Noeud,EtiqClair,minimum size=15mm](3)at(3.0,-1){\scalebox{2}{\MotzOpB}};
            \node[Feuille](4)at(4.0,-2.5){};
            \node[Feuille](5)at(1,-1.5){};
            \node[Feuille](6)at(3,-2.5){};
            \draw[Arete](3)--(2);
            \draw[Arete](3)--(4);
            \draw[Arete](1)--(0);
            \draw[Arete](1)--(3);
            \draw[Arete](1)--(5);
            \draw[Arete](3)--(6);
        \end{tikzpicture}}\end{split}.
    \end{equation}
    \end{minipage}

    This rewrite rule is terminating. Indeed, it is plain that for any
    rewriting~$T_0 \to T_1$, we have~$\Poids(T_0) < \Poids(T_1)$.
    \smallskip

    Moreover, the normal forms of~$\mapsto$ are all syntax trees
    of~$\CalF\left(\left\{\MotzOpA, \MotzOpB\right\}\right)$
    which have no internal node with an internal node as leftmost child.
    This set~$\CalS$ of trees admits the following regular specification
    \begin{equation}
        \begin{split}\CalS \end{split} = \Feuille
        \enspace + \enspace
        \begin{split}\scalebox{.3}{\begin{tikzpicture}
            \node[Noeud,EtiqClair,minimum size=15mm](1)at(0,0)
                {\scalebox{2}{\MotzOpA}};
            \node[Feuille](2)at(-1.5,-1.5){};
            \node(3)at(1.5,-1.5){\Huge $\CalS$};
            \draw[Arete](1)--(2);
            \draw[Arete](1)--(3);
        \end{tikzpicture}}\end{split}
        \enspace + \enspace
        \begin{split}\scalebox{.3}{\begin{tikzpicture}
            \node[Noeud,EtiqClair,minimum size=15mm](1)at(0,0)
                {\scalebox{2}{\MotzOpB}};
            \node[Feuille](2)at(-1.5,-1.5){};
            \node(3)at(0,-1.5){\Huge $\CalS$};
            \node(4)at(1.5,-1.5){\Huge $\CalS$};
            \draw[Arete](1)--(2);
            \draw[Arete](1)--(3);
            \draw[Arete](1)--(4);
        \end{tikzpicture}}\end{split}.
    \end{equation}
    Hence, the generating series~$F(t)$ of~$\CalS$ satisfies
    \begin{equation}
        F(t) = t + t F(t) + t {F(t)}^2,
    \end{equation}
    that is the generating series of Motzkin words. By
    Proposition~\ref{prop:ElemMotz},~$F(t)$ also is the Hilbert series
    of~$\Motz$.
    \smallskip

    Hence, by Lemma~\ref{lem:PresentationReecriture},~$\Motz$
    admits the claimed presentation.
\end{proof}
\medskip

\subsubsection{A ns operad on integer compositions}
Let~$\Comp$ be the ns suboperad of~$\T \EnsNat_2$ generated
by~$00$ and~$01$. The following table shows the first elements of~$\Comp$.
\begin{center}
    \begin{tabular}{c|p{11cm}}
        Arity & Elements of~$\Comp$ \\ \hline \hline
        $1$ & $0$ \\ \hline
        $2$ & $00$, $01$ \\ \hline
        $3$ & $000$, $001$, $010$, $011$ \\ \hline
        $4$ & $0000$, $0001$, $0010$, $0011$, $0100$, $0101$, $0110$, $0111$ \\ \hline
        $5$ & $00000$, $00001$, $00010$, $00011$, $00100$, $00101$, $00110$,
              $00111$, $01000$, $01001$, $01010$, $01011$, $01100$, $01101$,
              $01110$, $01111$
    \end{tabular}
\end{center}
\medskip

Since~$\FCat{1}$ is the ns suboperad of~$\T \EnsNat$ generated
by~$00$ and~$01$, and since~$\T \EnsNat_2$ is a quotient of~$\T \EnsNat$,~$\Comp$
is a quotient of~$\FCat{1}$. Moreover,~$\FCat{0}$ is a quotient of~$\Comp$
by the ns operadic congruence~$\equiv$ defined for all~$x, y \in \Comp$
by~$x \equiv y$.
\medskip

One has the following characterization of the elements of~$\Comp$:
\begin{Proposition} \label{prop:ElemComp}
    The elements of~$\Comp$ are exactly the words on the alphabet~$\{0, 1\}$
    beginning by~$0$.
\end{Proposition}
\begin{proof}
    It is immediate, from the definition of~$\Comp$ and
    Lemma~\ref{lem:GenerationOpNSEns}, that any element of this ns operad
    starts by~$0$ since its generators~$00$ and~$01$ all start by~$0$.
    \smallskip

    Let us now show by induction on the length of the words that~$\Comp$
    contains any word~$x$ satisfying the statement. This is
    true when~$|x| = 1$. When~$n := |x| \geq 2$, let us observe that if~$x$
    only consists in letters~$0$,~$\Comp$ contains~$x$ because~$x$ can be
    obtained by composing the generator~$00$ with itself. Otherwise,
    $x$ has at least one occurrence of~$1$. Since its first letter is~$0$,
    there is in~$x$ a factor~$x_i x_{i + 1} = 01$. By setting
    \begin{equation}
        y := (x_1, \dots, x_i, x_{i + 2}, \dots, x_n),
    \end{equation}
    we have~$x = y \circ_i 01$. Since~$y$ satisfies the statement, by
    induction hypothesis~$\Comp$ contains~$y$. Hence,~$\Comp$ also
    contains~$x$.
\end{proof}
\medskip

An {\em integer composition} is a sequence~$u_1 \dots u_k$  of positive
integers. The size~$|u|$ of an integer composition~$u$ is the sum of its
letters. It is well-known that there are~$2^{n - 1}$ integer compositions
of size~$n$. We shall represent an integer composition~$u := u_1 \dots u_k$
by a {\em ribbon diagram}, that is the diagram in which each letter~$u_i$
of~$u$ is encoded by a column consisting in~$u_i$ boxes, and the column
encoding the letter~$u_{i + 1}$ is attached on the right edge of  the
bottommost box of the column encoding~$u_i$, for any~$i \in [k - 1]$.
The {\em $i$th box} of a ribbon diagram~$D$ is the $i$th encountered box
by traversing~$D$ column by column from left to right and from top to
bottom. The {\em transpose} of~$D$ is the ribbon diagram obtained by
applying on~$D$ the reflection through the line passing by its first and
its last boxes. There is a bijection~$\phi_\Comp$ between the words of~$\Comp$
of arity~$n$ and ribbon diagrams of integer compositions of size~$n$.
\medskip

To compute~$\phi_\Comp(x)$ where~$x$ is an element of~$\Comp$, iteratively
insert the letters of~$x$ from left to right according to the following
procedure. If~$|x| = 1$, then~$x = 0$ and~$\phi_\Comp(0)$ is the only
ribbon diagram consisting in one box. Otherwise, the insertion of a
letter~$\La$ into~$D$ consists in adding a new box below (resp. to the
right of) the right bottommost box of~$D$ if~$\La = 1$ (resp. $\La = 0$).
\medskip

The inverse bijection is computed as follows. Given a ribbon diagram~$D$
of an integer composition of size~$n$, one computes an element of~$\Comp$
of arity~$n$ by labelling the first box of~$D$ by~$0$ and the $i$th box~$b$
by~$0$ if the $(i - 1)$st box is on the left of~$b$ or by~$1$ otherwise,
for any~$1 \leq i \leq n$. The corresponding element of~$\Comp$ is obtained
by reading the labels of~$D$ from top to bottom and left to right.
\medskip

Since the elements of~$\Comp$ satisfy Proposition~\ref{prop:ElemComp},
$\phi_\Comp$ is well-defined. Hence, we can regard the elements of arity~$n$
of~$\Comp$ as ribbon diagrams with~$n$ boxes. Figure~\ref{fig:BijCompComp}
shows an example of this bijection.
\begin{figure}[ht]
    \centering
    \begin{equation*}
        \begin{split}0100001011011010 \quad \xrightarrow{\phi_\Comp} \quad \end{split}
        \begin{split}\scalebox{.25}{\begin{tikzpicture}
            \node[Boite,EtiqClair]at(0,0){$0$};
            \node[Boite,EtiqClair]at(0,-1){$1$};
            \node[Boite,EtiqClair]at(1,-1){$0$};
            \node[Boite,EtiqClair]at(2,-1){$0$};
            \node[Boite,EtiqClair]at(3,-1){$0$};
            \node[Boite,EtiqClair]at(4,-1){$0$};
            \node[Boite,EtiqClair]at(4,-2){$1$};
            \node[Boite,EtiqClair]at(5,-2){$0$};
            \node[Boite,EtiqClair]at(5,-3){$1$};
            \node[Boite,EtiqClair]at(5,-4){$1$};
            \node[Boite,EtiqClair]at(6,-4){$0$};
            \node[Boite,EtiqClair]at(6,-5){$1$};
            \node[Boite,EtiqClair]at(6,-6){$1$};
            \node[Boite,EtiqClair]at(7,-6){$0$};
            \node[Boite,EtiqClair]at(7,-7){$1$};
            \node[Boite,EtiqClair]at(8,-7){$0$};
        \end{tikzpicture}}\end{split}
        \quad \longleftrightarrow \quad
        \begin{split}\scalebox{.25}{\begin{tikzpicture}
            \node[Boite]at(0,0){};
            \node[Boite]at(0,-1){};
            \node[Boite]at(1,-1){};
            \node[Boite]at(2,-1){};
            \node[Boite]at(3,-1){};
            \node[Boite]at(4,-1){};
            \node[Boite]at(4,-2){};
            \node[Boite]at(5,-2){};
            \node[Boite]at(5,-3){};
            \node[Boite]at(5,-4){};
            \node[Boite]at(6,-4){};
            \node[Boite]at(6,-5){};
            \node[Boite]at(6,-6){};
            \node[Boite]at(7,-6){};
            \node[Boite]at(7,-7){};
            \node[Boite]at(8,-7){};
        \end{tikzpicture}}\end{split}
    \end{equation*}
    \caption{Interpretation of an element of the ns operad~$\Comp$
    in terms of integer compositions via the bijection~$\phi_\Comp$.
    Boxes of the ribbon diagram in the middle are labeled.}
    \label{fig:BijCompComp}
\end{figure}
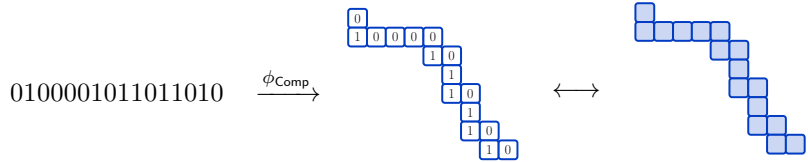
\medskip

Encoding integer compositions by ribbon diagrams offers an alternative
way to compute the composition of elements of~$\Comp$:
\begin{Proposition} \label{prop:SubsComp}
    Let~$C$ and~$D$ be two ribbon diagrams, $i$ be an integer, and~$c$ be
    the $i$th box of~$C$. Then, the composition $C \circ_i D$ in~$\Comp$
    amounts to replace~$c$ by~$D$ if~$c$ is the upper box of its column,
    or to replace~$c$ by the transpose ribbon diagram of~$D$ otherwise.
\end{Proposition}
\begin{proof}
    Let~$x \in \Comp(n)$ and~$y \in \Comp(m)$ such that~$C := \phi_\Comp(x)$
    and~$D := \phi_\Comp(y)$. Let~$E := \phi_\Comp(x \circ_i y)$.
    By definition of~$\phi_\Comp$ and the partial composition maps of~$\Comp$,~$E$
    is obtained by inserting the prefix of length~$i - 1$ of~$x$, then
    the letters of~$y$ incremented by~$x_i$, and finally, the suffix
    of length~$n - i$ of~$x$. If~$c$ is the upper box of its column,
    then~$x_i = 0$, and by definition of~$\phi_\Comp$,~$E$ is obtained by
    replacing~$c$ by~$D$ in~$C$. Otherwise, $c$ is not the upper box
    of its column, and then~$x_i = 1$. Immediately from the definition
    of~$\phi_\Comp$, for any word~$z$ on the alphabet~$\{0, 1\}$, the
    ribbon diagram~$\phi_\Comp(0z)$ is the transpose
    of~$\phi_\Comp(0 \bar z)$ where~$\bar z$ is the complementary of~$z$.
    This implies the statement.
\end{proof}
\medskip

Figure~\ref{fig:SubsComp} shows two examples of compositions in~$\Comp$.
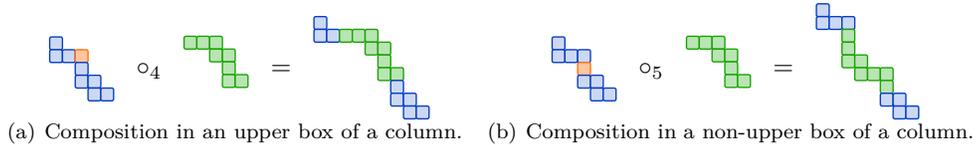
\begin{figure}[ht]
    \centering
    \subfigure[Composition in an upper box of a column.]{\makebox[.45\textwidth]{
    \scalebox{.17}{\raisebox{-6em}{\begin{tikzpicture}
        \node[Boite]at(0,0){};
        \node[Boite]at(0,-1){};
        \node[Boite]at(1,-1){};
        \node[Boite,Marque2]at(2,-1){};
        \node[Boite]at(2,-2){};
        \node[Boite]at(2,-3){};
        \node[Boite]at(3,-3){};
        \node[Boite]at(3,-4){};
        \node[Boite]at(4,-4){};
    \end{tikzpicture}}}
    $\enspace \circ_4 \enspace$
    \scalebox{.17}{\raisebox{-3em}{\begin{tikzpicture}
        \node[Boite,Marque1]at(0,0){};
        \node[Boite,Marque1]at(1,0){};
        \node[Boite,Marque1]at(2,0){};
        \node[Boite,Marque1]at(2,-1){};
        \node[Boite,Marque1]at(3,-1){};
        \node[Boite,Marque1]at(3,-2){};
        \node[Boite,Marque1]at(3,-3){};
        \node[Boite,Marque1]at(4,-3){};
    \end{tikzpicture}}}
    $\enspace = \enspace$
    \scalebox{.17}{\raisebox{-10em}{\begin{tikzpicture}
        \node[Boite]at(0,0){};
        \node[Boite]at(0,-1){};
        \node[Boite]at(1,-1){};
        \node[Boite,Marque1]at(2,-1){};
        \node[Boite,Marque1]at(3,-1){};
        \node[Boite,Marque1]at(4,-1){};
        \node[Boite,Marque1]at(4,-2){};
        \node[Boite,Marque1]at(5,-2){};
        \node[Boite,Marque1]at(5,-3){};
        \node[Boite,Marque1]at(5,-4){};
        \node[Boite,Marque1]at(6,-4){};
        \node[Boite]at(6,-5){};
        \node[Boite]at(6,-6){};
        \node[Boite]at(7,-6){};
        \node[Boite]at(7,-7){};
        \node[Boite]at(8,-7){};
    \end{tikzpicture}}}}}
    \subfigure[Composition in a non-upper box of a column.]{\makebox[.45\textwidth]{
    \scalebox{.17}{\raisebox{-6em}{\begin{tikzpicture}
        \node[Boite]at(0,0){};
        \node[Boite]at(0,-1){};
        \node[Boite]at(1,-1){};
        \node[Boite]at(2,-1){};
        \node[Boite,Marque2]at(2,-2){};
        \node[Boite]at(2,-3){};
        \node[Boite]at(3,-3){};
        \node[Boite]at(3,-4){};
        \node[Boite]at(4,-4){};
    \end{tikzpicture}}}
    $\enspace \circ_5 \enspace$
    \scalebox{.17}{\raisebox{-3em}{\begin{tikzpicture}
        \node[Boite,Marque1]at(0,0){};
        \node[Boite,Marque1]at(1,0){};
        \node[Boite,Marque1]at(2,0){};
        \node[Boite,Marque1]at(2,-1){};
        \node[Boite,Marque1]at(3,-1){};
        \node[Boite,Marque1]at(3,-2){};
        \node[Boite,Marque1]at(3,-3){};
        \node[Boite,Marque1]at(4,-3){};
    \end{tikzpicture}}}
    $\enspace = \enspace$
    \scalebox{.17}{\raisebox{-10em}{\begin{tikzpicture}
        \node[Boite]at(0,0){};
        \node[Boite]at(0,-1){};
        \node[Boite]at(1,-1){};
        \node[Boite]at(2,-1){};
        \node[Boite,Marque1]at(2,-2){};
        \node[Boite,Marque1]at(2,-3){};
        \node[Boite,Marque1]at(2,-4){};
        \node[Boite,Marque1]at(3,-4){};
        \node[Boite,Marque1]at(3,-5){};
        \node[Boite,Marque1]at(4,-5){};
        \node[Boite,Marque1]at(5,-5){};
        \node[Boite,Marque1]at(5,-6){};
        \node[Boite]at(5,-7){};
        \node[Boite]at(6,-7){};
        \node[Boite]at(6,-8){};
        \node[Boite]at(7,-8){};
    \end{tikzpicture}}}}}
    \caption{Interpretation of the partial composition map of the ns
    operad~$\Comp$ in terms of ribbon diagrams.}
    \label{fig:SubsComp}
\end{figure}

\begin{Theoreme} \label{thm:PresentationComp}
    The ns operad~$\Comp$ admits the presentation
    \begin{equation}
        \Comp = \CalF\left(\left\{\CompOpA, \CompOpB\right\}\right)/_\equiv,
    \end{equation}
    where~$\CompOpA$ and~$\CompOpB$ are of arity~$2$, and~$\equiv$ is
    the ns operadic congruence generated by

    \begin{minipage}{.45\textwidth}
    \begin{equation}
        \CompOpA \circ_1 \CompOpA
            \: \leftrightarrow \: \CompOpA \circ_2 \CompOpA,
    \end{equation}
    \begin{equation}
        \CompOpB \circ_1 \CompOpA
            \: \leftrightarrow \: \CompOpA \circ_2 \CompOpB,
    \end{equation}
    \end{minipage}
    \begin{minipage}{.45\textwidth}
    \begin{equation}
        \CompOpB \circ_1 \CompOpB
            \: \leftrightarrow \: \CompOpB \circ_2 \CompOpA,
    \end{equation}
    \begin{equation}
        \CompOpA \circ_1 \CompOpB
            \: \leftrightarrow \: \CompOpB \circ_2 \CompOpB.
    \end{equation}
    \end{minipage}
\end{Theoreme}
\begin{proof}
    First, note that by replacing~$\CompOpA$ by~$00 \in \Comp(2)$
    and~$\CompOpB$ by~$01 \in \Comp(2)$, we have~$\Eval(x) = \Eval(y)$
    for the four relations~$x \leftrightarrow y$ of the statement of
    the Theorem. Indeed, the four equivalence classes are, respectively,
    the ones of the elements~$000$, $001$, $011$, and~$100$ of~$\Comp$.
    \smallskip

    Consider now the orientation of~$\leftrightarrow$ into the rewrite
    rule~$\mapsto$ defined by

    \begin{minipage}{.45\textwidth}
    \begin{equation}
        \begin{split}\scalebox{.35}{\begin{tikzpicture}
            \node[Feuille](0)at(0.0,-2){};
            \node[Noeud,EtiqClair,minimum size=12mm](1)at(1.0,-1){\scalebox{2}{\CompOpA}};
            \node[Feuille](2)at(2.0,-2){};
            \draw[Arete](1)--(0);
            \draw[Arete](1)--(2);
            \node[Noeud,EtiqClair,minimum size=12mm](3)at(3.0,0){\scalebox{2}{\CompOpA}};
            \node[Feuille](4)at(4.0,-1){};
            \draw[Arete](3)--(1);
            \draw[Arete](3)--(4);
        \end{tikzpicture}}\end{split}
        \begin{split}\enspace \mapsto \enspace \end{split}
        \begin{split}\scalebox{.35}{\begin{tikzpicture}
            \node[Feuille](0)at(0.0,-1){};
            \node[Noeud,EtiqClair,minimum size=12mm](1)at(1.0,0){\scalebox{2}{\CompOpA}};
            \node[Feuille](2)at(2.0,-2){};
            \node[Noeud,EtiqClair,minimum size=12mm](3)at(3.0,-1){\scalebox{2}{\CompOpA}};
            \node[Feuille](4)at(4.0,-2){};
            \draw[Arete](3)--(2);
            \draw[Arete](3)--(4);
            \draw[Arete](1)--(0);
            \draw[Arete](1)--(3);
        \end{tikzpicture}}\end{split},
    \end{equation}
    \begin{equation}
        \begin{split}\scalebox{.35}{\begin{tikzpicture}
            \node[Feuille](0)at(0.0,-2){};
            \node[Noeud,EtiqClair,minimum size=12mm](1)at(1.0,-1){\scalebox{1.8}{\CompOpA}};
            \node[Feuille](2)at(2.0,-2){};
            \draw[Arete](1)--(0);
            \draw[Arete](1)--(2);
            \node[Noeud,EtiqClair,minimum size=12mm](3)at(3.0,0){\scalebox{1.8}{\CompOpB}};
            \node[Feuille](4)at(4.0,-1){};
            \draw[Arete](3)--(1);
            \draw[Arete](3)--(4);
        \end{tikzpicture}}\end{split}
        \begin{split}\enspace \mapsto \enspace \end{split}
        \begin{split}\scalebox{.35}{\begin{tikzpicture}
            \node[Feuille](0)at(0.0,-1){};
            \node[Noeud,EtiqClair,minimum size=12mm](1)at(1.0,0){\scalebox{1.8}{\CompOpA}};
            \node[Feuille](2)at(2.0,-2){};
            \node[Noeud,EtiqClair,minimum size=12mm](3)at(3.0,-1){\scalebox{1.8}{\CompOpB}};
            \node[Feuille](4)at(4.0,-2){};
            \draw[Arete](3)--(2);
            \draw[Arete](3)--(4);
            \draw[Arete](1)--(0);
            \draw[Arete](1)--(3);
        \end{tikzpicture}}\end{split},
    \end{equation}
    \end{minipage}
    \begin{minipage}{.45\textwidth}
    \begin{equation}
        \begin{split}\scalebox{.35}{\begin{tikzpicture}
            \node[Feuille](0)at(0.0,-2){};
            \node[Noeud,EtiqClair,minimum size=12mm](1)at(1.0,-1){\scalebox{1.8}{\CompOpB}};
            \node[Feuille](2)at(2.0,-2){};
            \draw[Arete](1)--(0);
            \draw[Arete](1)--(2);
            \node[Noeud,EtiqClair,minimum size=12mm](3)at(3.0,0){\scalebox{1.8}{\CompOpB}};
            \node[Feuille](4)at(4.0,-1){};
            \draw[Arete](3)--(1);
            \draw[Arete](3)--(4);
        \end{tikzpicture}}\end{split}
        \begin{split}\enspace \mapsto \enspace \end{split}
        \begin{split}\scalebox{.35}{\begin{tikzpicture}
            \node[Feuille](0)at(0.0,-1){};
            \node[Noeud,EtiqClair,minimum size=12mm](1)at(1.0,0){\scalebox{1.8}{\CompOpB}};
            \node[Feuille](2)at(2.0,-2){};
            \node[Noeud,EtiqClair,minimum size=12mm](3)at(3.0,-1){\scalebox{1.8}{\CompOpA}};
            \node[Feuille](4)at(4.0,-2){};
            \draw[Arete](3)--(2);
            \draw[Arete](3)--(4);
            \draw[Arete](1)--(0);
            \draw[Arete](1)--(3);
        \end{tikzpicture}}\end{split},
    \end{equation}
    \begin{equation}
        \begin{split}\scalebox{.35}{\begin{tikzpicture}
            \node[Feuille](0)at(0.0,-2){};
            \node[Noeud,EtiqClair,minimum size=12mm](1)at(1.0,-1){\scalebox{1.8}{\CompOpB}};
            \node[Feuille](2)at(2.0,-2){};
            \draw[Arete](1)--(0);
            \draw[Arete](1)--(2);
            \node[Noeud,EtiqClair,minimum size=12mm](3)at(3.0,0){\scalebox{1.8}{\CompOpA}};
            \node[Feuille](4)at(4.0,-1){};
            \draw[Arete](3)--(1);
            \draw[Arete](3)--(4);
        \end{tikzpicture}}\end{split}
        \begin{split}\enspace \mapsto \enspace \end{split}
        \begin{split}\scalebox{.35}{\begin{tikzpicture}
            \node[Feuille](0)at(0.0,-1){};
            \node[Noeud,EtiqClair,minimum size=12mm](1)at(1.0,0){\scalebox{1.8}{\CompOpB}};
            \node[Feuille](2)at(2.0,-2){};
            \node[Noeud,EtiqClair,minimum size=12mm](3)at(3.0,-1){\scalebox{1.8}{\CompOpB}};
            \node[Feuille](4)at(4.0,-2){};
            \draw[Arete](3)--(2);
            \draw[Arete](3)--(4);
            \draw[Arete](1)--(0);
            \draw[Arete](1)--(3);
        \end{tikzpicture}}\end{split}.
    \end{equation}
    \end{minipage}

    This rewrite rule is terminating. Indeed, it is plain that for any
    rewriting~$T_0 \to T_1$, we have~$\Poids(T_0) < \Poids(T_1)$.
    \smallskip

    Moreover, the normal forms of~$\mapsto$ are all syntax trees
    of~$\CalF\left(\left\{\CompOpA, \CompOpB\right\}\right)$ which have
    no internal node with an internal node as leftmost child. Hence, the
    generating series~$F(t)$ of the normal forms of~$\mapsto$ is
    \begin{equation}
        F(t) = \sum_{n \geq 1} 2^{n - 1} t^n.
    \end{equation}
    By Proposition~\ref{prop:ElemComp},~$F(t)$ also is the Hilbert
    series of~$\Comp$.
    \smallskip

    Hence, by Lemma~\ref{lem:PresentationReecriture},~$\Comp$ admits
    the claimed presentation.
\end{proof}
\medskip

\subsubsection{A ns operad on directed animals}
Let~$\AnD$ be the ns suboperad of~$\T \EnsNat_3$ generated by~$00$ and~$01$.
We shall here denote by $\bar 1$ the representative of the equivalence
class of~$2$ in~$\EnsNat_3$. The following table shows the first elements
of~$\AnD$.
\begin{center}
    \renewcommand{\arraystretch}{1.2}
    \begin{tabular}{c|p{11cm}}
        Arity & Elements of~$\AnD$ \\ \hline \hline
        $1$ & $0$ \\ \hline
        $2$ & $00$, $01$ \\ \hline
        $3$ & $000$, $001$, $010$, $011$, $01\bar1 $ \\ \hline
        $4$ & $0000$, $0001$, $0010$, $0011$, $001\bar1 $, $0100$, $0101$,
              $0110$, $0111$, $011\bar1$, $01\bar1 0$, $01\bar1 1$, $01\bar1 \bar1$ \\ \hline
        $5$ & $00000$, $00001$, $00010$, $00011$, $0001\bar1$, $00100$, $00101$,
              $00110$, $00111$, $0011\bar1$, $001\bar10$, $001\bar11$, $001\bar1\bar1$,
              $01000$, $01001$, $01010$, $01011$, $0101\bar1$, $01100$, $01101$, $01110$,
              $01111$, $0111\bar1$, $011\bar10$, $011\bar11$, $011\bar1\bar1$,
              $01\bar100$, $01\bar101$, $01\bar10\bar1$, $01\bar110$, $01\bar111$,
              $01\bar11\bar1$, $01\bar1\bar10$, $01\bar1\bar11$, $01\bar1\bar1\bar1$
    \end{tabular}
\end{center}
\medskip

Since~$\FCat{1}$ is the ns suboperad of~$\T \EnsNat$ generated by~$00$
and~$01$, and since~$\T \EnsNat_3$ is a quotient of~$\T \EnsNat$,
$\AnD$ is a quotient of~$\FCat{1}$. Moreover,~$\AnD$ is a quotient
of~$\FCat{0}$ by the ns operadic congruence~$\equiv$ defined for
all~$x, y \in \AnD$ by~$x \equiv y$.
\medskip

\begin{Proposition} \label{prop:BijAnDPrefMotz}
    Let~$\phi_\AnD : \AnD(n) \to \left\{\bar 1, 0, 1\right\}^{n - 1}$
    the mapping defined for any element~$x$ of arity~$n$ of~$\AnD$ by
    \begin{equation}
        \phi_\AnD(x) := (x_1 * x_2, x_2 * x_3, \dots, x_{n - 1} * x_n),
    \end{equation}
    where~$x_i * x_{i + 1} := x_{i + 1} - x_i \mod 3$.
    Then, $\phi_\AnD$ is a bijection between the elements of arity~$n$
    of~$\AnD$ and prefixes of Motzkin words of length~$n - 1$.
\end{Proposition}
\begin{proof}
    Let us first show by induction on the length of the words
    that for any~$x \in \AnD$, $\phi_\AnD(x)$ is a prefix of a Motzkin
    word of length~$|x| - 1$. This is true when~$|x| = 1$. When~$|x| \geq 2$, by
    Lemma~\ref{lem:GenerationOpNSEns}, there is an element~$y$ of~$\AnD$
    of length~$n := |x| - 1$, an integer~$i \in [n]$, and~$g \in \{00, 01\}$
    such that~$x = y \circ_i g$. We now have two cases depending on~$g$.
    \begin{enumerate}[label = {\it Case \arabic*.}, fullwidth]
        \item If~$g = 00$, then
        \begin{equation}
            x = (y_1, \dots, y_i, y_i, y_{i + 1}, \dots, y_n).
        \end{equation}
        By induction hypothesis,~$\phi_\AnD(y)$ is a prefix a Motzkin word
        of length~$n - 1$. Since~$x$ is obtained from~$y$ by duplicating
        its $i$th letter,~$\phi_\AnD(x)$ is obtained from~$\phi_\AnD(y)$
        by inserting a~$0$ at an appropriate place. Hence,~$\phi_\AnD(x)$ is
        a prefix of a Motzkin word of length~$n$.
        \smallskip

        \item  Otherwise, we have~$g = 01$ and then,
        \begin{equation}
            x = (y_1, \dots, y_i, y_i + 1, y_{i + 1}, \dots, y_n),
        \end{equation}
        where~$+$ denotes the addition in~$\EnsNat_3$. We have now two sub-cases
        whether~$y_i$ is the last letter of~$y$.
        \begin{enumerate}[label = {\it Case \arabic{enumi}.\arabic*.}, fullwidth]
            \item If it is the case, then~$x_{i + 1} = y_i + 1$ is the last letter of~$x$
            and~$\phi_\AnD(x)$ is obtained from~$\phi_\AnD(y)$ by concatenating
            a~$1$ on the right. Hence, since by induction hypothesis,
            $\phi_\AnD(y)$ is a prefix of a Motzkin word of length~$n - 1$,
            $\phi_\AnD(x)$ is a prefix of a Motzkin word of length~$n$.
            \smallskip

            \item Otherwise, we have~$i < n$. We observe that~$\phi_\AnD(x)$
            is obtained from~$\phi_\AnD(y)$ by replacing a letter~$0$
            (resp. $1$, $\bar 1$) by a factor~$1\bar 1$ (resp. $10$, $11$)
            at an appropriate place. Hence, since by induction hypothesis,
            $\phi_\AnD(y)$ is a prefix of a Motzkin word of length~$n - 1$,
            $\phi_\AnD(x)$ is a prefix of a Motzkin word of length~$n$.
        \end{enumerate}
    \end{enumerate}
    \smallskip

    Let us now show that~$\phi_\AnD$ is a bijection between the elements
    of arity~$n$ of~$\AnD$ and prefixes of Motzkin words of length~$n - 1$.
    \smallskip

    The injectivity of~$\phi_\AnD$ is a direct consequence of the fact that,
    given a element~$x$ of~$\AnD$ and a letter~$\La \in \{\bar 1, 0, 1\}$,
    there is at most one letter~$\Lb \in \{\bar 1, 0, 1\}$ such
    that~$\phi_\AnD(x \Lb) = \phi_\AnD(x) \La$.
    \smallskip

    Let us finally show that~$\phi_\AnD$ is a surjection. We proceed by
    induction on the length of the words to construct for any prefix of a
    Motzkin word~$u$ an element~$x$ of~$\AnD$ such that~$\phi_\AnD(x) = u$.
    When~$u$ is the empty word,~$x := 0$ is an element of~$\AnD(1)$ and
    since~$\phi_\AnD(x)$ is the empty word, the property is satisfied.
    When~$n := |u| \geq 1$, one has two cases to consider depending on
    the last letter~$u_n$ of~$u$.
    \begin{enumerate}[label = {\it Case \arabic*'.}, fullwidth]
        \item If~$u_n \in \{0, 1\}$, by induction hypothesis, there is an
        element~$y$ of~$\AnD(n)$ such that $\phi_\AnD(y) = u_1 \dots u_{n - 1}$.
        Hence, by setting~$x := y \circ_n 0 u_n$, its follows, by definition
        of~$\phi_\AnD$, that~$x$ is a preimage of~$u$ for~$\phi_\AnD$.
        \smallskip

        \item Otherwise, we have~$u_n = \bar 1$ and there is at least one
        occurrence of a~$1$ in~$u$. Hence, let~$i \in [n - 1]$ be the greatest
        integer such that~$u_i = 1$. We now have two sub-cases depending on
        the value of~$u_{i + 1}$.
        \begin{enumerate}[label = {\it Case \arabic{enumi}'.\arabic*.}, fullwidth]
            \item If~$u_{i + 1} = 0$, the word
            \begin{equation}
                u' := u_1 \dots u_i u_{i + 2} \dots u_n
            \end{equation}
            is still a prefix of a Motzkin word. Then, by induction hypothesis,
            there is an element~$y$ of~$\AnD(n)$ such that~$\phi_\AnD(y) = u'$.
            Hence, by setting~$x := y \circ_i 00$, its follows, by definition
            of~$\phi_\AnD$, that~$x$ is a preimage of~$u$ for~$\phi_\AnD$.
            \smallskip

            \item Otherwise, we have~$u_{i + 1} = \bar 1$. Then, the word
            \begin{equation}
                u' := u_1 \dots u_{i - 1} 0 u_{i + 2} \dots u_n
            \end{equation}
            is still a prefix of a Motzkin word. Then, by induction
            hypothesis, there is an element~$y$ of~$\AnD(n)$ such
            that~$\phi_\AnD(y) = u'$. Hence by setting~$x := y \circ_i 01$,
            its follows, by definition of~$\phi_\AnD$, that~$x$ is a
            preimage of~$u$ for~$\phi_\AnD$.
        \end{enumerate}
    \end{enumerate}
    \smallskip

    We then have proved that~$\phi_\AnD$ is well-defined, injective, and
    surjective. Hence, it is a bijection between elements of arity~$n$
    of~$\AnD$ and prefixes of Motzkin words of length~$n - 1$.
\end{proof}
\medskip

Here are two examples of images by~$\phi_\AnD$ of elements of~$\AnD$.
\begin{equation}
    \phi_\AnD(011\bar 1\bar 10\bar 101) = 10101\bar111,
\end{equation}
\begin{equation}
    \phi_\AnD(010010101\bar 11) = 1\bar101\bar11\bar111\bar1.
\end{equation}
\medskip

Recall that a {\em directed animal} is a subset~$A$ of~$\EnsNat^2$ such
that~$(0, 0) \in A$ and $(i, j) \in A$ with~$i \geq 1$ or~$j \geq 1$
implies~$(i - 1, j) \in A$ or~$(i, j - 1) \in A$. The size of a directed
animal~$A$ is its cardinality. Figure~\ref{fig:AnimalD} shows a
directed animal.
\begin{figure}[ht]
    \centering
    \scalebox{.3}{
    \begin{tikzpicture}
        \draw[Grille] (0,0) grid (7,6);
        \node[SommetAnimal](0)at(0,0){};
        \node[SommetAnimal](1)at(1,0){};
        \node[SommetAnimal](2)at(1,1){};
        \node[SommetAnimal](3)at(1,2){};
        \node[SommetAnimal](4)at(1,3){};
        \node[SommetAnimal](5)at(1,4){};
        \node[SommetAnimal](6)at(1,5){};
        \node[SommetAnimal](7)at(1,6){};
        \node[SommetAnimal](8)at(2,1){};
        \node[SommetAnimal](9)at(2,2){};
        \node[SommetAnimal](10)at(2,4){};
        \node[SommetAnimal](11)at(3,2){};
        \node[SommetAnimal](12)at(3,3){};
        \node[SommetAnimal](13)at(3,4){};
        \node[SommetAnimal](14)at(4,3){};
        \node[SommetAnimal](15)at(5,3){};
        \node[SommetAnimal](16)at(5,4){};
        \node[SommetAnimal](17)at(5,5){};
        \node[SommetAnimal](18)at(6,3){};
        \node[SommetAnimal](19)at(6,5){};
        \node[SommetAnimal](20)at(7,5){};
        \draw[PasDyck](0)--(1);
        \draw[PasDyck](1)--(2);
        \draw[PasDyck](2)--(3);
        \draw[PasDyck](3)--(4);
        \draw[PasDyck](4)--(5);
        \draw[PasDyck](5)--(6);
        \draw[PasDyck](6)--(7);
        \draw[PasDyck](2)--(8);
        \draw[PasDyck](5)--(10);
        \draw[PasDyck](10)--(13);
        \draw[PasDyck](8)--(9);
        \draw[PasDyck](9)--(11);
        \draw[PasDyck](11)--(12);
        \draw[PasDyck](12)--(13);
        \draw[PasDyck](12)--(14);
        \draw[PasDyck](14)--(15);
        \draw[PasDyck](15)--(18);
        \draw[PasDyck](15)--(16);
        \draw[PasDyck](16)--(17);
        \draw[PasDyck](17)--(19);
        \draw[PasDyck](19)--(20);
        \draw[PasDyck](3)--(9);
    \end{tikzpicture}}
    \caption{A directed animal of size~$21$. The point~$(0, 0)$ is the
    lowest and leftmost point.}
    \label{fig:AnimalD}
\end{figure}
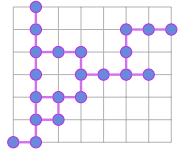
\medskip

According to~\cite{GBV88}, there is a bijection~$\alpha$ between the set
of prefixes of Motzkin words of length~$n - 1$ and the set of directed
animals of size~$n$. Hence, by Proposition~\ref{prop:BijAnDPrefMotz}, the
map~$\alpha \circ \phi_\AnD$ is a bijection between the elements of~$\AnD$
of arity~$n$ and directed animals of size~$n$ and moreover, $\AnD$ can
be seen as a ns operad on directed animals.
\medskip

\begin{Theoreme} \label{thm:PresentationAnD}
    The ns operad~$\AnD$ admits the presentation
    \begin{equation*}
        \AnD = \CalF\left(\left\{\AnDOpA, \AnDOpB\right\}\right)/_\equiv,
    \end{equation*}
    where~$\AnDOpA$ and~$\AnDOpB$ are of arity~$2$, and~$\equiv$ is the ns
    operadic congruence generated by

    \begin{minipage}{.45\textwidth}
    \begin{equation}
        \AnDOpA \circ_1 \AnDOpA
            \: \leftrightarrow \: \AnDOpA \circ_2 \AnDOpA,
    \end{equation}
    \begin{equation}
        \AnDOpB \circ_1 \AnDOpA
            \: \leftrightarrow \: \AnDOpA \circ_2 \AnDOpB,
    \end{equation}
    \end{minipage}
    \begin{minipage}{.45\textwidth}
    \begin{equation}
        \AnDOpB \circ_1 \AnDOpB
            \: \leftrightarrow \: \AnDOpB \circ_2 \AnDOpA,
    \end{equation}
    \begin{equation} \label{eq:RelAnDDeg3}
        (\AnDOpA \circ_1 \AnDOpB) \circ_2 \AnDOpB
            \: \leftrightarrow \:
        (\AnDOpB \circ_2 \AnDOpB) \circ_3 \AnDOpB.
    \end{equation}
    \end{minipage}
\end{Theoreme}
\begin{proof}
    First, note that by replacing~$\AnDOpA$ by~$00 \in \AnD(2)$ and~$\AnDOpB$
    by~$01 \in \AnD(2)$, we have~$\Eval(x) = \Eval(y)$ for the four
    relations~$x \leftrightarrow y$ of the statement of the Theorem.
    Indeed, the four equivalence classes are, respectively, the ones of the
    elements~$000$, $001$, $011$, and~$01\bar10$ of~$\AnD$.
    \smallskip

    Consider now the orientation of~$\leftrightarrow$ into the rewrite
    rule~$\mapsto$ defined by

    \begin{minipage}{.45\textwidth}
    \begin{equation} \label{eq:RelOrientationAnD1}
        \begin{split}\scalebox{.35}{\begin{tikzpicture}
            \node[Feuille](0)at(0.0,-2){};
            \node[Noeud,EtiqClair,minimum size=12mm](1)at(1.0,-1){\scalebox{2}{\AnDOpA}};
            \node[Feuille](2)at(2.0,-2){};
            \draw[Arete](1)--(0);
            \draw[Arete](1)--(2);
            \node[Noeud,EtiqClair,minimum size=12mm](3)at(3.0,0){\scalebox{2}{\AnDOpA}};
            \node[Feuille](4)at(4.0,-1){};
            \draw[Arete](3)--(1);
            \draw[Arete](3)--(4);
        \end{tikzpicture}}\end{split}
        \begin{split}\enspace \mapsto \enspace \end{split}
        \begin{split}\scalebox{.35}{\begin{tikzpicture}
            \node[Feuille](0)at(0.0,-1){};
            \node[Noeud,EtiqClair,minimum size=12mm](1)at(1.0,0){\scalebox{2}{\AnDOpA}};
            \node[Feuille](2)at(2.0,-2){};
            \node[Noeud,EtiqClair,minimum size=12mm](3)at(3.0,-1){\scalebox{2}{\AnDOpA}};
            \node[Feuille](4)at(4.0,-2){};
            \draw[Arete](3)--(2);
            \draw[Arete](3)--(4);
            \draw[Arete](1)--(0);
            \draw[Arete](1)--(3);
        \end{tikzpicture}}\end{split},
    \end{equation}
    \begin{equation} \label{eq:RelOrientationAnD2}
        \begin{split}\scalebox{.35}{\begin{tikzpicture}
            \node[Feuille](0)at(0.0,-2){};
            \node[Noeud,EtiqClair,minimum size=12mm](1)at(1.0,-1){\scalebox{2}{\AnDOpA}};
            \node[Feuille](2)at(2.0,-2){};
            \draw[Arete](1)--(0);
            \draw[Arete](1)--(2);
            \node[Noeud,EtiqClair,minimum size=12mm](3)at(3.0,0){\scalebox{2}{\AnDOpB}};
            \node[Feuille](4)at(4.0,-1){};
            \draw[Arete](3)--(1);
            \draw[Arete](3)--(4);
        \end{tikzpicture}}\end{split}
        \begin{split}\enspace \mapsto \enspace \end{split}
        \begin{split}\scalebox{.35}{\begin{tikzpicture}
            \node[Feuille](0)at(0.0,-1){};
            \node[Noeud,EtiqClair,minimum size=12mm](1)at(1.0,0){\scalebox{2}{\AnDOpA}};
            \node[Feuille](2)at(2.0,-2){};
            \node[Noeud,EtiqClair,minimum size=12mm](3)at(3.0,-1){\scalebox{2}{\AnDOpB}};
            \node[Feuille](4)at(4.0,-2){};
            \draw[Arete](3)--(2);
            \draw[Arete](3)--(4);
            \draw[Arete](1)--(0);
            \draw[Arete](1)--(3);
        \end{tikzpicture}}\end{split},
    \end{equation}
    \end{minipage}
    \begin{minipage}{.5\textwidth}
    \begin{equation} \label{eq:RelOrientationAnD3}
        \begin{split}\scalebox{.35}{\begin{tikzpicture}
            \node[Feuille](0)at(0.0,-1){};
            \node[Noeud,EtiqClair,minimum size=12mm](1)at(1.0,0){\scalebox{2}{\AnDOpB}};
            \node[Feuille](2)at(2.0,-2){};
            \node[Noeud,EtiqClair,minimum size=12mm](3)at(3.0,-1){\scalebox{2}{\AnDOpA}};
            \node[Feuille](4)at(4.0,-2){};
            \draw[Arete](3)--(2);
            \draw[Arete](3)--(4);
            \draw[Arete](1)--(0);
            \draw[Arete](1)--(3);
        \end{tikzpicture}}\end{split}
        \begin{split}\enspace \mapsto \enspace \end{split}
        \begin{split}\scalebox{.35}{\begin{tikzpicture}
            \node[Feuille](0)at(0.0,-2){};
            \node[Noeud,EtiqClair,minimum size=12mm](1)at(1.0,-1){\scalebox{2}{\AnDOpB}};
            \node[Feuille](2)at(2.0,-2){};
            \draw[Arete](1)--(0);
            \draw[Arete](1)--(2);
            \node[Noeud,EtiqClair,minimum size=12mm](3)at(3.0,0){\scalebox{2}{\AnDOpB}};
            \node[Feuille](4)at(4.0,-1){};
            \draw[Arete](3)--(1);
            \draw[Arete](3)--(4);
        \end{tikzpicture}}\end{split},
    \end{equation}
    \begin{equation}
        \begin{split}\scalebox{.35}{\begin{tikzpicture}
            \node[Feuille](0)at(0.0,-1){};
            \node[Noeud,EtiqClair,minimum size=12mm](1)at(1.0,0){\scalebox{2}{\AnDOpB}};
            \node[Feuille](2)at(2.0,-2){};
            \node[Noeud,EtiqClair,minimum size=12mm](3)at(3.0,-1){\scalebox{2}{\AnDOpB}};
            \node[Feuille](4)at(4.0,-3){};
            \node[Noeud,EtiqClair,minimum size=12mm](5)at(5.0,-2){\scalebox{2}{\AnDOpB}};
            \node[Feuille](6)at(6.0,-3){};
            \draw[Arete](5)--(4);
            \draw[Arete](5)--(6);
            \draw[Arete](3)--(2);
            \draw[Arete](3)--(5);
            \draw[Arete](1)--(0);
            \draw[Arete](1)--(3);
        \end{tikzpicture}}\end{split}
        \begin{split}\enspace \mapsto \enspace \end{split}
        \begin{split}\scalebox{.35}{\begin{tikzpicture}
            \node[Feuille](0)at(0.0,-2){};
            \node[Noeud,EtiqClair,minimum size=12mm](1)at(1.0,-1){\scalebox{2}{\AnDOpB}};
            \node[Feuille](2)at(2.0,-3){};
            \node[Noeud,EtiqClair,minimum size=12mm](3)at(3.0,-2){\scalebox{2}{\AnDOpB}};
            \node[Feuille](4)at(4.0,-3){};
            \draw[Arete](3)--(2);
            \draw[Arete](3)--(4);
            \draw[Arete](1)--(0);
            \draw[Arete](1)--(3);
            \node[Noeud,EtiqClair,minimum size=12mm](5)at(5.0,0){\scalebox{2}{\AnDOpA}};
            \node[Feuille](6)at(6.0,-1){};
            \draw[Arete](5)--(1);
            \draw[Arete](5)--(6);
        \end{tikzpicture}}\end{split}.
    \end{equation}
    \end{minipage}

    This rewrite rule is terminating. Indeed, let~$T$ be a syntax tree
    of~$\CalF\left(\left\{\AnDOpA, \AnDOpB\right\}\right)$. By associating
    the pair~$(-k_T, \Poids(T))$ with~$T$, where~$k_T$ is the sum, for all
    internal nodes~$x$ of~$T$ labeled by~$\AnDOpB$, of the number of
    internal nodes constituting the right subtree of~$x$, it is plain that
    for any rewriting~$T_0 \to T_1$, we have~$k_{T_0} < k_{T_1}$,
    or~$k_{T_0} = k_{T_1}$ and $\Poids(T_0) < \Poids(T_1)$.
    \smallskip

    Moreover, the normal forms of~$\mapsto$ are all syntax trees
    of~$\CalF\left(\left\{\AnDOpA, \AnDOpB\right\}\right)$ such that no
    internal node labeled by~$\AnDOpA$ has a left child labeled by~$\AnDOpA$,
    no internal node labeled by~$\AnDOpB$ has a child labeled
    by~$\AnDOpA$, and no internal node labeled by~$\AnDOpB$ has a right
    child labeled by~$\AnDOpB$ which has a right child labeled by~$\AnDOpB$.
    This set~$\CalS$ of syntax trees admits the following regular specification
    \begin{equation}
        \begin{split} \CalS = \CalT \end{split}
        \enspace + \enspace
        \begin{split}\scalebox{.3}{\begin{tikzpicture}
            \node[Noeud,EtiqClair,minimum size=15mm](1)at(0,0)
                {\scalebox{2}{\AnDOpA}};
            \node(2)at(-1.5,-1.5){\Huge $\CalT$};
            \node(3)at(1.5,-1.5){\Huge $\CalS$};
            \draw[Arete](1)--(2);
            \draw[Arete](1)--(3);
        \end{tikzpicture}}\end{split},
    \end{equation}
    where~$\CalT$ is the set of syntax trees admitting the
    following regular specification
    \begin{equation}
        \begin{split} \CalT = \Feuille \end{split}
        \enspace + \enspace
        \begin{split}\scalebox{.3}{\begin{tikzpicture}
            \node[Noeud,EtiqClair,minimum size=15mm](1)at(0,0)
                {\scalebox{2}{\AnDOpB}};
            \node(2)at(-1.5,-1.5){\Huge $\CalT$};
            \node[Feuille](3)at(1.5,-1.5){};
            \draw[Arete](1)--(2);
            \draw[Arete](1)--(3);
        \end{tikzpicture}}\end{split}
        \enspace + \enspace
        \begin{split}\scalebox{.3}{\begin{tikzpicture}
            \node[Noeud,EtiqClair,minimum size=15mm](1)at(0,0)
                {\scalebox{2}{\AnDOpB}};
            \node(2)at(-1.5,-1.5){\Huge $\CalT$};
            \node[Noeud,EtiqClair,minimum size=15mm](3)at(2,-2)
                {\scalebox{2}{\AnDOpB}};
            \node(4)at(1,-3.5){\Huge $\CalT$};
            \node[Feuille](5)at(3,-3.5){};
            \draw[Arete](1)--(2);
            \draw[Arete](1)--(3);
            \draw[Arete](3)--(4);
            \draw[Arete](3)--(5);
        \end{tikzpicture}}\end{split}.
    \end{equation}

    Hence, the generating series~$F(t)$ of~$\CalS$ satisfies
    \begin{equation}
        F(t) = \frac{1 - 3t - (1 - 2t -3t^2)^{1/2}}{6t - 2},
    \end{equation}
    which is the generating function of directed animals. By
    Proposition~\ref{prop:BijAnDPrefMotz},~$F(t)$ also is the Hilbert
    series of~$\AnD$.
    \smallskip

    Hence, by Lemma~\ref{lem:PresentationReecriture},~$\AnD$ admits the
    claimed presentation.
\end{proof}
\medskip

Since the nontrivial relation~\eqref{eq:RelAnDDeg3} has degree~$3$, the
presentation of~$\AnD$ exhibited by Theorem~\ref{thm:PresentationAnD} is not
quadratic. Moreover, $\AnD$ is not a quadratic ns operad since, as an exhaustive
inspection can show, there is no quadratic ns operad generated by two generators
of arity~$2$ which has the same dimensions as~$\AnD$.
\medskip

\subsubsection{A ns operad on segmented integer compositions}
Let~$\SComp$ be the ns suboperad of~$\T \EnsNat_3$ generated by~$00$, $01$,
and~$02$. The following table shows the first elements of~$\SComp$.
\begin{center}
    \begin{tabular}{c|p{11cm}}
        Arity & Elements of~$\SComp$ \\ \hline \hline
        $1$ & $0$ \\ \hline
        $2$ & $00$, $01$, $02$ \\ \hline
        $3$ & $000$, $001$, $002$, $010$, $011$, $012$, $020$, $021$, $022$ \\ \hline
        $4$ & $0000$, $0001$, $0002$, $0010$, $0011$, $0012$, $0020$, $0021$,
              $0022$, $0100$, $0101$, $0102$, $0110$, $0111$, $0112$, $0120$,
              $0121$, $0122$, $0200$, $0201$, $0202$, $0210$, $0211$, $0212$,
              $0220$, $0221$, $0222$
    \end{tabular}
\end{center}
\medskip

Since~$\FCat{2}$ is the ns suboperad of~$\T \EnsNat$ generated by~$00$, $01$,
and~$02$, and since~$\T \EnsNat_3$ is a quotient of~$\T \EnsNat$, $\SComp$
is a quotient of~$\FCat{2}$. Moreover, since~$\AnD$ is generated by~$00$
and~$01$,~$\AnD$ is a ns suboperad of~$\SComp$.
\medskip

One has the following characterization of the elements of~$\SComp$:
\begin{Proposition} \label{prop:ElemSComp}
    The elements of~$\SComp$ are exactly the words on the alphabet~$\{0, 1, 2\}$
    beginning by~$0$.
\end{Proposition}
\begin{proof}
    It is immediate, from the definition of~$\SComp$ and
    Lemma~\ref{lem:GenerationOpNSEns}, that any element of this ns operad
    starts by~$0$ since its generators~$00$, $01$, and~$02$ all start by~$0$.
    \medskip

    Let us now show by induction on the length of the words that~$\SComp$
    contains any word~$x$ satisfying the statement. This is
    true when~$|x| = 1$. When~$n := |x| \geq 2$, let us observe that if~$x$
    only consists in letters~$0$, $\SComp$ contains~$x$ because~$x$ can
    be obtained by composing the generator~$00$ with itself. Otherwise,
    $x$ has at least one occurrence of a~$1$ or a~$2$. Since its first
    letter is~$0$, there is in~$x$ a factor~$x_i x_{i + 1} =: g$ such
    that~$g \in \{01, 02\}$. By setting
    \begin{equation}
        y := (x_1, \dots, x_i, x_{i + 2}, \dots, x_n),
    \end{equation}
    we have~$x = y \circ_i g$. Since~$y$ satisfies the statement, by
    induction hypothesis~$\SComp$ contains~$y$. Hence,~$\SComp$ also
    contains~$x$.
\end{proof}
\medskip

A {\em segmented integer composition} is a sequence~$(S_1, \dots, S_\ell)$
of integers compositions. The size~$|S|$ of a segmented integer composition
is the sum of the sizes of the integer compositions which constitute~$S$.
It is well-known that there are~$3^{n - 1}$ segmented integer compositions
of size~$n$. We shall represent a segmented integer composition~$S$ by a
{\em ribbon diagram}, that is the diagram consisting in the sequence of
the ribbon diagrams of the integer compositions that constitute~$S$.
There is a bijection between the words of~$\SComp$ of arity~$n$ and ribbon
diagrams of segmented compositions of size~$n$.
\medskip

To compute~$\phi_\SComp(x)$ where~$x$ is an element of~$\SComp$,
factorize~$x$ as~$x = 0 x^{(1)} \dots 0 x^{(\ell)}$ such that for
any~$i \in [\ell]$, the factor $x^{(i)}$ has no occurrence of~$0$, and
compute the sequence
$\left(\phi_\Comp\left(0\bar x^{(1)}\right), \dots, \phi_\Comp\left(0\bar x^{(\ell)}\right)\right)$,
where for any~$i \in [\ell]$, $\bar x^{(i)}$ is the word obtained
from~$x^{(i)}$ by decreasing all letters.
\medskip

The inverse bijection is computed as follows.
Given a ribbon diagram~$S := (S_1, \dots, S_\ell)$ of a segmented integer
composition of size $n$, one computes an element of~$\SComp$ of arity~$n$
by computing the sequence $\left(u^{(1)}, \dots, u^{(\ell)}\right)$ where
for any~$i \in [|\ell|]$, $u^{(i)}$ is the word of~$\Comp$ obtained by
applying the inverse bijection of~$\phi_\Comp$ on~$u^{(i)}$, then by
incrementing in each~$u^{(i)}$ all letters, excepted the first one, and
finally by concatenating each words of the sequence.
\medskip

Since the elements
of~$\SComp$ satisfy  Proposition~\ref{prop:ElemSComp},~$\phi_\SComp$ is
well-defined. Figure~\ref{fig:BijSComp} shows an example of this bijection.
\begin{figure}[ht]
    \centering
    \begin{equation*}
        \begin{split}0102012210 \quad \xrightarrow{\phi_\SComp} \quad \end{split}
        \begin{split}\scalebox{.25}{\begin{tikzpicture}
            \node[Boite,EtiqClair]at(0,0){$0$};
            \node[Boite,EtiqClair]at(1,0){$0$};
            \node[Boite,EtiqClair]at(3,0){$0$};
            \node[Boite,EtiqClair]at(3,-1){$1$};
            \node[Boite,EtiqClair]at(5,0){$0$};
            \node[Boite,EtiqClair]at(6,0){$0$};
            \node[Boite,EtiqClair]at(6,-1){$1$};
            \node[Boite,EtiqClair]at(6,-2){$1$};
            \node[Boite,EtiqClair]at(7,-2){$0$};
            \node[Boite,EtiqClair]at(9,0){$0$};
        \end{tikzpicture}}\end{split}
        \quad \longleftrightarrow \quad
        \begin{split}\scalebox{.25}{\begin{tikzpicture}
            \node[Boite]at(0,0){};
            \node[Boite]at(1,0){};
            \node[Boite]at(3,0){};
            \node[Boite]at(3,-1){};
            \node[Boite]at(5,0){};
            \node[Boite]at(6,0){};
            \node[Boite]at(6,-1){};
            \node[Boite]at(6,-2){};
            \node[Boite]at(7,-2){};
            \node[Boite]at(9,0){};
        \end{tikzpicture}}\end{split}
    \end{equation*}
    \caption{Interpretation of an element of the operad~$\SComp$
    in terms of a segmented composition via the bijection~$\phi_\SComp$.
    Boxes of the ribbon diagram in the middle are labeled.}
    \label{fig:BijSComp}
\end{figure}
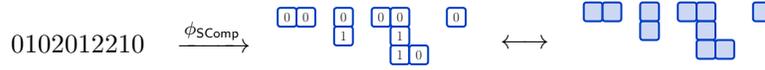
\medskip

\begin{Theoreme} \label{thm:PresentationSComp}
    The ns operad~$\SComp$ admits the presentation
    \begin{equation}
        \SComp =
        \CalF\left(\left\{\SCompOpA, \SCompOpB, \SCompOpC\right\}\right)/_\equiv,
    \end{equation}
    where~$\SCompOpA$, $\SCompOpB$, and $\SCompOpC$ are of arity~$2$,
    and~$\equiv$ is the ns operadic congruence generated by

    \begin{minipage}{.45\textwidth}
    \begin{equation}
        \SCompOpA \circ_1 \SCompOpA \: \leftrightarrow \:
            \SCompOpA \circ_2 \SCompOpA,
    \end{equation}
    \begin{equation}
        \SCompOpB \circ_1 \SCompOpA \: \leftrightarrow \:
            \SCompOpA \circ_2 \SCompOpB,
    \end{equation}
    \begin{equation}
        \SCompOpB \circ_1 \SCompOpB \: \leftrightarrow \:
            \SCompOpB \circ_2 \SCompOpA,
    \end{equation}
    \begin{equation}
        \SCompOpA \circ_1 \SCompOpB \: \leftrightarrow \:
            \SCompOpB \circ_2 \SCompOpC,
    \end{equation}
    \begin{equation}
        \SCompOpB \circ_1 \SCompOpC \: \leftrightarrow \:
            \SCompOpC \circ_2 \SCompOpC,
    \end{equation}
    \end{minipage}
    \begin{minipage}{.45\textwidth}
    \begin{equation}
        \SCompOpA \circ_1 \SCompOpC \: \leftrightarrow \:
            \SCompOpC \circ_2 \SCompOpB,
    \end{equation}
    \begin{equation}
        \SCompOpC \circ_1 \SCompOpA \: \leftrightarrow \:
            \SCompOpA \circ_2 \SCompOpC,
    \end{equation}
    \begin{equation}
        \SCompOpC \circ_1 \SCompOpB \: \leftrightarrow \:
            \SCompOpB \circ_2 \SCompOpB,
    \end{equation}
    \begin{equation}
        \SCompOpC \circ_1 \SCompOpC \: \leftrightarrow \:
            \SCompOpC \circ_2 \SCompOpA.
    \end{equation}
    \end{minipage}
\end{Theoreme}
\begin{proof}
    First, note that by replacing~$\SCompOpA$ by~$00 \in \SComp(2)$,
    $\SCompOpB$ by~$01 \in \SComp(2)$, and~$\SCompOpC$ by~$02 \in \SComp(2)$,
    we have~$\Eval(x) = \Eval(y)$ for the nine relations~$x \leftrightarrow y$
    of the statement of the Theorem. Indeed, the nine equivalence classes
    are, respectively, the ones of the elements~$000$, $001$, $011$, $010$,
    $021$, $020$, $002$, $012$, and~$022$.
    \smallskip

    Consider now the orientation of~$\leftrightarrow$ into the rewrite
    rule~$\mapsto$ defined by

    \begin{minipage}{.45\textwidth}
    \begin{equation}
        \begin{split}\scalebox{.35}{\begin{tikzpicture}
            \node[Feuille](0)at(0.0,-2){};
            \node[Noeud,EtiqClair,minimum size=12mm](1)at(1.0,-1){\scalebox{2}{\SCompOpA}};
            \node[Feuille](2)at(2.0,-2){};
            \draw[Arete](1)--(0);
            \draw[Arete](1)--(2);
            \node[Noeud,EtiqClair,minimum size=12mm](3)at(3.0,0){\scalebox{2}{\SCompOpA}};
            \node[Feuille](4)at(4.0,-1){};
            \draw[Arete](3)--(1);
            \draw[Arete](3)--(4);
        \end{tikzpicture}}\end{split}
        \begin{split}\enspace \to \enspace \end{split}
        \begin{split}\scalebox{.35}{\begin{tikzpicture}
            \node[Feuille](0)at(0.0,-1){};
            \node[Noeud,EtiqClair,minimum size=12mm](1)at(1.0,0){\scalebox{2}{\SCompOpA}};
            \node[Feuille](2)at(2.0,-2){};
            \node[Noeud,EtiqClair,minimum size=12mm](3)at(3.0,-1){\scalebox{2}{\SCompOpA}};
            \node[Feuille](4)at(4.0,-2){};
            \draw[Arete](3)--(2);
            \draw[Arete](3)--(4);
            \draw[Arete](1)--(0);
            \draw[Arete](1)--(3);
        \end{tikzpicture}}\end{split},
    \end{equation}
    \begin{equation}
        \begin{split}\scalebox{.35}{\begin{tikzpicture}
            \node[Feuille](0)at(0.0,-2){};
            \node[Noeud,EtiqClair,minimum size=12mm](1)at(1.0,-1){\scalebox{2}{\SCompOpA}};
            \node[Feuille](2)at(2.0,-2){};
            \draw[Arete](1)--(0);
            \draw[Arete](1)--(2);
            \node[Noeud,EtiqClair,minimum size=12mm](3)at(3.0,0){\scalebox{2}{\SCompOpB}};
            \node[Feuille](4)at(4.0,-1){};
            \draw[Arete](3)--(1);
            \draw[Arete](3)--(4);
        \end{tikzpicture}}\end{split}
        \begin{split}\enspace \to \enspace \end{split}
        \begin{split}\scalebox{.35}{\begin{tikzpicture}
            \node[Feuille](0)at(0.0,-1){};
            \node[Noeud,EtiqClair,minimum size=12mm](1)at(1.0,0){\scalebox{2}{\SCompOpA}};
            \node[Feuille](2)at(2.0,-2){};
            \node[Noeud,EtiqClair,minimum size=12mm](3)at(3.0,-1){\scalebox{2}{\SCompOpB}};
            \node[Feuille](4)at(4.0,-2){};
            \draw[Arete](3)--(2);
            \draw[Arete](3)--(4);
            \draw[Arete](1)--(0);
            \draw[Arete](1)--(3);
        \end{tikzpicture}}\end{split},
    \end{equation}
    \begin{equation}
        \begin{split}\scalebox{.35}{\begin{tikzpicture}
            \node[Feuille](0)at(0.0,-2){};
            \node[Noeud,EtiqClair,minimum size=12mm](1)at(1.0,-1){\scalebox{2}{\SCompOpB}};
            \node[Feuille](2)at(2.0,-2){};
            \draw[Arete](1)--(0);
            \draw[Arete](1)--(2);
            \node[Noeud,EtiqClair,minimum size=12mm](3)at(3.0,0){\scalebox{2}{\SCompOpB}};
            \node[Feuille](4)at(4.0,-1){};
            \draw[Arete](3)--(1);
            \draw[Arete](3)--(4);
        \end{tikzpicture}}\end{split}
        \begin{split}\enspace \to \enspace \end{split}
        \begin{split}\scalebox{.35}{\begin{tikzpicture}
            \node[Feuille](0)at(0.0,-1){};
            \node[Noeud,EtiqClair,minimum size=12mm](1)at(1.0,0){\scalebox{2}{\SCompOpB}};
            \node[Feuille](2)at(2.0,-2){};
            \node[Noeud,EtiqClair,minimum size=12mm](3)at(3.0,-1){\scalebox{2}{\SCompOpA}};
            \node[Feuille](4)at(4.0,-2){};
            \draw[Arete](3)--(2);
            \draw[Arete](3)--(4);
            \draw[Arete](1)--(0);
            \draw[Arete](1)--(3);
        \end{tikzpicture}}\end{split},
    \end{equation}
    \begin{equation}
        \begin{split}\scalebox{.35}{\begin{tikzpicture}
            \node[Feuille](0)at(0.0,-2){};
            \node[Noeud,EtiqClair,minimum size=12mm](1)at(1.0,-1){\scalebox{2}{\SCompOpB}};
            \node[Feuille](2)at(2.0,-2){};
            \draw[Arete](1)--(0);
            \draw[Arete](1)--(2);
            \node[Noeud,EtiqClair,minimum size=12mm](3)at(3.0,0){\scalebox{2}{\SCompOpA}};
            \node[Feuille](4)at(4.0,-1){};
            \draw[Arete](3)--(1);
            \draw[Arete](3)--(4);
        \end{tikzpicture}}\end{split}
        \begin{split}\enspace \to \enspace \end{split}
        \begin{split}\scalebox{.35}{\begin{tikzpicture}
            \node[Feuille](0)at(0.0,-1){};
            \node[Noeud,EtiqClair,minimum size=12mm](1)at(1.0,0){\scalebox{2}{\SCompOpB}};
            \node[Feuille](2)at(2.0,-2){};
            \node[Noeud,EtiqClair,minimum size=12mm](3)at(3.0,-1){\scalebox{2}{\SCompOpC}};
            \node[Feuille](4)at(4.0,-2){};
            \draw[Arete](3)--(2);
            \draw[Arete](3)--(4);
            \draw[Arete](1)--(0);
            \draw[Arete](1)--(3);
        \end{tikzpicture}}\end{split},
    \end{equation}
    \begin{equation}
        \begin{split}\scalebox{.35}{\begin{tikzpicture}
            \node[Feuille](0)at(0.0,-2){};
            \node[Noeud,EtiqClair,minimum size=12mm](1)at(1.0,-1){\scalebox{2}{\SCompOpC}};
            \node[Feuille](2)at(2.0,-2){};
            \draw[Arete](1)--(0);
            \draw[Arete](1)--(2);
            \node[Noeud,EtiqClair,minimum size=12mm](3)at(3.0,0){\scalebox{2}{\SCompOpB}};
            \node[Feuille](4)at(4.0,-1){};
            \draw[Arete](3)--(1);
            \draw[Arete](3)--(4);
        \end{tikzpicture}}\end{split}
        \begin{split}\enspace \to \enspace \end{split}
        \begin{split}\scalebox{.35}{\begin{tikzpicture}
            \node[Feuille](0)at(0.0,-1){};
            \node[Noeud,EtiqClair,minimum size=12mm](1)at(1.0,0){\scalebox{2}{\SCompOpC}};
            \node[Feuille](2)at(2.0,-2){};
            \node[Noeud,EtiqClair,minimum size=12mm](3)at(3.0,-1){\scalebox{2}{\SCompOpC}};
            \node[Feuille](4)at(4.0,-2){};
            \draw[Arete](3)--(2);
            \draw[Arete](3)--(4);
            \draw[Arete](1)--(0);
            \draw[Arete](1)--(3);
        \end{tikzpicture}}\end{split},
    \end{equation}
    \end{minipage}
    \begin{minipage}{.45\textwidth}
    \begin{equation}
        \begin{split}\scalebox{.35}{\begin{tikzpicture}
            \node[Feuille](0)at(0.0,-2){};
            \node[Noeud,EtiqClair,minimum size=12mm](1)at(1.0,-1){\scalebox{2}{\SCompOpC}};
            \node[Feuille](2)at(2.0,-2){};
            \draw[Arete](1)--(0);
            \draw[Arete](1)--(2);
            \node[Noeud,EtiqClair,minimum size=12mm](3)at(3.0,0){\scalebox{2}{\SCompOpA}};
            \node[Feuille](4)at(4.0,-1){};
            \draw[Arete](3)--(1);
            \draw[Arete](3)--(4);
        \end{tikzpicture}}\end{split}
        \begin{split}\enspace \to \enspace \end{split}
        \begin{split}\scalebox{.35}{\begin{tikzpicture}
            \node[Feuille](0)at(0.0,-1){};
            \node[Noeud,EtiqClair,minimum size=12mm](1)at(1.0,0){\scalebox{2}{\SCompOpC}};
            \node[Feuille](2)at(2.0,-2){};
            \node[Noeud,EtiqClair,minimum size=12mm](3)at(3.0,-1){\scalebox{2}{\SCompOpB}};
            \node[Feuille](4)at(4.0,-2){};
            \draw[Arete](3)--(2);
            \draw[Arete](3)--(4);
            \draw[Arete](1)--(0);
            \draw[Arete](1)--(3);
        \end{tikzpicture}}\end{split},
    \end{equation}
    \begin{equation}
        \begin{split}\scalebox{.35}{\begin{tikzpicture}
            \node[Feuille](0)at(0.0,-2){};
            \node[Noeud,EtiqClair,minimum size=12mm](1)at(1.0,-1){\scalebox{2}{\SCompOpA}};
            \node[Feuille](2)at(2.0,-2){};
            \draw[Arete](1)--(0);
            \draw[Arete](1)--(2);
            \node[Noeud,EtiqClair,minimum size=12mm](3)at(3.0,0){\scalebox{2}{\SCompOpC}};
            \node[Feuille](4)at(4.0,-1){};
            \draw[Arete](3)--(1);
            \draw[Arete](3)--(4);
        \end{tikzpicture}}\end{split}
        \begin{split}\enspace \to \enspace \end{split}
        \begin{split}\scalebox{.35}{\begin{tikzpicture}
            \node[Feuille](0)at(0.0,-1){};
            \node[Noeud,EtiqClair,minimum size=12mm](1)at(1.0,0){\scalebox{2}{\SCompOpA}};
            \node[Feuille](2)at(2.0,-2){};
            \node[Noeud,EtiqClair,minimum size=12mm](3)at(3.0,-1){\scalebox{2}{\SCompOpC}};
            \node[Feuille](4)at(4.0,-2){};
            \draw[Arete](3)--(2);
            \draw[Arete](3)--(4);
            \draw[Arete](1)--(0);
            \draw[Arete](1)--(3);
        \end{tikzpicture}}\end{split},
    \end{equation}
    \begin{equation}
        \begin{split}\scalebox{.35}{\begin{tikzpicture}
            \node[Feuille](0)at(0.0,-2){};
            \node[Noeud,EtiqClair,minimum size=12mm](1)at(1.0,-1){\scalebox{2}{\SCompOpB}};
            \node[Feuille](2)at(2.0,-2){};
            \draw[Arete](1)--(0);
            \draw[Arete](1)--(2);
            \node[Noeud,EtiqClair,minimum size=12mm](3)at(3.0,0){\scalebox{2}{\SCompOpC}};
            \node[Feuille](4)at(4.0,-1){};
            \draw[Arete](3)--(1);
            \draw[Arete](3)--(4);
        \end{tikzpicture}}\end{split}
        \begin{split}\enspace \to \enspace \end{split}
        \begin{split}\scalebox{.35}{\begin{tikzpicture}
            \node[Feuille](0)at(0.0,-1){};
            \node[Noeud,EtiqClair,minimum size=12mm](1)at(1.0,0){\scalebox{2}{\SCompOpB}};
            \node[Feuille](2)at(2.0,-2){};
            \node[Noeud,EtiqClair,minimum size=12mm](3)at(3.0,-1){\scalebox{2}{\SCompOpB}};
            \node[Feuille](4)at(4.0,-2){};
            \draw[Arete](3)--(2);
            \draw[Arete](3)--(4);
            \draw[Arete](1)--(0);
            \draw[Arete](1)--(3);
        \end{tikzpicture}}\end{split},
    \end{equation}
    \begin{equation}
        \begin{split}\scalebox{.35}{\begin{tikzpicture}
            \node[Feuille](0)at(0.0,-2){};
            \node[Noeud,EtiqClair,minimum size=12mm](1)at(1.0,-1){\scalebox{2}{\SCompOpC}};
            \node[Feuille](2)at(2.0,-2){};
            \draw[Arete](1)--(0);
            \draw[Arete](1)--(2);
            \node[Noeud,EtiqClair,minimum size=12mm](3)at(3.0,0){\scalebox{2}{\SCompOpC}};
            \node[Feuille](4)at(4.0,-1){};
            \draw[Arete](3)--(1);
            \draw[Arete](3)--(4);
        \end{tikzpicture}}\end{split}
        \begin{split}\enspace \to \enspace \end{split}
        \begin{split}\scalebox{.35}{\begin{tikzpicture}
            \node[Feuille](0)at(0.0,-1){};
            \node[Noeud,EtiqClair,minimum size=12mm](1)at(1.0,0){\scalebox{2}{\SCompOpC}};
            \node[Feuille](2)at(2.0,-2){};
            \node[Noeud,EtiqClair,minimum size=12mm](3)at(3.0,-1){\scalebox{2}{\SCompOpA}};
            \node[Feuille](4)at(4.0,-2){};
            \draw[Arete](3)--(2);
            \draw[Arete](3)--(4);
            \draw[Arete](1)--(0);
            \draw[Arete](1)--(3);
        \end{tikzpicture}}\end{split}.
    \end{equation}
    \end{minipage}

    This rewrite rule is terminating. Indeed, it is plain that for any
    rewriting~$T_0 \to T_1$, we have~$\Poids(T_0) < \Poids(T_1)$.
    \smallskip

    Moreover, the normal forms of~$\mapsto$ are all syntax trees
    of~$\CalF\left(\left\{\SCompOpA, \SCompOpB, \SCompOpC\right\}\right)$
    such that each internal node has no internal node as left son. Hence,
    the generating series~$F(t)$ of the normal forms of~$\mapsto$ is
    \begin{equation}
        F(t) = \sum_{n \geq 1} 3^{n - 1} t^n.
    \end{equation}
    By Proposition~\ref{prop:ElemSComp},~$F(t)$ also is the Hilbert series
    of~$\SComp$.
    \smallskip

    Hence, by Lemma~\ref{lem:PresentationReecriture},~$\SComp$ admits
    the claimed presentation.
\end{proof}
\medskip

\subsection{Operads from the multiplicative monoid}
\label{subsec:MonoideMultiplicatif}
We shall denote by~$\EnsNatM$ the multiplicative monoid of integers.
\medskip

Note that the ns suboperad of~$\T \EnsNatM$ generated by~$00$ and the ns
suboperad of~$\T \EnsNatM$ generated by~$11$ are both isomorphic to the
associative commutative operad~$\Com$.
\medskip

The operads constructed in this section fit into the diagram of ns operads
represented by Figure~\ref{fig:DiagrammeOperadesMonoideMultiplicatif}.
Table~\ref{tab:OperadesMult} summarizes some information about these ns
operads.
\begin{figure}[ht]
    \centering
    \begin{tikzpicture}[scale=.6]
        \node(TN)at(0,0){$\T \EnsNatM$};
        \node(Tr)at(0,-2){$\Tr$};
        \node(D)at(0,-4){$\DD$};
        \node(As)at(0,-6){$\Com$};
        \draw[Injection](Tr)--(TN);
        \draw[Injection](D)--(Tr);
        \draw[Surjection](D)--(As);
    \end{tikzpicture}
    \caption{The diagram of ns suboperads and quotients of~$\T \EnsNatM$.
    Arrows~$\rightarrowtail$ (resp.~$\twoheadrightarrow$) are injective
    (resp. surjective) ns operad morphisms.}
    \label{fig:DiagrammeOperadesMonoideMultiplicatif}
\end{figure}
\begin{table}[ht]
    \centering
    \begin{tabular}{c|c|c|c|c}
        Monoid & Ns operad & Generators & First dimensions
            & Combinatorial objects \\ \hline \hline
        \multirow{2}{*}{$\EnsNatM$} & $\Tr$ & $01$, $10$, $11$ & $1, 3, 7, 15, 31$
            & Binary words with at least one $1$ \\
        & $\DD$ & $01$, $10$ & $1, 2, 3, 4, 5$ & Binary words with exactly one $1$
    \end{tabular} \vspace{.5em}
    \caption{Ground monoids, generators, first dimensions, and combinatorial
    objects involved in the ns suboperads and quotients of~$\T \EnsNatM$.}
    \label{tab:OperadesMult}
\end{table}
\medskip

\subsubsection{The diassociative operad}
Let~$\DD$ be the ns suboperad~of $\T \EnsNatM$ generated
by~$01$ and~$10$. The following table shows the first elements of~$\DD$.
\begin{center}
    \begin{tabular}{c|p{11cm}}
        Arity & Elements of~$\DD$ \\ \hline \hline
        $1$ & $1$ \\ \hline
        $2$ & $01$, $10$ \\ \hline
        $3$ & $001$, $010$, $100$ \\ \hline
        $4$ & $0001$, $0010$, $0100$, $1000$ \\ \hline
        $5$ & $00001$, $00010$, $00100$, $01000$, $10000$ \\ \hline
        $6$ & $000001$, $000010$, $000100$, $001000$, $010000$, $100000$
    \end{tabular}
\end{center}
\medskip

Note that~$\Com$ is a quotient of~$\DD$ by the ns operadic congruence~$\equiv$
defined for all~$x, y \in \DD(n)$ by~$x \equiv y$.
\medskip

One has the following characterization of the elements of~$\DD$:
\begin{Proposition} \label{prop:ElemD}
    The elements of~$\DD$ are exactly the words on the alphabet~$\{0, 1\}$
    containing exactly one~$1$.
\end{Proposition}
\begin{proof}
    Let us first show by induction on the length of the words that any
    word~$x$ of~$\DD$ satisfies the statement. This is true
    when~$|x| = 1$ since~$1$ is the unit of~$\EnsNatM$. When~$|x| \geq 2$, by
    Lemma~\ref{lem:GenerationOpNSEns}, there is an element~$y$ of~$\DD$
    of length~$n := |x| - 1$, an integer~$i \in [n]$, and~$g \in \{01, 10\}$ such
    that~$x = y \circ_i g$. In all cases,~$x$ is obtained from~$y$ by
    inserting a~$0$ at an appropriate position. Since, by induction
    hypothesis,~$y$ satisfies the statement,~$x$ also satisfies the
    statement.
    \medskip

    Let us now show by induction on the length of the words that~$\DD$
    contains any word~$x$ satisfying the statement. This is
    true when~$|x| = 1$. When~$n := |x| \geq 2$, there is in~$x$ a
    factor~$x_i x_{i + 1} =: g$ such that~$g \in \{01, 10\}$. Assume
    without lost of generality that~$g = 01$. Then, by setting
    \begin{equation}
        y := (x_1, \dots, x_{i - 1}, x_{i + 1}, \dots, x_n),
    \end{equation}
    we have~$x = y \circ_i g$. Since~$y$ satisfies the statement,
    by induction hypothesis~$\DD$ contains~$y$. Hence,~$\DD$ also
    contains~$x$.
\end{proof}
\medskip

Recall that the {\em diassociative operad}~\cite{Lod01}~$\Dias$ is the ns
operad admitting the presentation
\begin{equation}
    \Dias := \CalF\left(\{\dashv, \vdash\}\right)/_\equiv,
\end{equation}
where~$\dashv$ and~$\vdash$ are of arity~$2$, and~$\equiv$ is the ns operadic
congruence generated by

\begin{minipage}{.55\textwidth}
\begin{equation}
    \dashv \circ_1 \dashv \enspace \leftrightarrow \enspace
    \dashv \circ_2 \dashv
    \enspace \leftrightarrow \enspace \dashv \circ_2 \vdash,
\end{equation}
\begin{equation}
    \vdash \circ_2 \vdash \enspace \leftrightarrow \enspace
    \vdash \circ_1 \vdash
    \enspace \leftrightarrow \enspace \vdash \circ_1 \dashv,
\end{equation}
\end{minipage}
\begin{minipage}{.4\textwidth}
\begin{equation}
    \dashv \circ_1 \vdash \enspace \leftrightarrow \enspace
    \vdash \circ_2 \dashv.
\end{equation}
\end{minipage}
\medskip

\begin{Proposition} \label{prop:IsoDiasD}
    The ns operads~$\DD$ and~$\Dias$ are isomorphic and the map
    \begin{equation}
        \phi : \Dias \to \DD
    \end{equation}
    satisfying~$\phi(\dashv) = 10$ and~$\phi(\vdash) = 01$ is an
    isomorphism.
\end{Proposition}
\begin{proof}
    By replacing each generator~$g$ of~$\Dias$ by~$\phi(g)$,
    the generators~$01$ and~$10$ of~$\DD$ satisfy the same relations
    as the generators~$\dashv$ and~$\vdash$ of~$\Dias$.
    Moreover, by Proposition~\ref{prop:ElemD}, the Hilbert series
    of~$\DD$ is
    \begin{equation}
        F(t) = \sum_{n \geq 1} n t^n,
    \end{equation}
    which also is the Hilbert series of~$\Dias$. Then, there is no
    nontrivial relation of degree greater than two involving generators
    of~$\DD$.
\end{proof}
\medskip

Proposition~\ref{prop:IsoDiasD} also shows that~$\DD$ is a realization
of the diassociative operad.
\medskip

\subsubsection{The triassociative operad}
Let~$\Tr$ be the ns suboperad~of $\T \EnsNatM$ generated by~$01$, $10$,
and~$11$. The following table shows the first elements of~$\Tr$.
\begin{center}
    \begin{tabular}{c|p{11cm}}
        Arity & Elements of~$\Tr$ \\ \hline \hline
        $1$ & $1$ \\ \hline
        $2$ & $01$, $10$, $11$ \\ \hline
        $3$ & $001$, $010$, $011$, $100$, $101$, $110$, $111$ \\ \hline
        $4$ & $0001$, $0010$, $0011$, $0100$, $0101$, $0110$, $0111$,
              $1000$, $1001$, $1010$, $1011$, $1100$, $1101$, $1110$, $1111$
    \end{tabular}
\end{center}
\medskip

Since~$\DD$ is generated by~$01$ and~$10$,~$\DD$ is a ns suboperad of~$\Tr$.
\medskip

One has the following characterization of the elements of~$\Tr$:
\begin{Proposition} \label{prop:ElemTr}
    The elements of~$\Tr$ are exactly the words on the alphabet~$\{0, 1\}$
    containing at least one~$1$.
\end{Proposition}
\begin{proof}
    Let us first show by induction on the length of the words that any
    word~$x$ of~$\Tr$ satisfies the statement. This is true
    when~$|x| = 1$ since~$1$ is the unit of~$\EnsNatM$. When~$|x| \geq 2$, by
    Lemma~\ref{lem:GenerationOpNSEns}, there is an element~$y$ of~$\Tr$
    of length~$n := |x| - 1$, an integer~$i \in [n]$, and~$g \in \{01, 10, 11\}$
    such that~$x = y \circ_i g$. By induction hypothesis,~$y$ contains
    at least one~$1$. Since all generators of~$\Tr$ contain at least one~$1$,~$x$
    also contains at least one~$1$.
    \smallskip

    Let us now show by induction on the length of the words that~$\Tr$
    contains any word~$x$ satisfying the statement. This is
    true when~$|x| = 1$. When $n := |x| \geq 2$, there is in~$x$
    a factor~$x_i x_{i + 1} =: g$ such that $g \in \{01, 10, 11\}$. Then,
    by setting
    \begin{equation}
        y := (x_1, \dots, x_{i - 1}, x_{i + 1}, \dots, x_n)
    \end{equation}
    if~$g = 01$, or
    \begin{equation}
        y := (x_1, \dots, x_i, x_{i + 2}, \dots, x_n)
    \end{equation}
    if~$g \in \{10, 11\}$, we have $x = y \circ_i g$. Since~$y$ satisfies
    the statement, by induction hypothesis~$\Tr$ contains~$y$. Hence,~$\Tr$
    also contains~$x$.
\end{proof}
\medskip

Recall that the {\em triassociative operad}~\cite{LR04}~$\Trias$ is the ns
operad admitting the presentation
\begin{equation}
    \Trias := \CalF(\{\dashv, \perp, \vdash\})/_\equiv,
\end{equation}
where~$\dashv$, $\perp$, and~$\vdash$ are of arity~$2$, and~$\equiv$ is
the ns operadic congruence generated by

\begin{minipage}{.4\textwidth}
\begin{equation}
    \dashv \circ_1 \vdash \enspace \leftrightarrow \enspace
    \vdash \circ_2 \dashv,
\end{equation}
\begin{equation}
    \perp \circ_1 \perp \enspace \leftrightarrow \enspace
    \perp \circ_2 \perp,
\end{equation}
\begin{equation}
    \dashv \circ_1 \perp \enspace \leftrightarrow \enspace
    \perp \circ_2 \dashv,
\end{equation}
\begin{equation}
    \perp \circ_1 \dashv \enspace \leftrightarrow \enspace
    \perp \circ_2 \vdash,
\end{equation}
\end{minipage}
\begin{minipage}{.55\textwidth}
\begin{equation}
    \perp \circ_1 \vdash \enspace \leftrightarrow \enspace
    \vdash \circ_2 \perp,
\end{equation}
\begin{equation}
    \dashv \circ_1 \dashv \enspace \leftrightarrow \enspace
        \dashv \circ_2 \dashv \enspace \leftrightarrow \enspace
        \dashv \circ_2 \vdash \enspace \leftrightarrow \enspace
        \dashv \circ_2 \perp,
\end{equation}
\begin{equation}
    \vdash \circ_2 \vdash \enspace \leftrightarrow \enspace
        \vdash \circ_1 \vdash \enspace \leftrightarrow \enspace
        \vdash \circ_1 \dashv \enspace \leftrightarrow \enspace
        \vdash \circ_1 \perp.
\end{equation}
\end{minipage}
\medskip

\begin{Proposition} \label{prop:IsoTriasTr}
    The ns operads~$\Tr$ and~$\Trias$ are isomorphic and the map
    \begin{equation}
        \phi : \Trias \to \Tr
    \end{equation}
    satisfying~$\phi(\dashv) = 10$, $\phi(\vdash) = 01$,
    and~$\phi(\perp) = 11$ is an isomorphism.
\end{Proposition}
\begin{proof}
    By replacing each generator~$g$ of~$\Trias$ by~$\phi(g)$,
    the generators~$01$,~$10$, and~$11$ of~$\Tr$ satisfy the same relations
    as the generators~$\dashv$,~$\vdash$, and~$\perp$ of~$\Trias$.
    Moreover, by Proposition~\ref{prop:ElemTr}, the Hilbert series
    of~$\Tr$ is
    \begin{equation}
        F(t) = \sum_{n \geq 1} (2^n - 1) t^n,
    \end{equation}
    which also is the Hilbert series of~$\Trias$. Then, there is no
    nontrivial relation of degree greater than two involving generators
    of~$\Tr$.
\end{proof}
\medskip

Proposition~\ref{prop:IsoTriasTr} also shows that~$\Tr$ is a realization of
the triassociative operad.
\medskip

\section*{Concluding remarks}
We have presented the functorial construction~$\T$ producing an operad
given a monoid. As we have seen, this construction is very rich from
a combinatorial point of view since most of the obtained operads coming
from usual monoids involve a wide range of combinatorial objects. There
are various way to continue this work. Let us address here the main directions.
\smallskip

In the first place, it appears that we have somewhat neglected the fact that
$\T$ is a functor to operads and not only to ns ones. Indeed, except for the
operads~$\End$, $\FP$, $\MT$, and~$\Per$, we only have regarded the obtained
operads as ns ones. Computer experiments let us think that the dimensions of the
operads~$\APE$, $\FCat{2}$, $\Motz$, $\AnD$ and $\SComp$ seen as symmetric ones
are, respectively, Sequences~\Sloane{A052882}, \Sloane{A050351}, \Sloane{A032181},
\Sloane{A101052}, and~\Sloane{A001047} of~\cite{Slo}. Bijections between elements
of these operads and combinatorial objects enumerated by these sequences, together
with presentations by generators and relations in this symmetric context, would
be worthwhile.
\smallskip

Furthermore, we have considered~$\T$ only in the category of sets, {\em i.e.},
it takes a monoid as input and constructs a set-operad as output. We can
obviously extend the definition of~$\T$ over the category of vector spaces.
In that event, $\T$ would be a functor from the category of unital associative
algebras to the category of operads in the category of vector spaces. It
is thus natural to ask what operads $\T$ produces in this category.
\smallskip

Another line of research is the following. It is well-known that the
Koszul dual (see~\cite{GK94} for Koszul duality of operads) of the
operads~$\Dias$ and~$\Trias$ are respectively the dendriform~$\Dendr$~\cite{Lod01}
and the tridendriform~$\TDendr$~\cite{LR04} operads. The tridendriform
operad is a generalization of the dendriform operad and further generalizations
were proposed, like the operads~$\Quad$~\cite{AL04} and~$\Ennea$~\cite{Ler04}.
Since the operads~$\DD$ and~$\Tr$, obtained from the~$\T$ construction,
are respectively isomorphic to the operads~$\Dias$ and~$\Trias$, we can
ask if there are generalizations of~$\DD$ and~$\Tr$ so that their Koszul
duals provide generalizations of the operads~$\Dendr$ and~$\TDendr$.
\medskip

\bibliographystyle{alpha}
\bibliography{Bibliographie}

\end{document}